\documentclass[12pt,titlepage]{article}
\usepackage{amssymb,amsmath,amsthm,bm,setspace,mathrsfs}
\usepackage{url,color,units,dsfont,wasysym,paralist,bbm,mathtools}
\RequirePackage[colorlinks,citecolor=blue!70!red!70,linkcolor=red!70!black!70]{hyperref}
\usepackage[T1]{fontenc}
\usepackage[margin=1in]{geometry}
\usepackage[font={footnotesize}]{caption}
\usepackage{graphicx,float}
\usepackage{subfigure,graphics}
\usepackage{xcolor,tikz}
\usepackage{tikz}\usetikzlibrary{shapes,arrows,patterns}
\include{rgb.tex}

\DeclareMathOperator{\Var}{Var}
\DeclareMathOperator{\Cov}{Cov}
\DeclareMathOperator{\Corr}{Corr}
\DeclareMathOperator{\Dim}{Dim_{_{\rm M}}}
\DeclareMathOperator{\oDim}{\overline{Dim}_{_{\rm M}}}
\DeclareMathOperator{\uDim}{\underline{Dim}_{_{\rm M}}}
\DeclareMathOperator{\cL}{\mathscr{L}}

\newcommand{\Z}{\mathbbm{Z}}
\newcommand{\Q}{\mathbbm{Q}}
\newcommand{\R}{\mathbbm{R}}
\newcommand{\C}{\mathbbm{C}}
\newcommand{\<}{\langle}
\renewcommand{\>}{\rangle}

\newcommand{\be}{\begin{equation}}
\newcommand{\ee}{\end{equation}}

\renewcommand{\P}{\mathrm{P}}
\newcommand{\E}{\mathrm{E}}
\newcommand{\F}{\mathscr{F}}

\newcommand{\1}{\mathbbm{1}}
\renewcommand{\d}{{\rm d}}

\newcommand{\e}{{\rm e}}
\renewcommand{\geq}{\geqslant}
\renewcommand{\leq}{\leqslant}
\renewcommand{\ge}{\geqslant}
\renewcommand{\le}{\leqslant}

\newcommand{\wh}{\widehat}
\newcommand{\wt}{\widetilde}

\author{ Jingyu Huang\\University of Utah
\and Davar Khoshnevisan\\University of Utah}

\title{\bf Analysis of a Stratified Kraichnan Flow\thanks{%
	Research supported in part by the NSF grants DMS-1307470 
	and DMS-1608575.}}
\date{Version: Nov 28, 2017}

\newtheorem{stat}{Statement}[section]
\newtheorem{proposition}[stat]{Proposition}
\newtheorem{corollary}[stat]{Corollary}
\newtheorem{theorem}[stat]{Theorem}
\newtheorem{lemma}[stat]{Lemma}
\theoremstyle{definition} 
\newtheorem{definition}[stat]{Definition}
\newtheorem{remark}[stat]{Remark}

\newtheorem{assumption}{Assumption}

\numberwithin{equation}{section}
\begin{document}\spacing{1.1}
\maketitle
\begin{abstract}
	We consider the stochastic convection--diffusion equation
	\[
		\partial_t \theta(t\,,\bm{x}) =\nu\Delta \theta(t\,,\bm{x}) + 
		V(t\,,x_1)\partial_{x_2} \theta(t\,,\bm{x}),
	\]
	for $t>0$ and $\bm{x}=(x_1\,,x_2)\in\R^2$, 
	subject to $\theta_0$ being a nice initial profile. Here,
	the velocity field $V$ is assumed to be centered Gaussian
	with covariance structure
	\[
		\Cov[V(t\,,a)\,,V(s\,,b)]= \delta_0(t-s)\rho(a-b)\qquad\text{for all
		$s,t\ge0$ and $a,b\in\R$},
	\]
	where $\rho$ is a continuous and bounded positive-definite function on $\R$.
	
	We prove a quite general existence/uniqueness/regularity theorem,
	together with a probabilistic representation of the solution that represents
	$\theta$ as an expectation functional of an exogenous infinite-dimensional
	Brownian motion. We use
	that probabilistic representation in order to study
	the It\^o/Walsh solution, when it exists, and relate it to the Stratonovich 
	solution which is shown to exist for all $\nu>0$. 
	
	Our \emph{a priori}
	estimates imply the physically-natural fact that, quite generally,
	the solution dissipates. In fact, very often,
	\begin{equation}\label{DR}
		\P\left\{ \sup_{|x_1|\le m}\sup_{x_2\in\R} |\theta(t\,,\bm{x})|
		= O\left( \frac{1}{\sqrt t}\right)\qquad\text{as $t\to\infty$}
		\right\}=1\qquad\text{for all $m>0$},
	\end{equation}
	and the $O(1/\sqrt t)$ rate is shown to be unimproveable.
	
	Our probabilistic representation is
	malleable enough to allow us to analyze the solution in two physically-relevant
	regimes: As $t\to\infty$ and as $\nu\to 0$.  Among other things,
	our analysis leads to a ``macroscopic multifractal analysis''
	of the rate of decay in \eqref{DR} in terms of the reciprocal of
	the Prandtl (or Schmidt) number, valid in a number of simple though
	still physically-relevant cases.\\
	

\end{abstract}
\tableofcontents
\newpage

\section{Introduction and general description of results}

Let $V:=\{V(t\,,x)\}_{t\ge0,x\in\R}$ denote a centered, generalized
Gaussian random field that is white in its ``time variable'' $t$
and spatially-homogeneous in its ``space variable'' $x$, with spatial
correlation function $\rho$. Somewhat more precisely, we suppose
that the covariance structure of $V$ is described as follows:
\begin{equation}\label{cov:v}
	\Cov\left[ V(t\,,x)\,, V(s\,,y) \right] = \delta_0(t-s) \rho(x-y)
	\qquad\text{for all $s,t\ge0$ and $x,y\in\R$},
\end{equation}
where
\begin{equation}\label{rho():cont}
	\text{$\rho:\R\to\R_+$ is assumed to be continuous.}
\end{equation}
We rule out degeneracies by assuming further that
\begin{equation}\label{rho(0)>0}
	\rho(0)>0.
\end{equation}

Choose and fix a constant $\nu>0$.
Our goal is to study the behavior of the solution, if and when one indeed exists, to
the following Stochastic Partial Differential Equation (SPDE):
\begin{equation}\label{pre:Kraichnan}
	\partial_t \theta(t\,,x\,,y) =
	\nu \Delta\theta(t\,,x\,,y) + V(t\,,x) 
	\partial_y \theta(t\,,x\,,y)
	\quad\text{for $t\ge0$ and $x,y\in\R$},
\end{equation}
subject to $\theta(0):=\theta(0\,,\cdot\,,\cdot) =\theta_0$
for a nicely-behaved initial profile $\theta_0$ that might be
random or non random, but independent of $V$  in any case. 

The SPDE \eqref{pre:Kraichnan} is an example
of the Kraichnan model, and describes the turbulent transport of
a passive scalar quantity immersed in an incompressible  two-dimensional fluid;
see Kraichnan \cite{Kraichnan1968,Kraichnan1987} and \S\ref{sec:Fluids} below. The 
stratified velocity
field $V$ has a form that was introduced by 
Majda \cite{Majda,Majdaa} in a slightly different setting.

If $V$ were instead a reasonably nice function, then
\eqref{pre:Kraichnan} can be, and has been, analyzed by both probabilistic
and analytic methods. See, for example, Cranston and Zhao \cite{CranstonZhao},
Osada \cite{Osada}, and Zhang \cite{Zhang}, and
their combined bibliography. In the present, rough/random,
setting, the situation is a little different.
In this case, there are two  standard ways to solve the SPDE 
\eqref{pre:Kraichnan}. One of the approaches works as follows:
One first interprets \eqref{pre:Kraichnan} pointwise as
an infinite-dimensional Stochastic Differential Equation (SDE),
\begin{equation}\label{pt:Kraichnan}
	\d\theta(t)=\nu\Delta\theta(t)\,\d t+  \partial_y\theta(t) \circ\d W(t),
\end{equation}
where ``$\circ$'' denotes the Stratonovich product,
and $t\mapsto W(t):=\int_0^t V(s)\,\d s$ denotes an infinite-dimensional Brownian motion with covariance
form
\[
	\Cov[W(t\,,x)\,,W(s\,,y) ] = \min(s\,,t)\rho(x-y)
	\qquad\text{for all $s,t\ge0$ and $x,y\in\R$}.
\]
Then, one solves \eqref{pt:Kraichnan} by appealing to the theory of
stochastic flows (see Le Jan and Raimond \cite{LeJanRaimond}). The intricate details of
this solution theory can be
found in Chapter 6 of the book by Kunita \cite{Kunita}. 

The pointwise nature of the SDE \eqref{pt:Kraichnan} suggests that
in order for \eqref{pt:Kraichnan} to have a unqiue strong solution,
the correlation function $\rho$ has to be reasonably smooth.
As far as we know, the strongest theorem of this type currently
requires that $\rho\in C^{6+\varepsilon}$ for some $\varepsilon>0$; 
see Kunita \cite{Kunita} and especially Remark \ref{rem:pt:Kraichnan} below.

The second approach to SPDEs of type \eqref{pre:Kraichnan} is to view it as an
It\^o--Walsh type SPDE, and use ideas from
Sobolev-space theory; see \cite[Section 3.7]{Chow} for example. 
This approach requires a general ``coercivity condition'' that turns 
out to have some connections with the relation \eqref{cond:phys} below. 
In this context, Krylov  \cite{Krylov}
has recently developed a powerful $L^p$ theory, where the analytic approach is carried out to analyze the regularity theory of more general SPDEs of the form 
\begin{equation*}
	\d u= \left( \sum_{i,j=1}^d a^{ij}(t\,,x)\frac{\partial^2 u}{\partial x^i \partial x^j} 
	+ f(t\,, x\,, u\, , Du)\right)\d t +  \sum_{k= 1}^\infty 
	\left(\sum_{i=1}^d \sigma^{ik}(t\,,x)\frac{\partial u }{\partial x^i} + 
	g^k(t\,,x\,,u) \right)\d w_t^k\,,
\end{equation*}
where $\{w^k\}_{k=1}^\infty$ are i.i.d.\ standard Brownian motions,
and the basic idea is that the above equation defines a homeomorphism 
between the solution space (called stochastic Banach spaces) and the space 
of initial data. The study of the particular equation is thus reduced to the study 
of the functions in the solution space, which is still quite involved.

The starting point of the present article is to take a different, third, approach to 
the Kraichnan SPDE \eqref{pre:Kraichnan}, and try and produce
a unique solution to \eqref{pre:Kraichnan}, with the following 
nearly-minimal requirements in mind:
\begin{compactenum}[(1)]
\item $\rho$ is  assumed  only to satisfy \eqref{rho():cont} and \eqref{rho(0)>0}; and more
	significantly, 
\item \label{2} The product of $V(t\,,x)$ and $\partial_y\theta(t\,,x\,,y)$
	in \eqref{pre:Kraichnan} is interpretted as an It\^o/Walsh
	product, as opposed to the Stratonovich product.
\end{compactenum}
The utility of \eqref{2} will become apparent soon, after we
describe applications of our theory to the detailed analysis of
the  solution of \eqref{pre:Kraichnan}.

As it will turn out, one can prove that
\eqref{pre:Kraichnan} has a unique strong It\^o/Walsh type
solution when, and only when,
\begin{equation}\label{cond:phys}
	\nu > \tfrac12\rho(0).
\end{equation}
This is unfortunate because, in terms of the underlying
fluid problem, condition \eqref{cond:phys} implies
that the fluid is allowed to experience only low levels of turbulence. After all, 
$\nu$ is inversely proportional to the Reynolds number of
the fluid, and $\frac12\rho(0)$ denotes  turbulent diffusivity.
One can state this limitation of \eqref{cond:phys} in another
essentially-equivalent manner: If \eqref{cond:phys} holds then 
we cannot study the Kraichnan model in the fully-turbulent regime
$\nu\approx0$, in spite of the fact that the fully-turbulent
regime is the subject of a vast literature
on this subject. For some of the more modern treatments
see  Celani and Vincenzi \cite{CelaniVincenzi},
Grossmann and Lohse \cite{GrossmannLohse},
Holzer and Siggia \cite{HolzerSiggia},
and particularly Warhaft \cite{Warhaft}, as well as their combined,
extensive bibliography.

Our aim to reconcile these seemingly-contradictory assertions
naturally leads us to study the following
slightly more general It\^o/Walsh type SPDE,
\begin{equation}\label{Kraichnan}
	\partial_t \theta(t\,,x\,,y) = 
	\nu_1 \partial^2_x \theta(t\,,x\,,y) + 
	\nu_2 \partial^2_y \theta(t\,,x\,,y) +
	\partial_y \theta(t\,,x\,,y)\, V(t\,,x),
\end{equation}
for $t\ge0$ and $x,y\in\R$. Here, $\nu_1$ and $\nu_2$ are positive
parameters, and the initial profile is still
a nice possibly-random function $\theta_0$ that is independent of $V$.
Thus,  the Kraichnan model \eqref{pre:Kraichnan} is 
the same as the SPDE \eqref{Kraichnan} in the case that $\nu_1=\nu_2$.
And the mentioned analysis of \eqref{pre:Kraichnan} generalizes
immediately to show that \eqref{Kraichnan} has a unique It\^o/Walsh 
solution provided that \eqref{cond:phys} is replaced by 
\begin{equation}\label{cond:phys:2}
	\nu_2 > \tfrac12\rho(0);
\end{equation}
there are no restrictions on $\nu_1$ other than strict positivity.

We will use 
ideas from the Malliavin calculus in order to represent the solution
to \eqref{Kraichnan}, probabilistically, in terms of an exogenous
Wiener measure; see Theorems \ref{th:FK} and \ref{th:Kraichnan:Strat:measure}
below. That probabilistic representation has a number of consequences,
many of which are the central, most novel, findings of this paper. 

As a first application of our probabilistic representation we
construct a Stratonovich-type solution to \eqref{Kraichnan}, and in particular
to \eqref{pre:Kraichnan} using only conditions \eqref{rho():cont} and \eqref{rho(0)>0}. 
In order
to describe this work in more detail let $\{\phi_\varepsilon\}_{\varepsilon>0}$
denote a suitably-regular approximation to the identity on
$\R_+\times\R$ and define
$V_\varepsilon = \phi_\varepsilon *V$
for all $\varepsilon>0$,
where the space-time integral in the latter
convolution is understood as a Wiener integral.
It is not difficult to see that the two-parameter
Gaussian random field $V_\varepsilon$
is almost surely $C^\infty$ for every fixed $\varepsilon>0$. Therefore, the following
regularized version of \eqref{Kraichnan} is  a standard linear PDE, albeit
with a random velocity term $V_\varepsilon$:
\[\left[\begin{split}
	& \partial_t \theta_\varepsilon(t\,,x\,,y) = 
		\nu_1 \partial^2_x \theta_\varepsilon(t\,,x\,,y) + 
		\nu_2 \partial^2_y \theta_\varepsilon(t\,,x\,,y) + 
		\partial_y \theta_\varepsilon(t\,,x\,,y)\, V_\varepsilon(t\,,x),\\
	&\text{subject to }\theta_\varepsilon(0)=\theta_0.
\end{split}\right.\]
It is an elementary fact that the preceding PDE a.s.\ has a unique
$C^\infty$ solution $\theta_\varepsilon$ for every $\varepsilon>0$.
We will use our probabilistic representation to prove
that, as $\varepsilon\downarrow0$, the random field
$\theta_\varepsilon$ converges in a strong sense to the solution
of \eqref{Kraichnan}, but with $\nu_2$ 
replaced by $\nu_2':=\nu_2+\frac12\rho(0)$;
see Theorem \ref{th:Strat} for a precise statement.
This yields a particular infinite-dimensional version 
of the Wong--Zakai theorem (\cite{WongZakai}; see also
McShane \cite{McShane} and
Ikeda, Nakao, and Yamato \cite{INY}) of classical
It\^o calculus. In light of the work of Wong and Zakai, it makes sense to 
refer to the preceding solution to \eqref{Kraichnan}
as its  ``Stratonovich solution,'' which we will do henceforth.\footnote{It might be
	possible to show that our ``Stratonovich solution'' is 
	in fact associated to a Stratonovich-type
	integration theory. We have refrained from doing that here, as it seems
	to be of secondary relevance.}
In any case, because $\nu_2':=\nu_2+\frac12\rho(0)>\frac12\rho(0)$  tautologically satisfies
\eqref{cond:phys:2} for every $\nu_2>0$, it follows that \eqref{Kraichnan}
has a Stratonovich solution---in the sense that we just described---for every
possible $\nu_1,\nu_2>0$. Moreover, the Stratonovich solution 
to \eqref{Kraichnan} with parameters $\nu_1,\nu_2>0$ coincides
with the It\^o/Walsh solution to \eqref{Kraichnan} with parameters
$\nu_1$ and $\nu_2':=\nu_2+\frac12\rho(0)$. In particular, our probabilistic
representation of the solution to the It\^o/Walsh formulation of
\eqref{Kraichnan} immediately yields also a probabilistic representation of
the Stratonovich solution. Set $\nu_1=\nu_2$ to see that the Stratonovich solution
to the Kraichnan model \eqref{pre:Kraichnan} with parameter
$\nu>0$ is, in particular, the It\^o/Walsh solution to
\eqref{pre:Kraichnan} with parameters $\nu_1=\nu$ 
and $\nu_2=\nu+\frac12\rho(0)$. And that the solution exists provided
only that $\rho$ is contiunous and non degenerate [see \eqref{rho():cont}
and \eqref{rho(0)>0}].
This is a significant improvement over the current state of existence and
uniqueness of the Stratonovich solution to \eqref{pre:Kraichnan}.
Let us emphasize further that the said solution also has a 
probabilistic representation in terms of
an exogenous Wiener measure.
Thus, we may yet again apply that probabilistic representation to study the
Stratonovich solution to \eqref{Kraichnan} in greater detail.

One of the immediate corollaries of our probabilistic representation is
that the Stratonovich
solution to \eqref{Kraichnan} converges as $\nu\to 0$ to a nice
random field that is formally the method-of-characteristics solution
to the inviscid form of \eqref{Kraichnan}; 
see Corollary \ref{cor:inviscid}. More precisely,
\begin{equation}\label{limlim}
	\lim_{\nu\downarrow0}\theta(t\,,x\,,y) = 
	\theta_0\left( x\,,
	y  -\int_0^t V(s\,,x)\,\d s\right) ,
\end{equation}
where the convergence holds in $\cap_{k=2}^\infty L^k(\Omega)$
and the Gaussian random field
$\int_0^t V(s\,,x)\,\d s$ will be defined rigorously in \S\ref{sec:pf:th:cont:U} below.
The preceding result is not consistent with
some of the physical predictions
of this field (see, for example, Warhaft \cite[\S5]{Warhaft}). Closely-related results can
be found in the applied mathematics literature as well; see for
example, Bernard,
Gaw\c{e}dzki, and Kupiainnen
\cite{BGK1996,BGK1998}, Eyink and Xin \cite[Ref.\ 21]{EyinkXin2000}, and Vanden Eijnden \cite{VandenEijnden}. 
In order to have a solution with
properties that are
consistent with the various existing physical predictions, one needs
to initialize \eqref{Kraichnan} not with a nice function $\theta_0$---%
as we have done above---%
but rather with a singular measure $\theta_0$;
see also Bernard et al \cite{BGK1996,BGK1998}. A particularly natural choice
is the point mass $\theta_0=\delta_0\otimes\delta_0$
on $(0\,,0)\in\R^2$. Our theory extends fairly readily to cover
such singular initial profiles. In those cases, we obtain results that are consistent
with---and perhaps also better explain---%
some of the existing predictions of the literature.
An example of such a result is that, 
when $\theta_0=\delta_0\otimes\delta_0$, the Stratonovich solution
to \eqref{Kraichnan} satisfies the following: 
$\lim_{\nu\to 0^+}\E[\theta(t\,,0\,,y)]=\infty$ for all $t>0$ and $y\in\R$;
whereas $\lim_{\nu\to 0^+}\E[\theta(t\,,x\,,y)]=0$ for all $t>0$, $y\in\R$,
and $x\in\R\setminus\{0\}$. We remind the reader that when $\theta_0$
is a nice function, the behavior of $\E[\theta(t\,,x\,,y)]$ is radically difference
as $\nu\downarrow0$; see \eqref{limlim}.
For more details on this topic see Theorem 
\ref{th:nu:to:0:Gamma}.\footnote{As is explained also in due
	time, the function
	$\Gamma^{(\nu)}$ of Theorem \ref{th:nu:to:0:Gamma} is reserved
	to represent the Stratonovich solution to \eqref{Kraichnan}
	in the case that $\theta_0=\delta_0\otimes\delta_0$.}

One of the {\it a priori} consequences of the approach of the present paper is the
physically-natural fact that the Stratonovich solution
[or It\^o/Walsh solution, for that matter]  to the Kraichnan model
\eqref{pre:Kraichnan} typically dissipates with time;
that is, $\theta(t)\to 0$ as $t\to\infty$ [in a strong sense, in fact]. 
Moreover, the optimal dissipation rate is  shown to be of sharp order
$1/\sqrt t$ as $t\to\infty$ when $\theta_0$ is a nice function; 
see Theorem \ref{th:exist:U}, Proposition \ref{pr:diss}, and Remark \ref{rem:diss}.

By contrast, the exact rate of dissipation of $\theta(t)$
is  shown to be of sharp order $1/t$ when the initial data is 
$\theta_0=\delta_0\otimes\delta_0$;.
See Theorems \ref{th:Gamma:Dim:local:1} and \ref{th:Gamma:Dim:local:2};
see Eq.\ \eqref{eq:sup:Gamma} for a related observation. 
For instance, based on the preceding claim,
one expects $t\theta(t\,,0\,,0)$
to be a.s.\ of sharp order one as $t\to\infty$. This turns out to be not
quite true on the level of the sample-function trajectories; in fact, it turns out that
$t\theta(t\,,0\,,0)$ dissipates in a ``multifractal'' fashion as $t\to\infty$.
Slightly more precisely put,
we will use our probabilistic representation
of the solution to show that, when $\theta_0=\delta_0\otimes\delta_0$
and $\rho$ is a constant, the set of times where 
$t\theta(t\,,0\,,0)$ goes to zero faster than $(\log t)^{-\delta}$ a.s.\
has ``macroscopic fractal dimension''
\begin{equation}\label{D}
	\mathscr{D}(\delta)
	:=\max\left( 0\,, 1-\frac{2\delta\nu}{\rho(0)}\right)
	\qquad\text{for all $\delta>0$}.
\end{equation}
Though the details are likely to change, we expect the preceding
macroscopic multifractal
formalism to continue to continue to hold in the more
physically-interesting case that $\rho$ is non constant.

One can restate \eqref{D} as follows:
When $\theta_0=\delta_0\otimes\delta_0$,
the Stratonovich solution to \eqref{Kraichnan} decays as
$t(\log t)^{-\delta}$ on non trivial, macroscopically fractal, time sets
of fractal dimension $\mathscr{D}(\delta)\in(0\,,1)$ for every value of $\delta>0$
less than $\rho(0)/(2\nu)$. Note that this discussion applies to the Stratonovich
solution and, as such, \eqref{D} and the ensuing remarks apply to all
values of $\nu>0$.

Figure \ref{fig:sim1}
shows a large-time simulation
of $t\theta(t\,,0\,,0)$---for the Stratonovich solution 
to \eqref{pre:Kraichnan}---with $\nu=10^{-7}$,
up to time $t=10^5$. 
\begin{figure}[h!]\centering
	\includegraphics[width=.7\linewidth]{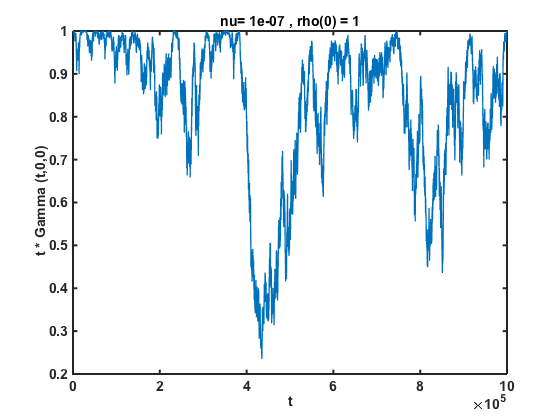}
	\caption{A simulation of the intermittent 
		behavior of $t\mapsto t\theta(t\,,0\,,0)$. The ``Gamma'' on the
		axes refers to our later notation for the Stratonovich solution $\theta$
		in the special case that $\theta_0=\delta_0\otimes\delta_0$.
		See \eqref{def:Gamma}.}
	\label{fig:sim1}
\end{figure}
And Figure \ref{fig:sim2} shows a large-time
simulation of two trajectories of $t\theta(t\,,0\,,0)$---for the Stratonovich solution to
\eqref{pre:Kraichnan}, using the same noise---where the parameter $\nu$ of the 
more fluctuating graph (red) is $14\%$ of the $\nu$ of the other (red).
Perhaps one can recognize the large-time multifractal, intermittent
structure of $t\mapsto \theta(t\,,0\,,0)$ in these simulations?
\begin{figure}[h!]\centering
	\includegraphics[width=.8\linewidth]{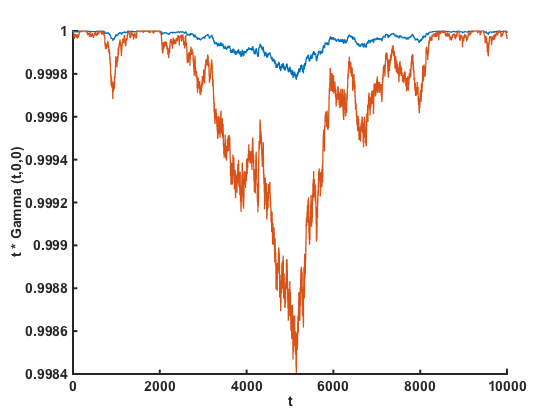}
	\caption{A simulation of two versions of $t\mapsto t\theta(t\,,0\,,0)$ --
		for the same noise -- where the parameter $\nu$ for
		one (red) is $14\%$ of the $\nu$ for the other (blue). The ``Gamma'' on the
		axes refers to our later notation for the Stratonovich solution $\theta$
		in the special case that $\theta_0=\delta_0\otimes\delta_0$.
		See \eqref{def:Gamma}.}
	\label{fig:sim2}
\end{figure}

Within the confines of the present,
restricted model, these results rigorously
justify---and give mathematical language to---some 
of the fluid intermittency assertions of the turbulence literature.
See Mandelbrot \cite[Section 10]{Mandelbrot} for a detailed discussion of this broad topic.

Throughout this paper, we consistently adopt the following notational convention, often
without  making explicit mention.\\

\noindent\textbf{Conventions.} 
If $Z\in L^k(\Omega)$ is a complex-valued random variables, then
$\|Z\|_k:=\{\E(|Z|^k)\}^{1/k}$.\\
Whenever $F$ is a real-valued function on $\R_+\times\R^2$,
we write $t\mapsto F(t)$ for the function that is defined by
\[
	F(t) (a\,,b) := F(t\,,a\,,b)\qquad\text{for all $t\ge0$ and $a,b\in\R$}.
\]
Furthermore, we write $F[b]$ for the function that is defined by
\[
	F[b] (t\,,a) := F(t\,,a\,,b)\qquad
	\text{for all $t\ge0$ and $a,b\in\R$},
\]
for every three-variable function $F$ on $\R_+\times\R^2$.\\

\noindent\textbf{Acknowledgements.}
	We learned this exciting topic many years ago from Professor
	Richard McLaughlin, to whom many unreserved thanks are due.
	We also thank Professor Jared Bronski, Gregory Forest, and
	Scott McKinley for a number of related discussions on stochastic
	fluid models. Last but not the least, we thank the National Science
	Foundation for their generous support of this research
	(grants DMS-1307470 and DMS-1608575), and in fact of a decade of 
	work that ultimately led to this.

\section{The It\^o/Walsh solution}\label{sec:The itowalsh solution}

In this section we study the generalized Kraichnan model
\eqref{Kraichnan}, a special case of which [see \eqref{pre:Kraichnan}]
is of particular interest. Namely, we consider the SPDE,
\begin{equation}\label{H:Kraichnan}
	\cL\theta = \partial_y \theta\cdot V
	\quad\text{on $(0\,,\infty)\times\R^2$, where}\quad
	\cL := \partial_t - \nu_1\partial^2_x
	- \nu_2 \partial^2_y,
\end{equation}
subject to $\theta(0)=\theta_0$. The product of $V$ and $\partial_y\theta$ is
interpretted in the It\^o sense.

\subsection{A presentation of the main results}
Let $(t\,,x)\mapsto p^{(\nu)}_t(x)$ denote the fundamental solution to the heat operator
$\partial_t - \nu\partial^2_x$; that is,
\begin{equation}\label{p}
	p^{(\nu)}_t(x) := p^{(\nu)}(t\,,x) := (4\pi\nu t)^{-1/2}
	\exp\left( -\frac{x^2}{4\nu t}\right)
	\qquad\text{for all $t>0$ and $x\in\R$}.
\end{equation}
Because
$(t\,,x\,,y) \mapsto p^{(\nu_1)}_t(x) \cdot p^{(\nu_2)}_t(y)$
defines the fundamental solution to the operator $\cL$,
we can define the notion of a mild solution to \eqref{pre:Kraichnan} as
in Walsh \cite{Walsh}. Namely, we have the following.

\begin{definition}\label{def:mild}
	We say that $(t\,,x\,,y)\mapsto\theta(t\,,x\,,y)$
	is a \emph{mild} solution to \eqref{H:Kraichnan} when
	$\theta$ is a predictable random field (see  \cite{Walsh}) that satisfies
	the following:
	\begin{compactenum}[(1)]
		\item For every $t>0$ and $x\in\R$, the random function
			$y\mapsto\theta(t\,,x\,,y)$ is a.s.\ $C^1$; and
		\item For all $t>0$ and $x,y\in\R$,
			\begin{align*}
				\theta(t\,,x\,,y) &= \int_{\R^2} 
					p^{(\nu_1)}_t(x-a)p^{(\nu_2)}_t(y-b)
					\theta_0(a\,,b)\,\d a\,\d b\\
				&\hskip1in+ \int_{\R_+\times\R^2}
					p^{(\nu_1)}_{t-s}(x-a)p^{(\nu_2)}_{t-s}(y-b)
					\partial_y\theta(s\,,a\,,b)\,V(s\,,a)\,\d s\,\d a\,\d b,
			\end{align*}
			almost surely,
			where the final integral is interpretted as a Walsh integral,
			and is tacitly assumed to exist in the sense of Walsh \cite{Walsh};
			see also Dalang \cite{Dalang}.
	\end{compactenum}
\end{definition}

Appendix \ref{sec:SI} below highlights a summary of some
of the salient features of Walsh stochastic integrals.
			
We pause to say two things about Definition \eqref{def:mild}.
First, recall that $\rho$ is the Fourier transform 
of a finite Borel measure $\wh{\rho}$ on $\R$ (Herglotz's theorem).
This is because $\rho$ is a correlation function that is bounded and continuous
[see \eqref{rho():cont}].
In particular,
\begin{equation}\label{rho:max}
	0\le \rho(a) = \int_{-\infty}^\infty \cos(az)\,\wh{\rho}(\d z)\le \wh{\rho}(\R)=\rho(0)
	\qquad\text{for all $a\in\R$}.
\end{equation}
[The nonnegativity of $\rho(a)$ holds by assumption.]

Our second remark on Definition \ref{def:mild}
is this: \eqref{rho:max} ensures that a sufficient condition for
the existence of the stochastic integral in Part (2) of Definition \ref{def:mild} is that
$\theta$ is a predictable random field such that
\begin{equation}\left[\label{L2:der:T}\begin{split}
	&\text{$y\mapsto \theta(t\,,x\,,y)$ is $C^1$
		a.s.\ for every $t>0$ and $x\in\R$, and}\\
	&\sup_{(t,x,y)\in K}\E\left(\left|\partial_y
		\theta(t\,,x\,,y) \right|^2\right) <\infty
		\quad\text{for every compact set 
		$K\subset(0\,,\infty)\times\R^2$}.
\end{split}\right.\end{equation}

One can also consider weak solutions [in the sense of PDEs] instead of mild solutions.
We introduce/recall that notion next.
But first let us recall that the formal adjoint to $\cL$ is
\[
	\cL^* = -\partial_t -\nu_1\partial^2_x 
	- \nu_2\partial^2_y.
\]

\begin{definition}\label{def:weak}
	We say that $\theta:=\{\theta(t\,,x\,,y)\}_{t>0,x,y\in\R}$ is a \emph{weak solution}
	to \eqref{Kraichnan} if:
	\begin{compactenum}[(1)]
	\item $\theta(t)$ is a.s.\ locally integrable on $\R^2$ for every $t>0$;
	\item For every $t>0$ and $x\in\R$, 
		the random mapping $y\mapsto\theta(t\,,x\,,y)$ is a.s.\ $C^1$;
	\item For every non-random $\psi_0\in C_c(0\,,\infty)$
		and $\psi_1,\psi_2\in\mathscr{S}(\R)$;
		i.e., the Schwartz functions on $\mathbb{R}$.
		\begin{equation}\label{weak}
			\<\theta\,,\cL^*\varphi\>_{L^2(\R_+\times\R^2)}
			= \int_{(0,\infty)\times\R^2}
			\partial_y \theta(t\,,x\,,y)\,
			\varphi(t\,,x\,,y)\,V(t\,,x)\,\d t\,\d x\,\d y
			\quad\text{a.s., and}
		\end{equation}
		\begin{equation}\label{IV}
			\lim_{t\downarrow0}\left\< \theta(t)\,, \psi_1
			\otimes\psi_2\right\>_{L^2(\R^2)}
			=\left\< \theta_0\,,\psi_1\otimes\psi_2\right\>_{L^2(\R^2)}
			\qquad\text{in $L^2(\Omega)$},
		\end{equation}
		where, for every $t>0$ and $x,y\in\R$,
		\[
			(\psi_1\otimes\psi_2)(x\,,y) := \psi_1(x)\psi_2(y)
			\quad\text{and}\quad
			\varphi(t\,,x\,,y) := \psi_0(t)\cdot\left( \psi_1\otimes\psi_2\right)(x\,,y),
		\]
		and the stochastic integral on the right-hand side of \eqref{weak}
		is tacitly assumed to exist as a Walsh stochastic integral. 
	\end{compactenum}
\end{definition}

The main result of this section is an existence and uniqueness theorem
about It\^o/Walsh solutions of the generalized Kraichnan model
\eqref{Kraichnan}. Before we state that result, let us identify four
requisite technical criteria
that will be assumed to hold throughout this section.
The first is the low turbulence condition \eqref{cond:phys:2} that we 
recall next.

\begin{assumption}[Low turbulence]\label{assum:1}
	$\nu_2> \frac12\rho(0)$.
\end{assumption}

We will also need three regularity hypotheses on the initial profile $\theta_0$.

\begin{assumption}[Integrability in the second variable]\label{assum:2}
	There exists $\eta\in(0\,,1]$ such that
	\[
		\sup_{x\in\R}\int_{-\infty}^\infty \left(1+|y|^\eta\right) \|\theta_0(x\,,y)\|_k\,\d y<\infty
		\qquad\text{for every $k\ge2$}.
	\]
\end{assumption}

\begin{assumption}[Smoothness in the first variable]\label{assum:3}
	There exists $\alpha\in(0\,,1]$
	such that for every $k\ge2$ there exists $C_0=C_0(k\,,\alpha\,,\gamma)>0$ such that
	\[
		\int_{-\infty}^\infty\| \theta_0(x\,,y) - \theta_0(x',y)\|_k\,\d y
		\le C_0|x-x'|^\alpha
		\quad\text{for every $x,x'\in\R$}.
	\]
\end{assumption}

\begin{assumption}\label{assum:4}
	$\theta_0$ is independent of $V$, and  is continuous a.s.
\end{assumption}

Armed with these four conditions, we are ready to present the main
existence and uniqueness theorem for the It\^o/Walsh solution of
\eqref{Kraichnan}.

\begin{theorem}[Existence and uniqueness]\label{th:exist:U}
	If Assumptions \ref{assum:1}, \ref{assum:2},  \ref{assum:3}, and \ref{assum:4}
	are met, then  the generalized Kraichnan model
	\eqref{Kraichnan} has a mild solution $\theta:=\{\theta(t\,,x\,,y)\}_{t>0,x,y\in\R}$
	in the sense of Definition \ref{def:mild}.
	Moreover, $\theta$ satisfies the following:
	\begin{compactenum}
		\item For every $T>0$,
			\begin{equation}\label{ET^2}
				\sup_{t\in[0,T]}\sup_{x\in \R}\left[{\int_{-\infty}^\infty 
				\left\|\partial_y \theta(t\,, x\,, y)\right\|_{L^2(\Omega)}
				\d y+\int_{-\infty}^\infty 
				\|\theta(t\,, x\,, y)\|_{L^2(\Omega)}\d y}\right] < \infty;
			\end{equation}
		\item  $\theta$ is also a weak solution in the sense of Definition \ref{def:weak};
		\item If $\wt{\theta}$ is any other predictable random field that
			satisfies \eqref{L2:der:T}, \eqref{weak}, and \eqref{ET^2},
			then $\theta$ and $\wt{\theta}$ are modifications of one another; that is,
			\[
				\P\left\{ \theta(t\,,x\,,y) = \wt{\theta}(t\,,x\,,y)\right\}=1\quad
				\text{for every $t>0$ and $x,y\in\R$.}
			\]
	\end{compactenum}
	Finally, the following dissipation estimates are valid: As $t\to\infty$,
	\begin{equation}\label{L2:diss}
		\sup_{x,y\in\R}\E\left(|\theta(t\,,x\,,y)|^2\right)=O(1/ t)
		\quad\text{and}\quad
		\sup_{x,y\in\R}
		\E\left(\left| \partial_y \theta(t\,,x\,,y) \right|^2\right)
		=O\left(1/{t^2}\right).
	\end{equation}
\end{theorem}

\begin{remark}
	Suppose that there exists $\alpha\in(0\,,1]$ such that
	for every $k\ge2$ 
	\begin{equation}\label{mod:T_0}
		C = C(k\,,\alpha) := \sup_{\substack{x,x'\in\R\\x\neq x'}}
		\sup_{y\in\R}
		\frac{\|\theta_0(x\,,y) - \theta_0(x',y)\|_k }{|x-x'|^\alpha}<\infty.
	\end{equation}
	Suppose also  that Assumption \ref{assum:2} holds;
	that is, suppose that
	\[
		A =A(\eta\,,k) := \sup_{x\in\R}\int_{-\infty}^\infty 
		|y|^\eta \|\theta_0(x\,,y)\|_k\,\d y<\infty.
	\]
	By Chebyshev's inequality,
	$\sup_{x\in\R}\int_{|y|> q}\| \theta_0(x\,,y)\|_k\,\d y
	\le Aq^{-\eta} $ for all $q>0$. Therefore,
	\begin{align*}
		\int_{-\infty}^\infty\| \theta_0(x\,,y) - \theta_0(x',y)\|_k\,\d y &\le
			2q\sup_{y\in\R}\|\theta_0(x\,,y) - \theta_0(x',y)\|_k + 
			\int_{|y|> q}\| \theta_0(x\,,y) - \theta_0(x',y)\|_k\,\d y\\
		&\le2Cq|x-x'|^{\alpha} + 2Aq^{-\eta}.
	\end{align*}
	Optimize the right-hand side over the ancillary parameter $q$ in order to find that
	\[
		\int_{-\infty}^\infty\| \theta_0(x\,,y) - \theta_0(x',y)\|_k\,\d y
		\le 2A^{1/(1+\eta)}C^{\eta/(1+\eta)} { \eta^{1/(1+\eta)}}
		\left[1+\eta^{-1}\right] |x-x'|^{\alpha\eta/(1+\eta)}.
	\]
	In other words, Assumption \ref{assum:2} and
	a standard continuity-type condition such as \eqref{mod:T_0} together
	imply that Assumption \ref{assum:3} holds [with $\gamma := \alpha\eta/(1+\eta)$].
\end{remark}
Let us mention also the following result on the regularity of the solution of
\eqref{Kraichnan}, which has an additional H\"older-continuity requirement
[see \eqref{rho:cont}]---at the origin---for the correlation function $\rho$.\footnote{%
	In fact, the following well-known argument from Fourier analysis
	implies that \eqref{rho:cont} is a uniform
	H\"older  condition:    \eqref{rho:max} ensures that for all $a,b\in\R$,
	\[
		|\rho(a) - \rho(b) | \le \int_{-\infty}^\infty \left| \e^{iaz}-\e^{ibz}
		\right|\wh{\rho}(\d z)
		=2\int_{-\infty}^\infty\sqrt{\left[ 1 - \cos((b-a)z)\right]} \wh{\rho}(\d z)
		\leq 2\sqrt{\rho(0)}\sqrt{\left[ \rho(0)-\rho(b-a)\right]}.
	\]
	}

\begin{theorem}\label{th:cont:U}
	Suppose there exist $\varpi\in(0\,,2]$ and $C_*>0$ such that
	\begin{equation}\label{rho:cont}
		\rho(0)-\rho(z) \le C_*|z|^\varpi
		\qquad\text{for all $z\in\R$},
	\end{equation}
	and that there exist $\alpha,\zeta\in(0\,,1]$
	such that for every $k\ge 2$ there exists a real number $\wt{A}_k$
	such that
	\begin{equation}\label{T_0-T_0}
		\E\left( |\theta_0(a\,,b) - \theta_0(a',b')|^k\right) \le \wt{A}_k\left\{
		|a-a'|^{k\alpha}+ |b-b'|^{k\zeta}\right\},
	\end{equation}
	uniformly for every $a,a',b,b'\in\R$. Then, with probability one:
	\[
		\theta \in \bigcap_{\substack{0<s<\frac12\min(\alpha,\zeta\varpi/2)\\
			0<x<\min(\alpha,\zeta\varpi/2)\\
			0<y<\zeta}}
		\mathscr{C}^{s,x,y} \left( (0\,,\infty)\times\R^2\right)\qquad\text{a.s.},
	\]
	where $\mathscr{C}^{s,x,y}((0\,,\infty)\times \R^2)$ denotes 
	the space of all real-valued functions $f:(0\,, \infty)\times \R^2\to \R$
	such that $a\mapsto f(a\,,\cdot\,,\cdot)$, $b\mapsto f(\cdot\,,b\,,\cdot)$,
	and $c\mapsto f(\cdot\,,\cdot\,,c)$ are respectively H\"older continuous
	with respective indices $s$, $x$, and $y$.
\end{theorem}

\begin{remark}
	It is a well-known fact that $\varpi\le 2$ unless $\rho(z)=\rho(0)$
	for all $z\in\R$. Here is the short proof:
	By Fatou's lemma and  \eqref{rho:max}, 
	\[
		\liminf_{z\to 0}\,\frac{\rho(0)-\rho(z)}{z^2}
		\ge \tfrac12\int_{-\infty}^\infty x^2 \,\wh\rho(\d x).
	\]
	If $\rho$ is not a constant function, then \eqref{rho:max}
	ensures that $\wh\rho \neq\delta_0$ and hence 
	$\int_{-\infty}^\infty x^2\,\wh\rho(\d x)\in(0\,,\infty]$.
	This is enough to imply that $\varpi\le 2$, as desired. 
\end{remark}

Theorems \ref{th:exist:U} and \ref{th:cont:U}, and their proofs,
have a number of consequences. We mention some of them next
in order to highlight the ``physical'' nature of the Kraichnan SPDE
\eqref{Kraichnan}. The first consequence is about the dissipative nature
of the solution to \eqref{Kraichnan}. We emphasize that the following
uniform a.s.\ decay rate is consistent with the distributional one
from \eqref{L2:diss}.

\begin{proposition}[Dissipation]\label{pr:diss}
	Suppose the conditions of Theorems \ref{th:exist:U} and \ref{th:cont:U}
	are met. Then,
	for every $m>0$, the following is valid with probability one:
	\[
		\sup_{|x|\le m}\sup_{y\in\R}|\theta(t\,,x\,,y)| = O\left( 1/\sqrt t\right)\qquad\text{as
		$t\downarrow0$.}  \ 
	\]
\end{proposition}
We will see in Remark \ref{rem:diss} below that the dissipation rate $1/\sqrt t$ is
unimproveable.

Next we mention three back-to-back 
consequences of Theorems \ref{th:exist:U} and \ref{th:cont:U}.
These results  are the analogues of the maximum principle
in the present, stochastic setting.

\begin{proposition}[Positivity]\label{pr:pos}
	Suppose that the conditions of Theorems \ref{th:exist:U} and \ref{th:cont:U}
	are met, and that $\P\{ \theta_0(x\,,y)>0\}=1$ for all $x,y\in\R$. Then,
	\[
		\P\left\{ \theta(t\,,x\,,y) >0\ \text{ for all $t>0$ and $x,y\in\R$}\right\}=1.
	\]
\end{proposition}

\begin{proposition}[Conservation of mass]\label{pr:mass}
	Suppose that the conditions of Theorems \ref{th:exist:U} and \ref{th:cont:U}
	are met, and $\P\{\int_{-\infty}^\infty \theta_0(x\,,y)\,\d y=1\}=1$ for all $x\in\R$. Then,
	\[	
		\P\left\{ \int_{-\infty}^\infty \theta(t\,,x\,,y)\,\d y=1\ \text{ for all 
		$t>0$ and $x\in\R$}\right\}=1.
	\]
\end{proposition}

\begin{proposition}[Comparison]\label{pr:comparison}
	Suppose that the conditions of Theorems \ref{th:exist:U} and \ref{th:cont:U}
	are met for $\theta_0$ and another initial data $\wt{\theta}_0$. Let
	$\theta$ and $\wt{\theta}$ denote the solutions to \eqref{Kraichnan},
	subject to respective initial data $\theta_0$ and $\wt{\theta}_0$. Then,
	\[
		\P\left(\left. \theta(t)\le \wt{\theta}(t)\text{ for all $t>0$}\ \right|\
		\theta_0\le\wt{\theta}_0\right)=1, 
	\]
	provided additionally that $\P\{\theta_0\le\wt{\theta}_0\}>0$.
\end{proposition}

\subsection{Outline of the proof of Theorem \ref{th:exist:U}}

As was observed by Majda \cite{Majda,Majdaa},
the  stratified structure of the velocity field $V$ in \eqref{Kraichnan}
lends itself well to an application of the Fourier
transform in the variable $y$ (see also Bronski and McLaughlin \cite{BM1997,BM2000,BM2000a}). 
With this in mind,
let us define $U$ to be the Fourier transform of $\theta$
in its $y$ variable; that is, somewhat informally,
\begin{equation}\label{U:T}
	U(t\,,x\,,\xi) := \int_{-\infty}^\infty\e^{i\xi y}\theta(t\,,x\,,y)\,\d y.
\end{equation}
In order to prove Theorem \ref{th:exist:U}, we first prove  that
$U$ exists, and has a sufficiently good version, thanks to Assumptions \ref{assum:1}
through \ref{assum:4}. And then we invert the Fourier
transform \eqref{U:T}, thereby also establish the existence and uniqueness of
$\theta$ as a by product.
 
Unfortunately, \eqref{U:T} is an informal definition: It will turn out that
$\theta(t\,,x\,,\cdot)$ is in general not integrable with probability one
for all $t>0$ and $x\in\R$. Still, one can 
think of $U$ as a Fourier transform in the sense of distributions
provided only that $\theta(t\,,x\,,\cdot)\in L^1_{\textit{loc}}(\R)$ a.s.\
for all $t>0$ and $x\in\R$. The ensuing {\it a priori}
estimates will show that this local integrability property holds
under Assumptions \ref{assum:1}--\ref{assum:4}.

By analogy with classical linear PDEs, if the random field $U$ were at 
all well defined, then it would have to solve the 
complex-valued SPDE,
\begin{equation}\label{PAM:U}
	\partial_t U(t\,,x\,,\xi) = 
	\nu_1 \partial^2_x U(t\,,x\,,\xi)
	- \nu_2\xi^2 U(t\,,x\,,\xi)+ i\xi U(t\,,x\,,\xi) V(t\,,x),
\end{equation}
subject to $U(0)=\wh{\theta}_0$. 
One can  interpret \eqref{PAM:U} easily 
as an infinite family of
complex-valued, but otherwise standard, It\^o/Walsh SPDEs, one for every $\xi\in\R$.
As such, it is not difficult to solve it in order to obtain the random
field $U$.
We plan to ``invert'' the Fourier transform operation---%
compare with \eqref{U:T}---in order to construct $\theta$. This endeavor
will require the assumptions of Theorem \ref{th:exist:U}. If and when
this is possible, it is not hard to prove that this procedure 
will yield the desired solution to \eqref{Kraichnan}, as well.

As an aside, let us mention that one could think of \eqref{PAM:U} as
a two-dimensional, real-valued SPDE as follows: Define
$X:=\text{\rm Re}\, U$ and $Y:=\text{\rm  Im}\, U$ in order see that
$(X\,,Y)$ solves
\begin{equation}\label{XY}\begin{split}
	\partial_t X(t\,,x\,,\xi) &= 
		\nu_1 \partial^2_x X(t\,,x\,,\xi)
		- \nu_2\xi^2 X(t\,,x\,,\xi)-\xi Y(t\,,x\,,\xi) V(t\,,x),\\
	\partial_t Y(t\,,x\,,\xi) &= 
		\nu_1 \partial^2_x Y(t\,,x\,,\xi)
		- \nu_2\xi^2 Y(t\,,x\,,\xi)+ \xi X(t\,,x\,,\xi) V(t\,,x),
\end{split}\end{equation}
subject to the obvious initial condition. In the case that $\nu_2$
is replaced by zero,  the SPDE \eqref{XY} is related loosely
to the mutually-catalytic super Brownian motion system of
D\"oring and Mytnik \cite{DoeringMytnik}. Though there  also are obvious differences
between \eqref{XY} and such super Brownian motions as well.

We now return to the construction of the random field $U$.
In accord with the theory of Walsh \cite{Walsh},
we seek to solve \eqref{PAM:U} by rewriting
it as the following Walsh-type 
stochastic-integral equation:
\begin{align*}
	U(t\,,x\,,\xi)
		&= \int_{-\infty}^\infty p_t^{(\nu_1)}(x-x')U_0(x',\xi)\,\d x'
		-\nu_2\xi^2\int_{(0,t)\times\R} p_{t-s}^{(\nu_1)}(x-x') U(s\,,x',\xi)\,\d s\,\d x'\\
	&\hskip2.3in+ i\xi\int_{(0,t)\times\R} p_{t-s}^{(\nu_1)}(x-x') U(s\,,x',\xi)
		V(s\,,x')\,\d s\,\d x',
\end{align*}
where $p^{(\nu)}$ denotes the fundamental solution of
the heat operator $\partial_t-\nu\partial^2_x$; see \eqref{p}. 

The preceding is
a complex version of the sort of SPDE that is treated in Walsh \cite{Walsh}. 
Therefore, it is not hard to 
use the technology of Walsh \cite{Walsh} to prove
that \eqref{PAM:U}---equivalently, \eqref{XY}---has a unique 
strong solution, among other things. 
The following, perhaps more interesting, {\it a priori} result
estimates carefully the moment Lyapunov exponent $\lambda_2$ of that solution
by showing that
\begin{equation}\label{lambda_2}
	\lambda_2(x\,,\xi) := 
	\limsup_{t\to\infty} t^{-1} \log\|U(t\,,x\,,\xi)\|_2
	\le - \nu_2 + \tfrac12\rho(0)
	\qquad\text{for all $x,\xi\in\R$}.
\end{equation}
Though we have not attempted to derive a matching lower bound, we believe that
the preceding inequality is an identity.
In any case, we can see from \eqref{lambda_2}
and Assumption \ref{assum:1}  that 
$\lambda_2$ is \emph{strictly} negative. 
A quantitative form of \eqref{lambda_2} will allow us to
``invert'' \eqref{U:T} under Assumption \ref{assum:1},
and hence establish Theorem \ref{th:exist:U}.

The derivation of \eqref{lambda_2} requires some
care, in part because the solution to \eqref{PAM:U} is complex valued.
So we shall proceed with care, paying careful attention to numerical constants
that arise along the way.

The above bound for $\lambda_2$ is based on the following,
more useful, quantitative result.

\begin{theorem}\label{th:U:exist}
	Suppose $U(0):(x\,,\xi)\mapsto U_0(x\,,\xi)$ is a jointly measurable random field
	that is independent of $V$ and satisfies $\sup_{x\in\R}\E(|U_0(x\,,\xi)|^k)<\infty$
	for every $\xi\in\R$ and $k \geq 2$.
	Choose and fix some $\nu_2>0$. Then,  for every $\xi\in\R$,
	\eqref{PAM:U} has a mild solution
	$U[\xi]$ that satisfies the following for every $\varepsilon\in(0\,,1)$,
	$t>0$, and $x,\xi\in\R$:
	\[
		\sup_{x\in\R}\E\left( |U(t\,,x\,,\xi)|^2\right) \le {\varepsilon^{-2}}
		\exp\left( -2\left[{\nu_2 - \frac{\rho(0)}{2(1-\varepsilon)^2}}\right]\xi^2 t\right)
		\sup_{x\in\R}\E\left( |U_0(x\,,\xi)|^2\right).
	\]
	Moreover, any other such mild solution is a modification of $U[\xi]$. Finally,
	for all $\varepsilon\in(0\,,1)$, $t\ge0$, and $x,\xi\in\R$,
	\[
		\sup_{x\in\R}\E\left( |U(t\,,x\,,\xi)|^k\right) \le {\varepsilon^{-k}}
		\exp\left( -k\left[{\nu_2 - \frac{4\rho(0)k}{(1-\varepsilon)^2}}\right]\xi^2 t\right)
		\sup_{x\in\R}\E\left( |U_0(x\,,\xi)|^k\right).
	\]
\end{theorem}

Once we have a good version of $U$ that has a well-controlled second-moment
Lyapunov exponent, we can readily ``invert'' the Fourier transform in \eqref{U:T} in order
to obtain the solution $\theta$ to the Kraichnan model \eqref{Kraichnan}. In the remainder
of this section we carry out the above program.

The astute reader might wonder why we have included bounds for all high-order
Lyapunov exponents when we claim that the important one is the second-moment  Lyapunov
exponent $\lambda_2$. The reason will become apparent when we use the 
high-order Lyapunov exponents
to obtain some of the required {\it a priori} regularity; see the discussion that follows
Lemma \ref{lem:conclusion} below, for example.

\subsection{Stochastic convolution}

Owing to Definition \ref{def:mild}, linear SPDEs are related to ``stochastic
convolutions'' in a manner that is analogous to the relationship between linear
PDEs and space-time convolutions. In this subsection we develop some norm
inequalities for stochastic convolutions. We will use these inequalities in 
the next subsection (see \S\ref{sec:U:Exist:unique}) in order
to verify Theorem \ref{th:U:exist}. 

Let us start with a more-or-less standard definition.

\begin{definition}\label{def:good:integrand}
	Suppose $\Phi=\{\Phi(t\,,x)\}_{t\ge0,x\in\R}$ is a space-time random field.
	We say that $\Phi$ is \emph{Walsh integrable} if $\Phi$ is predictable
	in the sense of Walsh \cite{Walsh} and satisfies
	\[
		\int_0^t\d s\int_{-\infty}^\infty\d x'\int_{-\infty}^\infty\d x''\
		p_{t-s}^{(\nu_1)}(x-x') p_{t-s}^{(\nu_1)}(x-x'')\E\left( |\Phi(s\,,x')
		\Phi(s\,,x'')|\right)\rho(x'-x'')<\infty,
	\]
	for every $t>0$ and $x\in\R$. If $\Phi$ is complex valued, then we say
	that $\Phi$ is a \emph{Walsh integrable} if the real part and 
	imaginary part of $\Phi$ are both Walsh integrable.
\end{definition}

Next is a simple extension of a standard definition to the present,  complex-valued setting.

\begin{definition}
	Let $\Phi:=\{\Phi(t\,,x)\}_{t\ge0,x\in\R}$ be a complex-valued,
	space-time random field. We say that $\Phi$ is \emph{predictable}
	if $\text{\rm Re\,}\Phi:=\{\text{\rm Re\,}\Phi(t\,,x)\}_{t\ge0,x\in\R}$
	and $\text{\rm Im\,}\Phi:=\{\text{\rm Im\,}\Phi(t\,,x)\}_{t\ge0,x\in\R}$
	are predictable random fields---in the sense of Walsh \cite{Walsh}---both with respect
	to the same filtration of $\sigma$-algebras.
\end{definition}

On a few occasions we will refer to the following simple fact,
which is isolated as a little lemma for ease of reference.

\begin{lemma}\label{lem:WI:easy}
	Let $\Phi :=\{\Phi(t\,,x)\}_{t\ge0,x\in\R}$ be a complex-valued
	predictable random field that satisfies 
	\begin{equation}\label{eq:WI}
		\int_0^t\d s\int_{-\infty}^\infty\d x'\int_{-\infty}^\infty\d x''\
		p_{t-s}^{(\nu_1)}(x-x') p_{t-s}^{(\nu_1)}(x-x'')\E\left( |\Phi(s\,,x')
		\Phi(s\,,x'')|\right)<\infty,
	\end{equation}
	for every $t>0$ and $x\in\R$.
	Then, $\Phi$ is Walsh integrable.
\end{lemma}

\begin{proof}
	Since $|zw|^2 \ge (\text{\rm Re}\,z\cdot \text{\rm Re}\,w)^2
	+ (\text{\rm Im}\,z\cdot\text{\rm Im}\,w)^2$
	for every two complex numbers $z$ and $w$, we take square roots 
	in order to see that
	\[
		|\text{\rm Re}\,z\cdot\text{\rm Re}\,w|
		+ |\text{\rm Im}\,z\cdot \text{\rm Im}\,w| \le
		2|zw|\qquad\text{for all $z,w\in\C$}.
	\]
	Thus, we may use this with $z=\Phi(s\,,x')$
	and $w=\Phi(s\,,x'')$, take expectations,
	and appeal to \eqref{rho:max} in order to find that,
	for all $t>0$ and $x\in\R$,
	\begin{align*}
		&\int_0^t\d s\int_{-\infty}^\infty\d x'\int_{-\infty}^\infty\d x''\
			p_{t-s}^{(\nu_1)}(x-x') p_{t-s}^{(\nu_1)}(x-x'')
			\E\left( |\text{\rm Re\,}\Phi(s\,,x')\cdot
			\text{\rm Re\,}\Phi(s\,,x'')|\right)\rho(x'-x'')\\
		&\quad +\int_0^t\d s\int_{-\infty}^\infty\d x'\int_{-\infty}^\infty\d x''\
			p_{t-s}^{(\nu_1)}(x-x') p_{t-s}^{(\nu_1)}(x-x'')
			\E\left( |\text{\rm Im\,}\Phi(s\,,x')\cdot
			\text{\rm Im\,}\Phi(s\,,x'')|\right)\rho(x'-x'')\\
		&\le 2{\rho(0)}\int_0^t\d s\int_{-\infty}^\infty\d x'\int_{-\infty}^\infty\d x''\
			p_{t-s}^{(\nu_1)}(x-x') p_{t-s}^{(\nu_1)}(x-x'')
			\E\left( |\Phi(s\,,x')\cdot \Phi(s\,,x'')|\right),
	\end{align*}
	which is finite.
\end{proof}

Thanks to Lemma \ref{lem:WI:easy},
in order to verify that a random field $\Phi:=\{\Phi(t\,,x)\}_{t\ge0,x\in\R}$ is Walsh integrable, 
it suffices to check that $\Phi$ is both predictable and satisfies the integrability
condition \eqref{eq:WI}. We first verify the latter integrability condition by developing a ``stochastic
Young's inequality'' as in Foondun and Khoshnevisan \cite{FK} and
Conus, Khoshnevisan \cite{CK}. With this aim in mind, let us introduce some
terminology.

\begin{definition}
	Let us define, for every complex-valued space-time random field $\Phi=\{\Phi(t\,,x)\}_{t\ge0,x\in\R}$
	and all real numbers $k\ge2$ and $\beta>0$,
	\begin{equation}\label{N}
		\mathcal{N}_{k,\beta}(\Phi) := \sup_{t\ge0}\sup_{x\in\R}
		\e^{-\beta t}\|\Phi(t\,,x)\|_k.
	\end{equation}
\end{definition}

Clearly, every $\mathcal{N}_{k,\beta}$ is a norm on the vector space
of all space-time random fields that have finite $\mathcal{N}_{k,\beta}$-norm,
provided that we identify two random fields when they are modifications of
one another (as one always does, any way). The following shows a
sufficient condition, in terms of the norms in \eqref{N}, for the integrability
condition \eqref{eq:WI} to hold.

\begin{lemma}\label{lem:WI}
	If $\Phi=\{\Phi(t\,,x)\}_{t\ge0,x\in\R}$ is a complex-valued,
	space-time random field that
	satisfies
	$\mathcal{N}_{2,\beta}(\Phi)<\infty$ for some $\beta>0$, then
	$\Phi$ satisfies the integrability condition \eqref{eq:WI}.
\end{lemma}

\begin{proof}
	By the Cauchy--Schwarz inequality, 
	\[
		\E( |\Phi(s\,,x')\Phi(s\,,x'')|) \le \sup_{a\in\R}\|\Phi(s\,,a)\|_2^2
		\le \e^{2\beta s}\left[\mathcal{N}_{2,\beta}(\Phi)\right]^2,
	\]
	uniformly
	for all $s>0$ and $x',x''\in\R$. Because 
	$\int_{-\infty}^\infty p^{(\nu_1)}_{t-s}(w)\,\d w=1$ for all $0<s<t$,
	this proves that
	\begin{align*}
		\int_0^t\d s\int_{-\infty}^\infty\d x'\int_{-\infty}^\infty\d x''\
			p_{t-s}^{(\nu_1)}(x-x') p_{t-s}^{(\nu_1)}(x-x'')
			\E\left( |\Phi(s\,,x')\Phi(s\,,x'')|\right)
			&\le \left[\mathcal{N}_{2,\beta}(\Phi)\right]^2\int_0^t\e^{2\beta s}\,\d s\\
		&\le \frac{\e^{2\beta t}}{2\beta}\,\left[\mathcal{N}_{2,\beta}(\Phi)\right]^2,
	\end{align*}
	which is finite.
\end{proof}

We now state and prove the stochastic Young's inequality that was alluded to earlier.

\begin{lemma}[A stochastic Young's inequality]\label{lem:Young}
	Let $\Phi=\{\Phi(t\,,x)\}_{t\ge0,x\in\R}$ be a complex-valued, predictable random field
	that satisfies $\mathcal{N}_{k,\beta}(\Phi)<\infty$ for some $k\ge 2$ and $\beta>0$.
	Then, for all $\nu>0$, the stochastic convolution,
	\begin{equation}\label{eq:SC}
		\left( p^{(\nu)}\circledast\Phi \right)(t\,,x) := 
		\int_{(0,t)\times\R} p_{t-s}^{(\nu)}(x-x') \Phi(s\,,x')V(s\,,x')\,\d s\,\d x'
	\end{equation}
	is a well-defined, complex-valued Walsh integral for every $t>0$ and $x\in\R$, and
	\begin{equation}\label{eq:Young}
		\mathcal{N}_{k,\beta}\left( p^{(\nu)}\circledast\Phi \right) 
		\le \sqrt{\frac{c_k\rho(0)}{2\beta}}\, \mathcal{N}_{k,\beta}(\Phi),
	\end{equation}
	where 
	\begin{equation}\label{c_k}
		c_k = \begin{cases}
			1&\text{if $k=2$},\\
			8k&\text{if $k>2$}.
		\end{cases}
	\end{equation}
\end{lemma}

\begin{definition}
	In order to make future notation consistent, from now on we  tacitly assume that
	$(p^{(\nu)}\circledast\Phi)(0\,,x)=0$ for all $x\in\R$ and all predictable, 2-parameter
	random fields $\Phi$.
\end{definition}

Before we prove Lemma \ref{lem:Young}, let us make two more observations.

\begin{remark}
	The preceding lemma says that
	$(p^{(\nu)}\circledast \text{\rm Re\,}\Phi)(t\,,x)$ and
	$(p^{(\nu)}\circledast \text{\rm Im\,}\Phi)(t\,,x)$ are well-defined
	Walsh integrals, and tacitly \emph{defines}
	\[
		(p^{(\nu)}\circledast\Phi)(t\,,x):=(p^{(\nu)}\circledast\text{\rm Re\,}\Phi)(t\,,x) 
		+ i(p^{(\nu)}\circledast\text{\rm Im\,}\Phi)(t\,,x),
	\]
	for every $t>0$ and $x\in\R$.
\end{remark}

\begin{remark}\label{rem:Young}
	A moment's thought shows that \eqref{eq:Young} is short hand for the moment inequality,
	\[
		\E\left( \left| \left( p^{(\nu)}\circledast\Phi \right)(t\,,x)\right|^k\right)
		\le \left(\frac{c_k\rho(0)}{2\beta}\right)^{k/2}[\mathcal{N}_{k,\beta}(\Phi)]^k
		\,\e^{\beta k t} 
		\quad\text{for all $t\ge0$, $\beta,\nu>0$, and $x\in\R$},.
	\]
	Of course, this inequality has content when, and only when, $\mathcal{N}_{k,\beta}(\Phi)$
	is finite.
\end{remark}

\begin{proof}[Proof of Lemma \ref{lem:Young}]
	Choose and fix $t>0$ and $x\in\R$, and define, for every $\tau\in(0\,,t]$,
	\begin{align*}
		\mathcal{M}_\tau &:= \int_{(0,\tau)\times\R} 
			p^{(\nu)}_{t-s}(x-x') \Phi(s\,,x')V(s\,,x')\,\d s\,\d x'\\
		&:= \int_{(0,\tau)\times\R} p_{t-s}^{(\nu)}(x-x') 
			\text{\rm Re\,}\Phi(s\,,x')V(s\,,x')\,\d s\,\d x'\\
		&\hskip2.7in+ 
			i\int_{(0,\tau)\times\R} p_{t-s}^{(\nu)}(x-x') 
				\text{\rm Im\,}\Phi(s\,,x')V(s\,,x')\,\d s\,\d x'.
	\end{align*}
	The theory of Walsh \cite{Walsh} insures that
	$M$ is a complex-valued, continuous local martingale that is indexed by $(0\,,t]$.
	This means that both $\text{\rm Re\,}\mathcal{M}:=
	\{\text{\rm Re\,}\mathcal{M}_\tau\}_{\tau\in(0,t]}$ and
	$\text{\rm Im\,}\mathcal{M}:=\{\text{\rm Im\,}
	\mathcal{M}_\tau\}_{\tau\in(0,t]}$ are real-valued, continuous
	local martingales [in the usual sense], both with respect to the same filtration. Because
	\begin{align*}
		\text{\rm Re\,}\mathcal{M}_\tau &= \int_{(0,\tau)\times\R} p_{t-s}^{(\nu)}(x-x') 
			\text{\rm Re\,}\Phi(s\,,x')V(s\,,x')\,\d s\,\d x'
			\quad\text{for $\tau\in(0\,,t]$, and}\\
		\text{\rm Im\,}\mathcal{M}_\tau &= \int_{(0,\tau)\times\R} p_{t-s}^{(\nu)}(x-x') 
			\text{\rm Im\,}\Phi(s\,,x')V(s\,,x')\,\d s\,\d x'
			\quad\text{for $\tau\in(0\,,t]$,}
	\end{align*}
	it follows from Walsh's theory that the preceding local martingales have respective
	quadratic variations,
	\begin{align*}
		\left\< \text{\rm Re\,}\mathcal{M}\right\>_\tau&= 
			\int_0^\tau\d s\int_{-\infty}^{\infty}\d x''
			\int_{-\infty}^\infty\d x'\ p_{t-s}^{(\nu)}(x-x') p_{t-s}^{(\nu)}(x-x'')
			\text{\rm Re\,}\Phi(s\,,x')\text{\rm Re\,}\Phi(s\,,x'')\rho(x'-x''),\\
		\left\< \text{\rm Im\,}\mathcal{M}\right\>_\tau&= 
			\int_0^\tau\d s\int_{-\infty}^{\infty}\d x''
			\int_{-\infty}^\infty\d x'\ p_{t-s}^{(\nu)}(x-x') p_{t-s}^{(\nu)}(x-x'')
			\text{\rm Im\,}\Phi(s\,,x')\text{\rm Im\,}\Phi(s\,,x'')\rho(x'-x'').
	\end{align*}
	We may now borrow from the proof of Lemma \ref{lem:WI} as follows:
	Fubini's theorem and \eqref{rho:max} together yield
	\begin{align*}
		\E\left( \left| \text{\rm Re\,}\mathcal{M}_t\right|^2\right) &\le
			\rho(0)\int_0^t\d s\int_{-\infty}^{\infty}\d x''
			\int_{-\infty}^\infty\d x'\ p_{t-s}^{(\nu)}(x-x') p_{t-s}^{(\nu_1)}(x-x'')
			\left\| \text{\rm Re\,}\Phi(s\,,x')\right\|_2
			\left\|\text{\rm Re\,}\Phi(s\,,x'')\right\|_2\\
		&\leq \rho(0)\int_0^t\d s \left( \int_{-\infty}^{\infty}\d x' p_{t-s}^{(\nu)}(x-x') 
			\left\|\text{\rm Re\,}\Phi(s\,,x')\right\|_2 \right)^2\\
		& \leq \rho(0)\int_0^t\d s \int_{-\infty}^{\infty}\d x' p_{t-s}^{(\nu)}(x-x') 
			\left\|\text{\rm Re\,}\Phi(s\,,x')\right\|_2 ^2
	\end{align*}
	since $\int_{-\infty}^\infty p^{(\nu)}_{t-s}(w)\,\d w=1$. 
	The same inequality holds when we replace
	$\text{\rm Re\,}\mathcal{M}$ by $\text{\rm Im\,}\mathcal{M}$. Therefore,
	\begin{align*}
		\E\left( \left| \left( p^{(\nu)}\circledast\Phi \right)(t\,,x)\right|^2\right) =
			\E\left( |\mathcal{M}_t|^2\right) 
			&=\E\left( \left| \text{\rm Re\,}\mathcal{M}_t\right|^2\right)
			+ \E\left( \left| \text{\rm Im\,}\mathcal{M}_t\right|^2\right)\\
		&\hskip1.5in\le \rho(0)\int_0^t \sup_{a\in\R}
			\left\|\Phi(s\,,a)\right\|_2^2\,\d s.
	\end{align*}
	We multiply and divide, inside the integral, by $\exp(-2\beta s)$ and maximize
	the resulting integrand in order to see that
	\[
		\E\left( \left| \left( p^{(\nu)}\circledast\Phi \right)(t\,,x)\right|^2\right) 
		\le \rho(0)\left[\mathcal{N}_{2,\beta}(\Phi)\right]^2\int_0^t \e^{2\beta s}\,\d s 
		\le \frac{\rho(0)\e^{2\beta t}}{2\beta}\, \left[\mathcal{N}_{2,\beta}(\Phi)\right]^2.
	\]
	Take square roots, divide both sides by $\exp(\beta t)$, and maximize
	both sides over $t$ and $x$ in order to deduce the announced bound---see
	\eqref{eq:Young}---%
	for $\mathcal{N}_{2,\beta}(p^{(\nu)}\circledast \Phi)$ in terms of $c_2$.
	
	For the $L^k(\Omega)$ norm inequalities
	we appeal to the Carlen--Kree \cite{CKree} bound on Davis' optimal constant \cite{Davis}
	in the Burkholder--Davis--Gundy inequality 
	\cite{Burkholder1966,BDG1972,BG1970} in order to see that
	\begin{equation}\label{BDG:Re:Im}\begin{split}
		\E\left( |\text{\rm Re\,}\mathcal{M}_t|^k\right) &\le 
			(4k)^{k/2}\E\left( \<\text{\rm Re\,}\mathcal{M}\>_t^{k/2}\right),\\
		\E\left( |\text{\rm Im\,}\mathcal{M}_t|^k\right) &\le 
			(4k)^{k/2}\E\left( \<\text{\rm Im\,}\mathcal{M}\>_t^{k/2}\right).
	\end{split}\end{equation}
	By working directly with the formula for $\<\text{\rm Re\,}\mathcal{M}\>_t$,
	and thanks to the Minkowski inequality, we can see that
	\begin{align*}
		&\left\| \<\text{\rm Re\,}\mathcal{M}\>_t\right\|_{k/2}\\
		&\le \rho(0)\int_0^t\d s\int_{-\infty}^{\infty}\d x''
			\int_{-\infty}^\infty\d x'\ p_{t-s}^{(\nu_1)}(x-x') p_{t-s}^{(\nu_1)}(x-x'')
			\left\| \text{\rm Re\,}\Phi(s\,,x')\text{\rm Re\,}\Phi(s\,,x'') \right\|_{k/2}\\\\
		&\le \rho(0)\int_0^t\d s\int_{-\infty}^{\infty}\d x''
			\int_{-\infty}^{\infty}\d x'\ p_{t-s}^{(\nu_1)}(x-x') p_{t-s}^{(\nu_1)}(x-x'')
			\left\| \text{\rm Re\,}\Phi(s\,,x')\right\|_k
			\left\|\text{\rm Re\,}\Phi(s\,,x'') \right\|_k.
	\end{align*}
	In the last line we used the Cauchy--Schwarz inequality in the following form:
	$\|XY\|_{k/2}\le \|X\|_k\|Y\|_k$ for every $X,Y\in L^k(\Omega)$. In any case,
	the preceding yields
	\begin{align}\notag
		\left\| \<\text{\rm Re\,}\mathcal{M}\>_t\right\|_{k/2}
			&\le \rho(0)\int_0^t\d s\int_{-\infty}^{\infty}\d x''
			\int_{-\infty}^\infty\d x'\ p_{t-s}^{(\nu_1)}(x-x') p_{t-s}^{(\nu_1)}(x-x'')
			\left\| \Phi(s\,,x')\right\|_k\left\|\Phi(s\,,x'') \right\|_k
			\label{thespot}\\
		&\le \rho(0)\int_0^t \sup_{a\in\R}\left\| \Phi(s\,,a)\right\|_k^2\,\d s\\\notag
		&\le \rho(0)\left[\mathcal{N}_{k,\beta}(\Phi)\right]^2
			\int_0^t\e^{2\beta s}\,\d s\\\notag
		&\le \frac{\rho(0)\e^{2\beta t}}{2\beta}\,
			\left[\mathcal{N}_{k,\beta}(\Phi)\right]^2.
	\end{align}
	The same inequality holds, for the same sort of reason, if we replace
	the real part of $M$ by its imaginary part. Therefore, by \eqref{BDG:Re:Im},
	\begin{align*}
		\E\left( \left| \left( p^{(\nu)}\circledast\Phi \right)(t\,,x)\right|^k\right) 
			&= \E\left( |\mathcal{M}_t|^k\right)\\
		&\le  2^{(k-2)/2}\left\{ \E\left( \left| \text{\rm Re\,}\mathcal{M}_t\right|^k\right)
			+ \E\left( \left| \text{\rm Im\,}\mathcal{M}_t
			\right|^k\right)\right\}\\
		&\le \frac{1}{2}(8k)^{k/2}\left\{ \E\left( 
			\<\text{\rm Re\,}\mathcal{M}\>_t^{k/2}\right)
			+ \E\left( \<\text{\rm Im\,}\mathcal{M}\>_t^{k/2}\right)\right\}\\
		&\le \left(\frac{c_k\rho(0)}{2\beta}\right)^{k/2}\e^{\beta kt}
			\,\left[\mathcal{N}_{k,\beta}(\Phi)\right]^k.
	\end{align*}
	[It might help to recall that $c_k:=8k$ because $k>2$.]
	Take $k$th root of both sides, divide both sides by $\exp(\beta t)$,
	and then optimize over $t$ and $x$ to finish.
\end{proof}

In some  of the ensuing applications---for example see
Lemma \ref{u-u:t=0}---the factor $\beta^{-1/2}$ on the right-hand side of
\eqref{eq:Young} will be too crude; see also Remark \ref{rem:Young}. The following 
finite time-horizon
variation of Lemma \ref{lem:Young} will be used in such instances.

\begin{lemma}\label{lem:Young:redux}
	Let $\Phi=\{\Phi(t\,,x)\}_{t\ge0,x\in\R}$ be a complex-valued, predictable random field
	that satisfies $\mathcal{N}_{k,\beta}(\Phi)<\infty$ for some $k\ge 2$ and $\beta>0$.
	Then, 
	\[
		\sup_{\nu>0}
		\sup_{x\in\R}\sup_{s\in(0,t)}\E\left( |(p^{(\nu)}\circledast\Phi)(s\,,x)|^k\right)\le
		({16\rho(0)kt})^{k/2}\sup_{x\in\R}\sup_{s\in(0,t)}\E\left( |\Phi(s\,,x)|^k\right)
	\]
\end{lemma}

\begin{proof}
	It is easy to see that $\sup_{x\in\R}\sup_{s\in(0,t)}\E(|\Phi(s\,,x)|^k)
	\le\exp(\beta kt)[\mathcal{N}_{k,\beta}(\Phi)]^k<\infty.$ Now,
	by the Burkholder--Davis--Gundy inequality and the second inequality in
	\eqref{thespot},
	\begin{align*}
		\E\left( |(p^{(\nu)}\circledast\text{\rm Re\,}\Phi)(t\,,x)|^k\right)
			&\le \left\{ 4\rho(0)k\int_0^t\sup_{a\in\R}
			\|\Phi(s\,,a)\|_k^2\,\d s\right\}^{k/2}\\
		&\le (4\rho(0)kt)^{k/2}\sup_{a\in\R}\sup_{s\in(0,t)}
			\E\left(|\Phi(s\,,a)|^k\right).
	\end{align*}
	The same quantity bounds the $k$th moment of $(p\circledast\text{\rm Im\,}\Phi)(t\,,x)$.
	The lemma follows readily from these observations.
\end{proof}

\section{Proof of Theorem \ref{th:U:exist}}\label{sec:U:Exist:unique}

In order to prove Theorem \ref{th:U:exist}, it is convenient to first 
define a new random field $u$ via
\begin{equation}\label{eq:Uu}
	u(t\,,x\,,\xi) = \e^{\nu_2\xi^2 t} U(t\,,x\,,\xi)
	\qquad[t\ge0,\, x,\xi\in\R],
\end{equation}
and note that if $U[\xi]$ is a mild solution to \eqref{PAM} for every $\xi\in\R$,
then $u[\xi]$ would have to be a mild solution to the following stochastic PDE
for every $\xi\in\R$:
\begin{equation}\label{PAM}
	\partial_t u(t\,,x\,,\xi) = 
	\nu_1 \partial^2_x u(t\,,x\,,\xi)
	+ i\xi u(t\,,x\,,\xi) \, V(t\,,x);
\end{equation}
subject to $u_0(x\,,\xi) = U_0(x\,,\xi)$ for every $x\in\R$.
That is, for every $t>0$ and $x,\xi\in\R$,
\begin{equation}\label{mild:u}
	u(t\,,x\,,\xi) = \int_{-\infty}^\infty p_t^{(\nu_1)}(x-x')u_0(x',\xi)\,\d x'
	+ i\xi\left( p^{(\nu_1)}\circledast u[\xi]\right)(t\,,x)
	\qquad\text{a.s.,}
\end{equation}
where ``$\circledast$'' denotes the stochastic convolution operator; see \eqref{eq:SC}.

One can also understand \eqref{PAM} as a system of two coupled,
real-valued SPDEs. Indeed, let
$\mathcal{X} := \text{\rm Re}\, u$ and $\mathcal{Y}:=\text{\rm Im}\, u$,
in order to see that, for every $\xi\in\R$,
the pair $(\mathcal{X}[\xi]\,,\mathcal{Y}[\xi])$ solves the SPDE
\begin{align*}
	\partial_t \mathcal{X}(t\,,x\,,\xi)&= 
		\nu_1 \partial^2_x \mathcal{X}(t\,,x\,,\xi) 
		-\xi \mathcal{Y}(t\,,x\,,\xi) \,V(t\,,x),\text{ \ and}\\
	\partial_t \mathcal{Y}(t\,,x\,,\xi)&= 
		\nu_1 \partial^2_x \mathcal{Y}(t\,,x\,,\xi)
		+ \xi \mathcal{X}(t\,,x\,,\xi)\, V(t\,,x),
\end{align*}
on $(0\,,\infty)\times\R$, subject to the following initial condition[s]:
For all $x,\xi\in\R$,
\[
	\mathcal{X}(0\,,x\,,\xi) = \text{\rm Re}\,\wh{\theta}_0(x\,,\xi)
	\quad\text{and}\quad
	\mathcal{Y}(0\,,x\,,\xi) = \text{\rm Im}\,\wh{\theta}_0(x\,,\xi).
\]

We plan to prove the following equivalent formulation of
Theorem \ref{th:U:exist}.

\begin{theorem}\label{th:u:exist}
	Suppose $u_0:\Omega\times\R^2\to\C$ is a measurable random field
	that is independent of $V$ and satisfies
	$\sup_{x\in\R}\E(|u_0(x\,,\xi)|^k)<\infty$ for every $k\geq 2$
	and  $\xi\in\R$.
	Choose and fix some $\nu_1>0$. Then,  for every $\xi\in\R$,
	\eqref{PAM} has a mild solution
	$u[\xi]$ that satisfies the following for every $k\ge2$, $\varepsilon\in(0\,,1)$,
	$t>0$, and $x,\xi\in\R$:
	\[
		\sup_{x\in\R}\E\left( |u(t\,,x\,,\xi)|^k\right) \le 
		\varepsilon^{-k} \exp\left( \frac{kc_k\rho(0)\xi^2}{2(1-\varepsilon)^2}t\right)
		\sup_{x\in\R}\E\left( |u_0(x\,,\xi)|^k\right),
	\]
	where $c_k$ was defined in \eqref{c_k}. Furthermore,
	suppose $v[\xi]$ is a mild solution \eqref{PAM} for every $\xi\in\R$, for
	$k=2$ and some $\varepsilon\in(0\,,1)$.
	Then, $v$ is a modification of $u$.
\end{theorem}

With Theorem \ref{th:u:exist} in mind,
let us begin with a standard Picard iteration argument. We first define
\[
	u_0(t\,,x\,,\xi) := u_0(x\,,\xi)\qquad
	\text{for all $t\ge0$ and $x,\xi\in\R$.}
\]
Then, we define iteratively for all $n\ge0$,
\begin{equation}\label{Picard:u}
	u_{n+1}(t\,,x\,,\xi) 
	= \int_{-\infty}^\infty p_t^{(\nu_1)}(x-x')u_0(x',\xi)\,\d x'
	+ i\xi\left( p^{(\nu_1)}\circledast u_n[\xi]\right)(t\,,x),
\end{equation}
where ``$\circledast$'' denotes stochastic convolution; see
\eqref{eq:SC}.
The preceding is well defined provided that the final Walsh integral 
is well defined; see Definition \ref{def:good:integrand}. 
That is, if $u_n[\xi]$ is a predictable random field 
for every $\xi\in\R$, and satisfies
\[
	\int_0^t\d s\int_{-\infty}^\infty\d x'\int_{-\infty}^\infty\d x''\
	p_{t-s}^{(\nu_1)}(x-x') p_{t-s}^{(\nu_1)}(x-x'')\E\left( |u_n(s\,,x',\xi)
	u_n(s\,,x'',\xi)|\right)\rho(x'-x'')<\infty,
\]
for every $t>0$ and $x,\xi\in\R$. 
The following lemma will ensure that this is the case.

\begin{lemma}\label{norm:iterate}
	Assume the hypotheses of Theorem \ref{th:u:exist} are met.
	Suppose also that there exists an integer $n\ge 0$ such that
	$u_n[\xi]$ is a predictable random field. Then,
	for every $\xi\in\R$, $u_{n+1}[\xi]$ is a predictable,
	two-parameter random field. Moreover, 
	for all $\xi\in\R$, $\varepsilon\in(0\,,1)$, and $k\ge2$,
	\begin{equation}\label{N:u_n+1}
		\mathcal{N}_{k,\beta_*}\left( u_{n+1}[\xi]\right)
		\le \varepsilon^{-1}\sup_{a\in\R}\|u_0(a\,,\xi)\|_k,
	\end{equation}
	where
	\begin{equation}\label{beta}
		\beta_* := \beta_*(k\,,\xi\,,\rho(0)\,,\varepsilon) = 
		\frac{c_k\rho(0)\xi^2}{2(1-\varepsilon)^2},
	\end{equation}
	and where $c_k$ was defined in \eqref{c_k}.
\end{lemma}

\begin{proof}
	The predictability of $u_{n+1}[\xi]$ follows from the predictability
	of $\text{\rm Re}\,u_{n+1}[\xi]$ and $\text{\rm Im}\,u_{n+1}[\xi]$,
	which in turn follows from a standard fact from
	stochastic analysis; see, for example \cite{Walsh}. We  verify 
	\eqref{N:u_n+1}, which is the main message of Lemma \ref{norm:iterate}.
	
	First of all,
	\begin{equation}\label{u_0}
		\left\| \int_{-\infty}^\infty p^{(\nu_1)}_t(x-x')u_0(x',\xi) \,\d x'\right\|_k
		\le \sup_{a\in\R}\|u_0(a\,,\xi)\|_k =
		\mathcal{N}_{k,\beta}\left( u_0[\xi]\right),
	\end{equation}
	uniformly for all $t,\beta>0$ and $x,\xi\in\R$. This bounds the
	first term on the right-hand side of \eqref{Picard:u}. We now
	estimate the second term, using Lemma \ref{lem:Young}, as follows:
	\[
		\mathcal{N}_{k,\beta}\left(i\xi p^{(\nu_1)}\circledast u_n[\xi]\right) =|\xi|
		\mathcal{N}_{k,\beta}\left( p^{(\nu_1)}\circledast u_n[\xi] \right)
		\le \sqrt{\frac{c_k\rho(0)\xi^2}{2\beta}}\,\mathcal{N}_{k,\beta}(u_n[\xi]).
	\]
	Therefore, \eqref{Picard:u} yields
	\[
		\mathcal{N}_{k,\beta}\left( u_{n+1}[\xi]\right) \le
		\mathcal{N}_{k,\beta}\left( u_0[\xi]\right) + 
		\sqrt{\frac{c_k\rho(0)\xi^2}{2\beta}}\,\mathcal{N}_{k,\beta}(u_n[\xi]),
	\]
	for every $n\ge0$ and $\xi\in\R$, and for every $\beta>0$. We now make
	the particular choice that $\beta= \beta_*$, where $\beta_*$ is given in \eqref{beta}:
	In this way we obtain the recursive inequality,
	\[
		\mathcal{N}_{k,\beta_*}\left( u_{n+1}[\xi]\right) \le
		\mathcal{N}_{k,\beta_*}\left( u_0[\xi]\right) + 
		(1-\varepsilon)\mathcal{N}_{k,\beta_*}(u_n[\xi]),
	\]
	valid for all $k\ge2$ and $n\ge0$. We iterate this inequality in order to find that
	\[
		\mathcal{N}_{k,\beta_*}\left( u_{n+1}[\xi]\right) \le
		\mathcal{N}_{k,\beta_*}\left( u_0[\xi]\right) \sum_{j=0}^{n+1}(1-\varepsilon)^j
		\le \varepsilon^{-1}\sup_{a\in\R}\|u_0(a\,,\xi)\|_k;
	\]
	see \eqref{u_0}. This is another way to state the lemma.
\end{proof}

Next we present two {\it a priori} regularity results;
see Lemmas \ref{un-un:x} and \ref{un-un:t}.
Both lemmas will be improved above later on. But, logically
speaking, we will need
the  {\it a priori} form of these lemmas first in order to 
establish the existence of a solution before we can use that solution
in order to establish our later, improved regularity results. This is
unfortunate, as it makes the proof of Theorem \ref{th:exist:U}
somewhat lengthy. But we do not know of another rational argument
that bypasses this lengthy procedure. Thus, we begin with an 
{\it a priori} regularity result in the space variable.

\begin{lemma}\label{un-un:x}
	Assume the hypotheses of Theorem \ref{th:u:exist} are met.
	Suppose also that there exists an integer $n\ge 0$ such that
	$u_n[\xi]$ is a predictable random field for every $\xi\in\R$.  
	Then, for every real number $k\ge2$,
	$t>0$ and $x,z,\xi\in\R$,
	\begin{align*}
		&\E\left( \left|  u_{n+1}(t\,,x\,,\xi) - u_{n+1}(t\,,z\,,\xi)\right|^k\right)\\
		&\qquad\le  {2^k}\sup_{a\in\R}\E\left( |u_0(a\,,\xi)|^k\right)
			\min\left( {2^k} ~,~\frac{|x-z|^k}{(\pi\nu_1 t)^{k/2}}\right)\\
		&\hskip.6in + {2^{4k}}\left(\frac{\rho(0)k \xi^2}{\pi\nu_1}\right)^{k/2}\,
			\sup_{a\in\R}\E\left( |u_0(a\,,\xi)|^k\right)\,
			\e^{{16\rho(0) k^2 \xi^2 t}}\,
			|x-z|^k\left[ \log_+\left(\frac{4\pi\nu_1 t}{|x-z|^2}\right) \right]^{k/2}.
	\end{align*}
\end{lemma}

\begin{proof}
	Choose and fix an integer $n\ge0$
	and real numbers $k\ge 2$, $t>0$, $\xi\in\R$, and $x,z\in\R$ that satisfy
	$|x-z|\le1$.
	Thanks to \eqref{Picard:u} we can write
	\[
		\left\| u_{n+1}(t\,,x\,,\xi) - u_{n+1}(t\,,z,\xi)\right\|_k
		\le T_1 + T_2,
	\]
	where
	\begin{equation}\label{T1T2}\begin{split}
		T_1 &:= \int_{-\infty}^\infty \left| p_t^{(\nu_1)}(x-x') -  p_t^{(\nu_1)}(z-x') 
			\right| \left\| u_0(x',\xi)\right\|_k\,\d x',\\
		T_2 &:=|\xi|\cdot\left\| \left( p^{(\nu_1)}\circledast u_n[\xi]\right)(t\,,x) -
			\left( p^{(\nu_1)}\circledast u_n[\xi]\right)(t\,,z)\right\|_k.
	\end{split}\end{equation}
	
	According to Lemma 6.4 of Joseph et al \cite{CJKS} (for the explicit constant mentioned
	below see the bound for $\mu_1(|x|)$ in {the proof of Lemma 6.4 in \cite{CJKS}}, all the time remembering that their
	constant $\varkappa/2$ is  $\nu_1$ in the present setting),
	\begin{equation}\label{pre:T1}\begin{split}
		T_1 &\le\sup_{a\in\R}\| u_0(a\,,\xi)\|_k
			\int_{-\infty}^\infty \left| p_t^{(\nu_1)}(x-x') -  p_t^{(\nu_1)}(z-x') \right| \d x'\\
		&\le \frac{|x-z|}{\sqrt{\pi\nu_1 t}}\,\sup_{a\in\R}\|u_0(a\,,\xi)\|_k.
	\end{split}\end{equation}
	Also, a trivial bound yields
	\[
		T_1 \le \sup_{a\in\R}\|u_0(a\,,\xi)\|_k\int_{-\infty}^\infty
		\left| p_t^{(\nu_1)}(x-x') + p_t^{(\nu_1)}(z-x')\right|\d x' 
		\le 2\sup_{a\in\R}\|u_0(a\,,\xi)\|_k.
	\]
	Therefore,
	\begin{equation}\label{T1}
		T_1 \le {\min\left( 2 ~,~\frac{|x-z|}{\sqrt{\pi\nu_1 t}}\right)}
		\cdot \sup_{a\in\R}\|u_0(a\,,\xi)\|_k.
	\end{equation}
	See also Lemma 6.4 of \cite{CJKS}.
	
	When $\xi=0$, we have $T_2=0$ and \eqref{T1} completes the proof in that case.
	Now consider the case that $\xi\neq0$.
	
	We may observe that
	\begin{equation}\label{p-p:x}\begin{split}
		&\left\| \left( p^{(\nu_1)}\circledast u_n[\xi]\right)(t\,,x) -
			\left( p^{(\nu_1)}\circledast u_n[\xi]\right)(t\,,z)\right\|_k^k\\
		&\hskip.5in= \E\left( \left| \int_{(0,t)\times\R}
			\left\{ p_{t-s}^{(\nu_1)}(x-x')-p_{t-s}^{(\nu_1)}(z-x')\right\} 
			u_n(s\,,x',\xi) V(s\,,x')\,\d s\,\d x'
			\right|^k\right).
	\end{split}\end{equation}
	As before, we consider $\text{\rm Re\,}u_n$ and $\text{\rm Im\,}u_n$
	separately, using the Burkholder--Davis--Gundy inequality. Let us fix 
	$n\ge0$, $t>0$, and $x,z,\xi\in\R$,
	and write
	\[
		\mathcal{D}(s\,,x') := p_{t-s}^{(\nu_1)}(x-x')-p_{t-s}^{(\nu_1)}(z-x')\
		\quad\text{and}\quad
		\mathcal{R}(s\,,x') := \text{\rm Re\,}u_n(s\,,x',\xi),
	\]
	for all $s\in(0\,,t)$ and $x'\in\R$. We respectively define $T_{21}$ and $T_{22}$ 
	to be the same expressions as $T_2$, but with $u_n[\xi]$ replaced by 
	$\text{\rm Re}\,  u_n[\xi]$ and $\text{\rm Im}\,  u_n[\xi]$. 
	Then, we use similar ideas as those that were used in the proof of Lemma \ref{lem:Young}
	in order to see that
	\begin{align*}
		{T_{21}^k}&=|\xi|^k\E\left(\left|\int_{(0,t)\times\R}
			\mathcal{D}(s\,,x')\mathcal{R}(s\,,x') V(s\,,x')\,
			\d s\,\d x'\right|^k\right)\\
		&\le (4\rho(0)k\xi^2)^{k/2}\E\left( \left[ \int_0^t\d s
			\int_{-\infty}^\infty\d x'\int_{-\infty}^\infty\d x''
			\left|\mathcal{D}(s\,,x')\mathcal{D}(s\,,x'')
			\mathcal{R}(s\,,x')\mathcal{R}(s\,,x'')
			\right|\right]^{k/2}\right).
	\end{align*}
	In particular,
	\[
		{T_{21}^k}\le(4\rho(0)k\xi^2)^{k/2}\left\| \int_0^t\left\{
		\int_{-\infty}^\infty
		\left|\mathcal{D}(s\,,x') \mathcal{R}(s\,,x')\right|
		\d x'\right\}^2\,\d s\right\|_{k/2}^{k/2},
	\]
	thanks to the triangle inequality.
	Two back-to-back applications of Minkowski's inequality 
	now imply that
	\begin{align*}
		\left\|\int_{(0,t)\times\R}\mathcal{D}(s\,,x')\mathcal{R}(s\,,\xi) V(s\,,x')\,
			\d s\,\d x'\right\|_k^2
		&\le 4\rho(0)k \int_0^t\left\|\left\{
			\int_{-\infty}^\infty
			\left|\mathcal{D}(s\,,x')\mathcal{R}(s\,,x')\right| 
			\d x'\right\}^2\right\|_{k/2}\d s\\
		&= 4\rho(0)k \int_0^t\left\|
			\int_{-\infty}^\infty
			\left|\mathcal{D}(s\,,x')\mathcal{R}(s\,,x')\right| 
			\d x' \right\|_k^2\,\d s\\
		&\le 4\rho(0)k \int_0^t\left[
			\int_{-\infty}^\infty
			\left|\mathcal{D}(s\,,x')\right|  \left\| \mathcal{R}(s\,,x') \right\|_k\,
			\d x' \right]^2\,\d s.
	\end{align*}
	In particular,
	\begin{align*}
		&\left\|\int_{(0,t)\times\R}\mathcal{D}(s\,,x')\mathcal{R}(s\,,\xi) V(s\,,x')\,
			\d s\,\d x'\right\|_k^2\\
		&\hskip1in\le 4\rho(0)k\int_0^t\sup_{a\in\R}\|\mathcal{R}(s\,,a)\|_k^2
			\left[\int_{-\infty}^\infty \left| p_{t-s}^{(\nu_1)}(x-x')-p_{t-s}^{(\nu_1)}(z-x')\right|
			\d x' \right]^2\,\d s\\
		&\hskip1in\le 4\rho(0)k\int_0^t\sup_{a\in\R}\|u_n(s\,,a\,,\xi)\|_k^2
			\left[\frac{|x-z|}{\sqrt{\pi\nu_1(t-s)}}\wedge 2\right]^2\,\d s;
	\end{align*}
	see \eqref{pre:T1}. The same bound holds if we replace $\mathcal{R}(s\,,x')=
	\text{\rm Re\,}u_n(s\,,x',\xi)$  by
	$\text{\rm Im\,}u_n(s\,,x',\xi)$. Therefore, this and \eqref{p-p:x}
	together yield
	\[
		T_2^2\le {16\rho(0)}k\xi^2\int_0^t\sup_{a\in\R}\|u_n(s\,,a\,,\xi)\|_k^2
		\left[\frac{|x-z|}{\sqrt{\pi\nu_1(t-s)}}\wedge 2\right]^2\,\d s.
	\]
	Thus, we learn from Theorem \ref{th:u:exist} [with $\varepsilon=\frac12$, say] that
	\begin{align*}
		T_2^2&\le {2^{4}}\rho(0)k \xi^2
			\sup_{a\in\R}\|u_0(a\,,\xi)\|_k^2\cdot
			\int_0^t \e^{32\rho(0) k \xi^2 s}
			\left[\frac{|x-z|}{\sqrt{\pi\nu_1(t-s)}}\wedge 2\right]^2\,\d s\\
		&\le {2^{4}}\rho(0)k \xi^2\,
			\sup_{a\in\R}\|u_0(a\,,\xi)\|_k^2\,
			\e^{32\rho(0) k \xi^2 t}\int_0^t 
			\left[\frac{|x-z|^2}{\pi\nu_1 s}\wedge 4\right]\,\d s.
	\end{align*}
	For every $a>0$,
	\[
		\int_0^t\left( \frac as\wedge4\right)\d s = a + a\ln\left( \frac{4t}{a}\vee 1\right)\le
		2a\log_+\left( \frac{4t}{a}\right).
	\]
	Therefore,
	\begin{equation}\label{T2}
		T_2\le {2^{3}\sqrt{\frac{2\rho(0)k \xi^2}{\pi\nu_1}}}\,
		\sup_{a\in\R}\|u_0(a\,,\xi)\|_k\,
		\e^{16\rho(0) k \xi^2 t}\,
		|x-z|\sqrt{\log_+\left( \frac{4\pi\nu_1 t}{|x-z|^2}\right)}.
	\end{equation}
	Since
	\[
		\E\left( |u_{n+1}(t\,,x\,,\xi) - u_{n+1}(t\,,z,\xi)|^k\right) \le
		2^k(T_1^k+T_2^k),
	\]
	we can deduce the lemma from \eqref{T1} and \eqref{T2}.
\end{proof}

The following is an {\it a priori} regularity result in the time variable,
and matches the result of the spatial Lemma \ref{un-un:x}.

\begin{lemma}\label{un-un:t}
	Assume the hypotheses of Theorem \ref{th:u:exist} are met.
	Suppose also that there exists an integer $n\ge 0$ such that
	$u_n[\xi]$ is a predictable random field for every $\xi\in\R$, and that
	\eqref{N:u_n+1} holds for all $k\ge 2$,
	except with $u_{n+1}$ there replaced by $u_n$ here.  Finally, suppose that
	there exists $\alpha>0$ such that for every $k\ge 2$ there
	exists a real number $M_k$ such that
	\begin{equation}\label{u_0:1}
		\sup_{\xi\in\R}\E\left( \left|u_0(a\,,\xi) - u_0(b\,,\xi) \right|^k\right)
		\le M_k|a-b|^{k\alpha},
	\end{equation}
	for all $a,b\in\R$.
	Then,  for every real number $k\ge2$, $t,h>0$, 
	and $x,\xi\in\R$,
	\begin{align*}
		&\E\left( \left| u_{n+1}(t+h\,,x\,,\xi) - u_{n+1}(t\,,x\,,\xi)\right|^k\right)\\
		&\hskip1in \le 3^k{\frac{M_k}{\sqrt{\pi}}} \Gamma\left( \frac{k\alpha+1}{2}\right)
			(4\nu_1)^{k\alpha/2} h^{k\alpha/2} \\
		&\hskip2in + {2\times\left(2304\rho(0)\xi^2 k \right)^{k/2} }
			\e^{32\rho(0)k^2 \xi^2(t+h)}
			\sup_{a\in\R}\E\left( |u_0(a\,,\xi)|^k\right)
			h^{k/2}.
	\end{align*}
\end{lemma}

\begin{proof}
	In accord with \eqref{Picard:u}, we may write
	\[
		\left\| u_{n+1}(t+h\,,x\,,\xi) - u_{n+1}(t\,,x\,,\xi) \right\|_k \le T_1 + T_2 + T_3,
	\]
	where
	\begin{align*}
		T_1 &:=\left\| \int_{-\infty}^\infty \left[p_{t+h}^{(\nu_1)}(x-x')-p_t^{(\nu_1)}(x-x')\right]
			u_0(x',\xi) \,\d x'\right\|_k,\\
		T_2 &:=|\xi|\left\| \int_{(t,t+h)\times\R} p_{t+h-s}^{(\nu_1)}(x-x') 
			u_n(s\,,x',\xi)V(s\,,x')\,\d s\,\d x'\right\|_k,\\
		T_3 &:= |\xi|\left\| \int_{(0,t)\times\R}\left\{ p_{t+h-s}^{(\nu_1)}(x-x') - p_{t-s}^{(\nu_1)}(x-x')\right\}
			u_n(s\,,x',\xi)V(s\,,x')\,\d s\,\d x'\right\|_k.
	\end{align*}
	We will estimate these in turn. 
	
	In order to estimate $T_1$, first let $B:=\{B_t\}_{t\ge0}$ denote a Brownian
	motion, run at speed $2\nu_1$ so that $p_t^{\nu_1}$ is the probability density function of
	$B_t$  for every $t>0$. By the conditional form of Jensen's inequality,
	\[
		T_1^k = \left\| {\E\left[\left. u_0(x+B_{t+h}\,,\xi) - u_0( x+B_t\,,\xi)\
		\right|\ u_0(x\,, \xi)\right]} \right\|_k^k
		\le \E\left( \left| u_0(B_{t+h}\,,\xi) - u_0( B_t\,,\xi) \right|^k\right).
	\]
	Because \eqref{u_0:1} ensures that
	\[
		\E\left( \left.\left| u_0(B_{t+h}\,,\xi) - u_0( B_t\,,\xi) \right|^k\
		\right|\ B\right) \le M_k \left| B_{t+h}-B_t\right|^{k\alpha},
	\]
	the tower property of conditional expectations yields
	\begin{equation}\label{T1:1}
		T_1 \le \left\{ M_k\E\left( \left| B_{t+h}-B_t\right|^{k\alpha}\right) \right\}^{1/k}
		= \left[ \frac{M_k}{\sqrt{\pi}} \Gamma\left( \frac{k\alpha+1}{2}\right)\right]^{1/k} 
		(4\nu_1 h)^{\alpha/2}.
	\end{equation}
	If $\xi=0$, then $T_2=T_3=0$ and the lemma is proved in that case.
	From now on we consider the case that $\xi\neq0$, and proceed to
	estimate $T_2$ and $T_3$ in this order.
	
	We estimate $T_2$ by following a similar reasoning as was done in the
	proof of Lemma \ref{norm:iterate}. 
	Namely, we first appeal to the Burkholder--Davis--Gundy
	inequality (as was done surrounding \eqref{BDG:Re:Im}) and \eqref{rho:max} to see that 
	\begin{align*}
		T_2^k &\le ({16}k\rho(0))^{k/2}|\xi|^k \times\\
		&\hskip-0.5cm\times \E\left( \left[\int_t^{t+h}\d s
			\int_{-\infty}^\infty\d x'\int_{-\infty}^\infty\d x''\
			p_{t+h-s}^{(\nu_1)}(x-x')p_{t+h-s}^{(\nu_1)}(x-x'')
			\left| u_n(s\,,x',\xi) u_n(s\,,x'',\xi)\right|
			\right]^{k/2}\right).
	\end{align*}
	The factor {$(16k)^{k/2}$} is put in place of the usual $(4k)^{k/2}$
	to account for two appeals to the BDG inequality: One for the real
	part and one for the imaginary part, {and also the inequality of the type 
	$\|\Phi\|_k^2= \|\text{\rm Re}\, \Phi + i 
	\text{\rm Im}\,\Phi\|_k^2\leq 2 (\|\text{\rm Re}\,\Phi\|_k^2+ \|\text{\rm Im}\,\Phi\|_k^2)$. } 
	In any case, the preceding yields
	\begin{align*}
		T_2^k&\le ({16}k\rho(0))^{k/2}|\xi|^k\E\left( \left[\int_t^{t+h}
			\left\{ \int_{-\infty}^\infty
			p_{t+h-s}^{(\nu_1)}(x-x')|u_n(s\,,x',\xi)|\,\d x'\right\}^2\d s\right]^{k/2}\right)\\
		&\le ({16}k\rho(0))^{k/2}|\xi|^k 
			\left\{ \int_t^{t+h} \left[\int_{-\infty}^\infty
			p_{t+h-s}^{(\nu_1)}(x-x') \|u_n(s\,,x',\xi) \|_k\,\d x' \right]^2 \d s\right\}^{k/2}\\
		&\le ({16}k\rho(0)h)^{k/2}|\xi|^k \sup_{a\in\R}\sup_{\tau\in(t,t+h)}
			\E\left( |u_n(\tau\,,a\,,\xi)|^k\right).
	\end{align*}
	Therefore, the definition \eqref{N} of the norm $\mathcal{N}_{k,\beta_*}$,
	and the definition \eqref{beta} of the constant $\beta_*$ together yield
	\[
		T_2^k\le ({16}k\rho(0)h)^{k/2}|\xi|^k \e^{(t+h)\beta_*k}\left[
			\mathcal{N}_{k,\beta_*}\left( u_n[\xi]\right)\right]^k.
	\]
	Consequently, we may appeal to \eqref{N:u_n+1}---with $u_{n+1}$ there replaced
	by $u_n$ here---in order to deduce the following:
	\begin{equation}\label{T2:1}
		T_2^k
		\le \left(\frac{{16}k\rho(0)h}{\varepsilon^2}\right)^{k/2}|\xi|^k \e^{(t+h)\beta_*k}
		\sup_{a\in\R}\E\left( |u_0(a\,,\xi)|^k\right).
	\end{equation}
	
	Finally, we estimate $T_3$. Define
	\[
		\mathcal{D}(s\,,x') := p_{t+h-s}^{(\nu_1)}(x-x') - p_{t-s}^{(\nu_1)}(x-x')
		\qquad\text{for all $s>0$ and $x'\in\R$,}
	\]
	where we are viewing $t$, $h$, and $x$ as fixed numbers to simplify the notation
	in the ensuing calculation:
	\begin{align*}
		T_3^k &\le({16}k\rho(0))^{k/2}|\xi|^k\E\left( \left[\int_0^t
			\left\{ \int_{-\infty}^\infty
			\left|\mathcal{D}(s\,,x')u_n(s\,,x',\xi) \right|\d x'\right\}^2\d s\right]^{k/2}\right)\\
		&\le ({16}k\rho(0))^{k/2}|\xi|^k \sup_{a\in\R}\sup_{\tau\in(t,t+h)}
			\E\left( |u_n(\tau\,,a\,,\xi)|^k\right)\left\{ \int_0^t \left[\int_{-\infty}^\infty
			|\mathcal{D}(s\,,x')| \,\d x' \right]^2 \d s\right\}^{k/2}.
	\end{align*}
	 Then, by arguing as before we find that
	\begin{align*}
		T_3^k 
			&\le ({16}k\rho(0))^{k/2}|\xi|^k \e^{(t+h)\beta_*k}
			\left[\mathcal{N}_{k,\beta_*}\left(u_n[\xi]\right)\right]^k
			\left\{ \int_0^t \left[\int_{-\infty}^\infty
			|\mathcal{D}(s\,,x')| \,\d x' \right]^2 \d s\right\}^{k/2}\\
		&\le\left(\frac{{16}k\rho(0)}{\varepsilon^2}\right)^{k/2}|\xi|^k \e^{(t+h)\beta_*k}
			\sup_{a\in\R}\E\left( |u_0(a\,,\xi)|^k\right)\left\{ \int_0^t \left[\int_{-\infty}^\infty
			|\mathcal{D}(s\,,x')| \,\d x' \right]^2 \d s\right\}^{k/2}.
	\end{align*}
	The final object [under $\{\,\cdots\}^{k/2}$] involves a real-variable integral
	that is easy to estimate directly,  as follows:
	\begin{align*}
		\int_0^t\left[\int_{-\infty}^\infty
			|\mathcal{D}(s\,,x')| \,\d x' \right]^2 \d s
			&=\int_0^t\left[\int_{-\infty}^\infty
			\left| p_{s+h}^{\nu_1}(x') - p_s^{\nu_1}(x')\right|\d x' \right]^2 \d s\\
		&\le\int_0^t\d s\left[2\wedge \int_{-\infty}^\infty\d w 
			\int_s^{s+h}\d r\ \left| \frac{\partial p_r^{\nu_1}(w)}{\partial r}\right|\right]^2\\
		&=\int_0^t\d s\left[2\wedge \int_{-\infty}^\infty\d w
			\int_s^{s+h}\d r\ \left| \frac{w^2}{4\nu_1 r^2}-\frac{1}{2r}\right|
			p_r^{\nu_1}(w)\right]^2.
	\end{align*}
	Apply the triangle inequality, $|w^2/(4\nu_1 r^2) - 1/(2r)|\le
	w^2/(4\nu_1 r^2)+(2r)^{-1}$. We integrate the $\d w$-integral first in order to see that
	\begin{align*}
		\int_0^t\left[\int_{-\infty}^\infty
			|\mathcal{D}(s\,,x')| \,\d x' \right]^2 \d s
			&\le \int_0^t\left[2\wedge
			\int_s^{s+h}\frac{\d r}{r}\right]^2\d s\\
		&\le\int_0^t\left[2\wedge \frac hs \right]^2\,\d s\\
		&\le 4h,
	\end{align*}
	as can be seen by examining the integral according to whether or not $t<h/2$.
	Thus, we obtain the following:
	\begin{equation}\label{T3:1}
		T_3^k \le \left(\frac{{64}hk\rho(0)}{\varepsilon^2}\right)^{k/2}
		|\xi|^k \e^{(t+h)\beta_*k}
		\sup_{a\in\R}\E\left( |u_0(a\,,\xi)|^k\right).
	\end{equation}
	Since
	\[
		\E\left( \left| u_{n+1}(t+h\,,x\,,\xi) - u_{n+1}(t\,,x\,,\xi) \right|^k\right)
		\le 3^kT_1^k + 3^kT_2^k + 3^kT_3^k,
	\]
	we may combine \eqref{T1:1}, \eqref{T2:1}, and \eqref{T3:1},
	and set $\varepsilon:=1/2$ [to be concrete] 
	to finish. 
\end{proof}

The preceding lemmas will play a role in establishing the following
regularity result.

\begin{theorem}\label{th:reg:u}
	Assume the hypotheses of Theorem \ref{th:u:exist} are met.
	Assume also that there exist $\alpha,\gamma\in(0\,,1]$
	such that for every $k\ge 2$ there exists
	a real number $A_k$ such that
	\begin{equation}\label{u_0:1.5}
		\E\left( \left|u_0(x\,,\xi) - u_0(x',\xi') \right|^k\right)
		\le A_k\left( |x-x'|^{k\alpha} + |\xi-\xi'|^{k\gamma}\right),
	\end{equation}
	uniformly for all $x,x',\xi,\xi'\in\R$. Then,
	the solution $(t\,,x\,,\xi)\mapsto u(t\,,x\,,\xi)$ to \eqref{PAM} has a
	version that is H\"older continuous on $\R_+\times\R^2$. 
\end{theorem}

\begin{remark}
	The proof of Theorem \ref{th:reg:u} shows that, in fact, 
	the 3-parameter stochastic process
	$(t\,,x\,,\xi)\mapsto u(t\,,x\,,\xi)$ has a version
	that is H\"older continuous with respective indices 
	$\frac12(\alpha\wedge 1)-\varepsilon$ (in the variable $t$), $(\alpha\wedge1)-\varepsilon$
	(in the variable $x$), and $\gamma-\varepsilon$
	(in the variable $\xi$)---for every fixed $\varepsilon>0$---uniformly on compact subsets
	of $\R_+\times\R^2.$
\end{remark}

Theorem \ref{th:reg:u} will follow immediately from Lemmas 
\ref{u-u:x}, \ref{u-u:t}, and \ref{u-u:xi} below, after an appeal
to a suitable form of the Kolmogorov continuity theorem (see \cite{minicourse},
for an example). Therefore, we will not write a proof for Theorem
\ref{th:reg:u}. Instead we merely state and prove the following
three auxilliary lemmas.

\begin{lemma}\label{u-u:x}
	Assume the hypotheses of Theorem \ref{th:u:exist} are met, and that
	\eqref{u_0:1.5} holds.
	Then, for every real number $k\ge2$,
	$t>0$, and $x,z,\xi\in\R$,
	\begin{align*}
		&\E\left( \left|  u(t\,,x\,,\xi) - u(t\,,z\,,\xi)\right|^k\right)\\
		&\qquad \le 2^k A_k|x-z|^{k\alpha}\\
		&\hskip.6in + {16^k\left(\frac{2\rho(0)k \xi^2}{\pi\nu_1}\right)^{k/2}}\,
			\sup_{a\in\R}\E\left( |u_0(a\,,\xi)|^k\right)\,
			\e^{16\rho(0) k^2 \xi^2 t}
			\, |x-z|^k\left[ \log_+\left(\frac{4\pi\nu t}{|x-z|^2}\right) \right]^{k/2}.
	\end{align*}
\end{lemma}

\begin{proof}
	We know from the proof of Lemma \ref{un-un:x}  that
	for all $n\in\Z_+$,
	\[
		\E\left( \left|  u_{n+1}(t\,,x\,,\xi) - u_{n+1}(t\,,z\,,\xi)\right|^k\right)
		\le 2^k \wt{T}_1^k + 2^k T_2^k,
	\]
	where
	\[
		\wt{T}_1 = \left\| \int_{-\infty}^\infty p^{(\nu_1)}_t(x') \left\{ u_0(x'-x\,,\xi)
		-u_0(x'-z\,,\xi)\right\}\d x'
		\right\|_k \le A_k^{1/k}|x-z|^\alpha,
	\]
	by \eqref{u_0:1.5} and Minkowski's inequality, and
	$T_2$ was defined in \eqref{T1T2}.
	Estimate $T_2$ by \eqref{T2}, then
	let $n\to\infty$ and use the fact---see the proof of Theorem \ref{th:u:exist}---%
	that $\lim_{n\to\infty}u_n(a\,,b\,,c)= u(a\,,b\,,c)$
	in $L^k(\Omega)$ for all $(a\,,b\,,c)\in\R_+\times\R^2$ 
	in order to complete the proof.
\end{proof}

Similarly, one can let $n\to\infty$ in Lemma \ref{un-un:t},
and appeal to Fatou's lemma in order to deduce the following 
inequality. It might help to
also compare \eqref{u_0:1} and \eqref{u_0:1.5} to see that
$M_k\le A_k.$

\begin{lemma}\label{u-u:t}
	Assume the hypotheses of Theorem \ref{th:u:exist} are met,
	and  that \eqref{u_0:1.5} holds.
	Then,  for every real number $k\ge2$, $t,h>0$, 
	and $x,\xi\in\R$,
	\begin{align*}
		&\E\left( \left| u(t+h\,,x\,,\xi) - u(t\,,x\,,\xi)\right|^k\right)\\
		& \le 3^k{\frac{A_k}{\sqrt{\pi}}} \Gamma\left( \frac{k\alpha+1}{2}\right)
			(4\nu_1)^{k\alpha/2} h^{k\alpha/2} 
			+ {2\times \left(2304\rho(0)\xi^2 k \right)^{k/2}} 
			\e^{32\rho(0)k^2 \xi^2(t+h)}
			\sup_{a\in\R}\E\left( |u_0(a\,,\xi)|^k\right)
			h^{k/2}.
	\end{align*}
\end{lemma}

The following addresses the same sort of estimate that Lemma
\ref{u-u:t} does, but now in the case that $t=0$. The reasoning is
slightly different, and so we include a proof. 

\begin{lemma}\label{u-u:t=0}
	Assume the hypotheses of Theorem \ref{th:u:exist} are met,
	and  that \eqref{u_0:1.5} holds.
	Then, for all $\varepsilon\in(0\,,1)$, $k\ge2$, $h\in[0\,,1]$, and $x,\xi\in\R$,
	\[
		\E\left( \left| u(h\,,x\,,\xi) - u_0(x\,,\xi)\right|^k\right)
		\le \wt{C} h^{k\alpha/2},
	\]
	where
	\begin{align*}
		\wt{C}  &=\wt{C}(k) := \frac{4^k A_k\nu_1^{k\alpha/2}\,
			[\Gamma((1+\alpha)/2)]^k}{\pi^{k/2}}\\
		&\hskip1in+ {\frac{1}{\varepsilon^k}}({64}\rho(0)\xi^2k)^{k/2}\,
			\exp\left( \frac{kc_k\rho(0)\xi^2 h}{2(1-\varepsilon)^2}\right)\,
			\sup_{x\in\R}\E\left(|u_0(x\,,\xi)|^k\right),
	\end{align*}
	for the same  constants $c_k$ and $A_k$ that appeared
	respectively  in \eqref{c_k} and in \eqref{u_0:1.5}.
\end{lemma}

\begin{proof}
	For every $k\ge 2$, $h>0$, and $x,\xi\in\R$, we may write 
	\[
		\E\left(| u(h\,,x\,,\xi) - u_0(x\,,\xi) |^k\right) \le 2^k(T_1^k + T_2^k),
	\]
	where
	\[
		T_1 := \left\| (p^{(\nu_1)}_h*u_0[\xi])(x) - u_0(x\,,\xi)\right\|_k
		\quad\text{and}\quad
		T_2 := |\xi| \cdot\left\| \left(p^{(\nu_1)}\circledast u[\xi]\right)(h\,,x) \right\|_k.
	\]
	Evidently,
	\begin{align*}
		T_1 &\le \int_{-\infty}^\infty 
			p^{(\nu_1)}_h(a)\left\| u_0(a-x\,,\xi)-u_0(x\,,\xi)\right\|_k\,\d a\\
		&\le A_k^{1/k}\int_{-\infty}^\infty |a|^\alpha p^{(\nu_1)}_h(a)\,\d a\hskip1in
			\text{[by \eqref{u_0:1.5}]}\\
		&= \frac{2^{\alpha}A_k^{1/k}\,\nu_1^{\alpha/2}\,
			\Gamma((1+\alpha)/2)}{\sqrt\pi}\, h^{\alpha/2}.
	\end{align*}
	Moreover, we may appeal first to Lemma \ref{lem:Young:redux} 
	and then to Theorem \ref{th:u:exist} 
	in order to see that
	\begin{align*}
		T_2 &\le (16\rho(0)\xi^2kh)^{1/2}\sup_{x\in\R}\sup_{s\le h}
			\| u(s\,,x\,,\xi)\|_k\\
		&\le (16\rho(0)\xi^2k)^{1/2}
			{\varepsilon^{-1}}\exp\left( \frac{c_k\rho(0)\xi^2h}{2(1-\varepsilon)^2}\right)
			\sup_{x\in\R}\|u_0(x\,,\xi)\|_k\, \sqrt h.
	\end{align*}
	Combine the estimates, using the fact that
	$h^{\alpha/2} \ge \sqrt h$ whenever $h\in[0\,,1]$.
\end{proof}

Finally, in the next lemma, we establish a regularity result in the auxilliary
variable $\xi$. We emphasize that the following result holds under exactly
the same conditions as does Theorem \ref{th:u:exist}, together with what 
\eqref{u_0:1.5}, which will turn out to be an innocuous condition on $u_0$.

\begin{lemma}\label{u-u:xi}
	Assume the hypotheses of Theorem \ref{th:u:exist} are met,
	and that \eqref{u_0:1.5} holds.
	Then, for all real numbers $k\ge2$ and $R,L>0$,
	\[
		S := 
		\sup_{\substack{|\xi|,|\xi'|\le L\\|\xi-\xi'|\le1}}\
		\sup_{t\in[0,R]}\
		\sup_{a\in\R}\E\left(\frac{%
		\left| u(t\,,a\,,\xi) - u(t\,,a\,,\xi') \right|^k}{|\xi-\xi'|^{k\gamma}} \right)
		<\infty.
	\]
\end{lemma}

\begin{remark}
	The proof in fact shows that
	\[
		S \le {192^{k/2}}\left[  A_k+ {(2\rho(0)kR )^{k/2}}\,
		\sup_{\substack{a\in\R\\ |b|\le L}}
		\E\left(|u_0(a\,,b)|^k\right) \right]
		\e^{30\rho(0)L^2k^2R};
	\]
	see \eqref{diff:u:xi} below.
\end{remark}

\begin{proof}
	By \eqref{mild:u},
	\[
		\left\| u(t\,,x\,,\xi) - u(t\,,x\,,\xi') \right\|_k
		\le T_1 + T_2 + T_3,
	\]
	where:
	\begin{align*}
		T_1 &:= \int_{-\infty}^\infty p^{(\nu_1)}_t(x-x')\left\|
			u_0(x',\xi) - u_0(x',\xi')\right\|_k\,\d x';\\
		T_2 &:= |\xi-\xi'| \cdot \left\| (p^{(\nu_1)}\circledast u[\xi])(t\,,x)\right\|_k;\\
		\intertext{and}
		T_3 &:= |\xi'| \left\| (p^{(\nu_1)}\circledast u[\xi])(t\,,x)
			- (p^{(\nu_1)}\circledast u[\xi'])(t\,,x)\right\|_k.
	\end{align*}
	We now estimate $T_1$, $T_2$, and $T_3$ in this order.
	
	Clearly,
	\begin{equation}\label{T1:2}
		T_1 \le \sup_{a\in\R}\left\| u_0(a\,,\xi) - u_0(a\,,\xi') \right\|_k
		\le A_k^{1/k}|\xi-\xi'|^\gamma.
	\end{equation}
	
	Thanks to Theorem \ref{th:u:exist} and Lemma \ref{lem:Young},
	$\mathcal{N}_{k,\beta}(u[\xi])<\infty$ for some $\beta$.
	Therefore, we may first apply Lemma \ref{lem:Young:redux}
	and then Theorem \ref{th:u:exist}---in this order--in order to see that
	for every $\varepsilon\in(0\,,1)$,
	\begin{equation}\label{T2:2}\begin{split}
		T_2 &\le (16\rho(0)kt )^{1/2}\sup_{x\in\R}\sup_{s\in(0,t)}
			\|u(s\,,x\,,\xi)\|_k
			\cdot |\xi-\xi'|\\
		&\le {\varepsilon^{-1}}(16\rho(0)kt )^{1/2}
			\exp\left( \frac{c_k\rho(0)\xi^2t}{2(1-\varepsilon)^2}\right)
			\sup_{x\in\R}
			\|u_0(x\,,\xi)\|_k\cdot
			|\xi-\xi'|.
	\end{split}\end{equation}
	
	Finally, the most interesting term $T_3$ we hold $t$, $\xi$, and $\xi'$ fixed and
	define 
	\[
		\mathcal{D}(s\,,x') := u(s\,,x',\xi)-u(s\,,x',\xi')
		\qquad\text{for all $s\in(0\,,t)$ and $x'\in\R$},
	\]
	in order to see that
	\begin{align}\label{T3:2}
		&T_3^k = |\xi'|^k\E\left(\left|
			\int_{(0,t)\times\R} p^{(\nu_1)}_{t-s}(x-x')\mathcal{D}(s\,,x')
			V(s\,,x')\,\d s\,\d x'\right|^k\right)\\\notag
		&\le (4\rho(0)k)^{k/2}|\xi'|^k\E\left( \left|
			\int_0^t\d s\int_{-\infty}^\infty\d x'\int_{-\infty}^\infty\d x''\
			p^{(\nu_1)}_{t-s}(x-x')p^{(\nu_1)}_{t-s}(x-x'')\left|\mathcal{D}(s\,,x')\mathcal{D}(s\,,x'')\right|
			\right|^{k/2}\right),
	\end{align}
	thanks to a by-now familiar appeal to a suitable form 
	of the Burkholder--Davis--Gundy inequality. And another familiar
	calculation now reveals that the latter expectation is bounded
	from above by
	\begin{align*}
		&\left\{ \int_0^t\d s\int_{-\infty}^\infty\d x'\int_{-\infty}^\infty\d x''\
			p^{(\nu_1)}_{t-s}(x-x')p^{(\nu_1)}_{t-s}(x-x'')\left\|
			\mathcal{D}(s\,,x')\mathcal{D}(s\,,x'')\right\|_{k/2}\right\}^{k/2}\\
		&\le\left\{ \int_0^t\left(\int_{-\infty}^\infty 
			p^{(\nu_1)}_{t-s}(x-x')\left\| \mathcal{D}(s\,,x')
			\right\|_k\,\d x'\right)^2\d s\right\}^{k/2}\\
		&\le\left\{\int_0^t
			\sup_{a\in\R}\left\| \mathcal{D}(s\,,a)\right\|_k^2\,\d s
			\right\}^{k/2}.
	\end{align*}
	We now plug this inequality into \eqref{T3:2} and combine with \eqref{T1:2}
	and \eqref{T2:2} in order to conclude that
	\begin{align*}
		\left\| \mathcal{D}(t\,,x)\right\|_k
			&\le A_k^{1/k} |\xi-\xi'|^\gamma
			+ \varepsilon^{-1}(16\rho(0)kt )^{1/2}
			\exp\left( \frac{c_k\rho(0)\xi^2t}{2(1-\varepsilon)^2}\right)
			\sup_{x\in\R}
			\|u_0(x\,,\xi)\|_k\cdot
			|\xi-\xi'|\\
		&\hskip2.8in +\left\{4\rho(0){|\xi'|^2}k\int_0^t
			\sup_{a\in\R}\left\| \mathcal{D}(s\,,a)\right\|_k^2\,\d s
			\right\}^{1/2}.
	\end{align*}
	Because the right-hand side is independent of $x$, and since
	$c_k\le 8k$ for all $k\ge 2$, we may optimize
	the left-hand side and write
	\[
		\mathbf{E}(t) := \sup_{a\in\R}\|\mathcal{D}(t\,,a)\|_k^2
		\quad\text{for all $t\ge0$},
	\]
	in order to see that\footnote{We have also used the elementary inequality,
	$(a+b+c)^2\le 3(a^2+b^2+c^2)$, valid for all $a,b,c\in\R$.}
	\begin{align*}
		\mathbf{E}(t) &\le 3A_k^{2/k}|\xi-\xi'|^{2\gamma}
			+ 48\varepsilon^{-2}\rho(0)kt 
			\exp\left( \frac{2c_k\rho(0)\xi^2t}{2(1-\varepsilon)^2}\right)
			\sup_{x\in\R}
			\|u_0(x\,,\xi)\|_k^2\cdot
			|\xi-\xi'|^2\\
		&\hskip3.6in+12\rho(0){|\xi'|^2}k\int_0^t\mathbf{E}(s)\,\d s,
	\end{align*}
	for all $t>0$ and $\varepsilon\in(0\,,1)$. Set $\varepsilon:=\tfrac12$,
	to be concrete, and appeal to Gronwall's lemma and the fact that
	$0<\gamma\le 1$ in order to deduce the following:\footnote{It might help to recall that
	$c_k\le 8k$ for all $k\ge 2$.} 
	\begin{align*}
		\mathbf{E}(t) &\le 
			\left[  3A_k^{2/k}|\xi-\xi'|^{2\gamma}
			+ 48\cdot 4\rho(0)kt \,
			\e^{4c_k\rho(0)\xi^2t}
			\sup_{x\in\R}
			\|u_0(x\,,\xi)\|_k^2\cdot
			|\xi-\xi'|^2\right]\e^{12\rho(0)|\xi'|^2kt}\\
		&\le \left[  3A_k^{2/k} + 48\cdot 4\rho(0)kt \,
			\sup_{x\in\R}
			\|u_0(x\,,\xi)\|_k^2 \right]
			\e^{60k\rho(0)L^2t}\cdot|\xi-\xi'|^{2\gamma},
	\end{align*}
	uniformly for all $t\ge0$, $\xi,\xi'\in\R$
	that satisfy 
	\begin{equation}\label{eta:xi:xi':L}
		|\xi-\xi'|\le 1
		\quad\text{and}\quad
		|\xi|\vee|\xi'|\le L.
	\end{equation}
	Thus, we can exchange the respective roles of $\xi$ and $\xi'$ in order
	to arrive at the following moment bound:
	\begin{equation}\label{diff:u:xi}\begin{split}
		&\sup_{a\in\R}\E\left(\left| u(t\,,a\,,\xi) - u(t\,,a\,,\xi') \right|^k \right)\\
		&\le \left[  3A_k^{2/k} + 48\cdot 4\rho(0)kt \,
			\sup_{x\in\R}
			\|u_0(x\,,\xi)\|_k^2 \right]^{k/2}
			\e^{30\rho(0)L^2k^2t}\cdot|\xi-\xi'|^{k\gamma}\\
		&\le {192^{k/2}}\left[  A_k+ (2 \rho(0)kt )^{k/2}\,
			\sup_{x\in\R}
			\E\left(|u_0(x\,,\xi)|^k\right) \right]
			\e^{30\rho(0)L^2k^2t}\cdot|\xi-\xi'|^{k\gamma},
	\end{split}\end{equation}
	uniformly for all $t\ge0$ and $\xi,\xi'\in\R$
	that satisfy \eqref{eta:xi:xi':L}. This completes the proof.
\end{proof}

\section{Proof of Theorem \ref{th:exist:U}}

We have laid the groundwork for the proof of Theorem \ref{th:exist:U} and begin the
task of proving that result now.

Thanks to Assumption \ref{assum:1} of the Introduction,
there exists $\varepsilon\in(0\,,1)$ such that
\begin{equation}\label{eps:rho:T}
	\nu_2 > \frac{\rho(0)}{2(1-\varepsilon)^2}.
\end{equation}
Throughout the proof, we choose and fix this $\varepsilon\in(0\,,1)$.

Suppose $\theta_0:=\{\theta_0(x\,,y)\}_{x,y\in\R}$ is a 2-parameter, complex-valued
random field that satisfies Assumptions \ref{assum:1}--\ref{assum:4} of
Section \ref{sec:The itowalsh solution}. Then, among other things, the following random field is 
well defined and independent of $V$:
\begin{equation}\label{u_0:T_0}
	u_0(x\,,\xi) := \wh{\theta}_0(x\,,\xi) := \int_{-\infty}^\infty\e^{iy\xi} \theta_0(x\,,y)\,\d y
	\qquad[x,\xi\in\R].
\end{equation}
Indeed, the integral is well defined because $\theta_0$ is continuous (Assumption \ref{assum:4})
and  because
\begin{equation}\label{supsup}
	\sup_{x,\xi\in\R}\| u_0(x\,,\xi)\|_k \le \sup_{x\in\R}\int_{-\infty}^\infty
	\| \theta_0(x\,,y)\|_k\,\d y<\infty 
\end{equation}
by Assumption \ref{assum:2}.
By Assumption \ref{assum:3}, for every $x,x',\xi\in\R$ and $k\ge2$,
\[
	\|u_0(x\,,\xi) - u_0(x',\xi)\|_k\le\int_{-\infty}^\infty
	\|\theta_0(x\,,y)-\theta_0(x',y)\|_k\,\d y \le C_0|x-x'|^\alpha,
\]
where $C_0=C_0(k)$. And since
\begin{align*}
	\left| \e^{ia} - \e^{ib}\right| &= \sqrt{2\left[ 1-\cos(a-b)\right]} 
	\le 2\left( |a-b|\wedge 1\right) \le 2|a-b|^\eta
	\qquad\text{for all $a,b\in\R$},
\end{align*}
for every $\eta\in(0\,,1]$, $x,\xi,\xi'\in\R$, and $k\ge 2$,
\[
	\|u_0(x\,,\xi) - u_0(x\,,\xi') \|_k \le
	2|\xi-\xi'|^\eta\sup_{a\in\R}\int_{-\infty}^\infty 
	|y|^\eta\|\theta_0(a\,,y)\|_k\,\d y,
\]
which goes to zero as $\xi'\to\xi$ by Assumption \ref{assum:2}.
These computations, in conjunction, verify the hypotheses
of Theorems \ref{th:u:exist} and \ref{th:reg:u}. Let us record these conclusions
next.

\begin{lemma}\label{lem:conclusion}
	Suppose $u_0$ is defined by \eqref{u_0:T_0}, where $\theta_0$ is a 2-parameter,
	complex-valued random field that satisfies the Assumptions \ref{assum:1}--\ref{assum:4}
	of Section \ref{sec:The itowalsh solution}. Then $u_0$ satisfies the assumptions of Theorems \ref{th:u:exist}
	and \ref{th:reg:u}; that is:
	\begin{compactenum}
		\item $\sup_{x,\xi\in\R}\|u_0(x\,,\xi)\|_k<\infty$ for all $k\ge 2$;
		\item $u_0$ satisfies \eqref{u_0:1.5} with $\gamma=\eta$.
	\end{compactenum}
\end{lemma}

Thus, it follows from Theorems \ref{th:u:exist} and \ref{th:reg:u} that
the SPDE \eqref{PAM} has a unique random field 
solution $u[\xi]$ for every $\xi\in\R$, starting from initial data $u_0$
given by \eqref{u_0:T_0}; and that $u$ is H\"older continuous, as
guaranteed by Theorem \ref{th:reg:u}. Finally, $u$ is subject to the moment growth
bound of Theorem \ref{th:u:exist}. 

Now let us define a 3-parameter, complex-valued, random field 
$U=\{U(t\,,x\,,\xi)\}_{t\ge0,x,\xi\in\R}$
via \eqref{eq:Uu}; that is, we recall,
$U(t\,,x\,,\xi) = \exp(-\nu_2 \xi^2 t)u(t\,,x\,,\xi)$ for all $t\ge0$ and
$x,\xi\in\R$. Since 
\[
	U_0(x\,,\xi) := u_0(x\,,\xi)\qquad\text{for all $x,\xi\in\R$},
\]
Theorem \ref{th:U:exist} and \ref{th:reg:u} ensure that $U$ is a continuous, complex-valued,
random field that satisfies the moment growth conditions of Theorem \ref{th:U:exist}
and solves uniquely the SPDE \eqref{PAM:U}. Now, motivated by the informal 
definition \eqref{U:T} of $U$, we may define a 3-parameter, complex-valued,
random field $\theta:=\{\theta(t\,,x\,,y)\}_{t\ge0,x,\xi\in\R}$ via
$\theta(0\,,x\,,y) := \theta_0(x\,,y)$ and
\begin{equation}\label{T:U}
	\theta(t\,,x\,,y) := \frac{1}{2\pi}\int_{-\infty}^\infty\e^{-iy\xi}\,
	U(t\,,x\,,\xi)\,\d\xi
	\qquad\text{for every $t\ge0,\, x,y\in\R$},
\end{equation}
where $U(0\,,x\,,\xi) := U_0(x\,,\xi) = u_0(x\,,\xi) = \wh{\theta}_0(x\,,\xi)$;
see \eqref{u_0:T_0}. 

In due time, we will prove that the random field $\theta$ is the unique mild solution to
the SPDE \eqref{Kraichnan} and derive the asserted properties of $\theta$ that were outlined 
in Theorem \ref{th:exist:U}.

First of all, let us remark that  $\theta$ is a well-defined, predictable random field.
This is because $U$ is continuous
(see Lemmas \ref{u-u:xi} and
\ref{lem:conclusion}), and
since Theorem \ref{th:U:exist} ensures that, for the same $\varepsilon\in(0\,,1)$
that appeared earlier in \eqref{eps:rho:T}, and for every $t>0$ and $x,y\in\R$,
\begin{equation}\label{L2:T}\begin{split}
	\|\theta(t\,,x\,,y)\|_2 &\le {\frac{1}{2\pi{\varepsilon}}}\,\sup_{x,\xi\in\R}\|U_0(x\,,\xi)\|_2
		\cdot\int_{-\infty}^\infty
		\exp\left( -\left[\nu_2 - \frac{\rho(0)}{2(1-\varepsilon)^2}\right]\xi^2 t\right)\d\xi\\
	&\le{\frac{1}{2\varepsilon\sqrt{\pi t}}}\,\sup_{a\in\R}\int_{-\infty}^\infty
		\| \theta_0(a\,,b)\|_2\,\d b
		\cdot \left[\nu_2 - \frac{\rho(0)}{2(1-\varepsilon)^2}\right]^{-1/2}\\
	&<\infty;
\end{split}\end{equation}
see also \eqref{u_0:T_0}. It also follows that
the first assertion of the
dissipation relation \eqref{L2:diss} holds.
The estimate \eqref{L2:T} has also the consequence that $y\mapsto \theta(t\,,x\,,y)$ is
locally integrable a.s.\ for every $t>0$ and $x\in\R$.
Since $\theta$ is the inverse Fourier transform of $U$, it
then follows from the Parseval identity that $U$ must then be the Fourier transform
of $\theta$ in the sense of distributions. This justifies some of the assertions surrounding
\eqref{U:T}.

Let us note next that, for every integer $n\ge1$ and all reals $t>0$
and $x\in\R$,
\begin{align*}
	\int_{-\infty}^\infty |\xi|^n \left\| \wh{\theta}(t\,,x\,,\xi)\right\|_2\,\d\xi
		&=\int_{-\infty}^\infty |\xi|^n \left\| U(t\,,x\,,\xi)\right\|_2\,\d\xi\\\notag
	&\le {\frac{1}{2\pi{\varepsilon}}}\,\sup_{x,\xi\in\R}\|u_0(x\,,\xi)\|_2
		\cdot\int_{-\infty}^\infty |\xi|^n 
		\exp\left( -\left[\nu_2 - \frac{\rho(0)}{2(1-\varepsilon)^2}\right]\xi^2 t\right)\d\xi\\
		\notag
	&<\infty.
\end{align*}
This shows that 
\[
	\int_{-\infty}^\infty |\xi|^n \left| \wh{\theta}(t\,,x\,,\xi)\right|\d\xi\in L^2(\Omega)
	\qquad\text{for every $n\ge1$, $t>0$, and $x\in\R$},
\]
and hence, owing to the inversion theorem
of Fourier analysis, 
$y\mapsto \theta(t\,,x\,,y)$ is a.s.\ $C^\infty$ for all $t>0$ and $x\in\R$,
with
\[
	\partial^n_y \theta(t\,,x\,,y) = \frac{1}{2\pi}
	\int_{-\infty}^\infty (i\xi)^n\e^{-iy\xi}\,
	\wh{\theta}(t\,,x\,,\xi)\,\d\xi
	\qquad\text{a.s.\ for every $n\ge1$, $t>0$, and $x,y\in\R$}.
\]
Also, the second assertion of  \eqref{L2:diss} follows from the above reasoning
(set $n=1$).

It remains to prove that $\theta$ is both a mild and a weak solution to
\eqref{Kraichnan} and it is unique in the sense that is stated.

Since $u$ is the mild solution to \eqref{PAM}, a standard application of
a stochastic Fubini theorem (see Theorem \ref{th:stochastic:Fubini})
implies that $u$ is also a weak solution to \eqref{PAM};
see Walsh \cite{Walsh}. We will use this fact next. 

Define
\[
	\phi(t\,,x\,,y\,,\xi) := \psi_0(t)  \psi_1(x) \psi_2(y)
	\e^{-\nu_2 \xi^2 t -iy \xi},
\]
where $\psi_0\in C_c(0\,,\infty)$ and $\psi_1,\psi_2\in\mathscr{S}(\R)$. Since $u$ is a weak solution to
\eqref{PAM},
\begin{align*}
	&-\int_0^{\infty}\d t\int_{\R^3}\d x\,\d y\,\d\xi\ 
		u(t\,,x\,,\xi)  \partial_t \phi(t\,,x\,,y\,,\xi)\\
	&\hskip1in= \nu_1 \int_0^{\infty}\d t\int_{\R^3}\d x\,\d y\,\d\xi\ 
		u(t\,,x\,,\xi) \partial^2_x\phi(t\,,x\,,y\,,\xi)\\
	&\hskip2in + i\int_0^{\infty}\d t\int_{\R^3}\d x\,\d y\,\d\xi\ 
		\xi u(t\,,x\,,\xi) V(t\,,x)\phi(t\,,x\,,y\,,\xi)
		\qquad\text{a.s.,}
\end{align*}
where the final stochastic integral is understood as a Walsh integral with
respect to the Gaussian noise $(t\,,x\,,y\,,\xi)\mapsto V(t\,,x)$.
Therefore, it follows from \eqref{T:U}, Fubini's theorem, and a stochastic Fubini theorem 
(Theorem \ref{th:stochastic:Fubini}) that
with probability one,
\begin{align*}
	&-\int_0^{\infty}\d t\int_{\R^2}\d x\,\d y\ 
		\theta(t\,,x\,,y) \left[\psi_0'(t)\psi_1(x)\psi_2(y) - 
		\nu_1\psi_0(t)\psi_1''(x)\psi_2(y) - \nu_2\psi_0(t)\psi_1(x)\psi_2''(y) \right]\\
	&\hskip2.5in  = \int_0^{\infty}\d t \int_{\R^2}\d x\,\d y\ 
		\partial_t\theta(t\,,x\,,y) 
		\psi_0(t)\psi_1(x)\psi_2(y) V(t\,,x).
\end{align*}
This verifies \eqref{weak}. 
Next we prove that $\theta(t\,,x\,,y)$ is a mild solution to equation \eqref{Kraichnan}.

Since $u(t\,,x\,,\xi)$ is a mild solution to $\C$-valued SPDE \eqref{PAM}, 
\[
	u(t\,,x\,,\xi) = \int_{-\infty}^\infty p_t^{(\nu_1)}(x-x')u_0(x', \xi) \,\d x' 
	+ i\xi\int_{(0,t)\times\R} p_{t-s}^{(\nu_1)}(x-x') \e^{-\nu_2 \xi^2 t}
	u(s\,, x',\xi)V(s\,, x') \,\d s\,\d x',
\]
almost surely.
We first multiply both sides by $(2\pi)^{-1}\exp(-\nu_2 \xi^2 t-i y \xi) $, 
integrate $[\d \xi]$, and then appeal to both Fubini and the stochastic Fubini
theorems (see Theorem \ref{th:stochastic:Fubini} for the latter), in order to find that
\begin{align}\notag
	\theta(t\,,x\,,y) &=\frac{1}{2\pi} \int_{-\infty}^\infty 
		\e^{-i y \xi-\nu_2 \xi^2 t} u(t\,,x\,,\xi)\,\d \xi 
		\qquad[\text{see \eqref{eq:Uu} and \eqref{T:U}}]\\
	&= \frac{1}{2\pi} \int_{\R^2} p_t^{(\nu_1)}(x-x') u_0(x', \xi)
		\e^{-\nu_2 \xi^2 t -i y \xi}\, \d \xi\, \d x' 
		\label{theta:I1:I2}\\\notag
	&\hskip.6in  + \frac{i}{2\pi} \int_{(0,t)\times\R^2} p_{t-s}^{(\nu_1)}(x-x') 
		\e^{-i y \xi-\nu_2 (t-s)\xi^2-\nu_2 s\xi^2} \xi u(s\,, x', \xi) V(s\,, x') \,\d s\,\d x'\,\d \xi\\
		\notag
	&:= I_1 + I_2 \qquad\text{a.s.}
\end{align}
Thanks to the stochastic Fubini theorem (Theorem \ref{th:stochastic:Fubini}), 
$I_2$ can be interpretted either as a Walsh integral
with respect to the Gaussian noise $(s\,,x',\xi)\mapsto V(s\,,x')$,
or equivalently, as 
\[
	\frac{i}{2\pi} \int_{-\infty}^\infty\left(
	\int_{(0,t)\times\R} p_{t-s}^{(\nu_1)}(x-x') \e^{-i y \xi-\nu_2 (t-s)\xi^2-\nu_2 s\xi^2}
	u(s\,, x', \xi) V(s\,, x') \,\d s\,\d x'\right)
	\xi \,\d \xi
\]
We skip the routine measure-theoretic details.

By \eqref{u_0:T_0} and the inversion theorem,
\[
	\frac{1}{2\pi}\int_{-\infty}^\infty u_0(x', \xi) \e^{-\nu_2 t \xi^2-i y \xi}
	\, \d \xi = \int_{-\infty}^\infty p_t^{(\nu_2)}(y-y')\theta_0(x', y')\,\d y'.
\]
The inversion theorem is applicable, owing to \eqref{supsup}.
Therefore, we may evaluate $I_1$ as follows:
\begin{equation}\label{theta:I1}
	I_1 = \int_{-\infty}^\infty p_t^{(\nu_1)} (x-x') 
	p_t^{(\nu_2)} (y-y')\theta_0(x', y')\, \d x'\, \d y'.
\end{equation}

In order to evaluate $I_2$ we apply the stochastic Fubini theorem 
(Theorem \ref{th:stochastic:Fubini}) to find that
\begin{align*}
	I_2 &=  \int_{(0,t)\times\R} p_{t-s}^{(\nu_1)}(x-x')
		\left[ \frac{i}{2\pi}\int_{-\infty}^\infty
		\e^{-i y \xi-\nu_2 (t-s)\xi^2-\nu_2 s\xi^2} \xi
		u(s\,, x', \xi)\, \d \xi\right] V(s\,, x')\,\d s\,\d x'\\
	&= \int_{(0,t)\times\R} p_{t-s}^{(\nu_1)}(x-x')
		\left[ \int_{-\infty}^\infty p_{t-s}^{(\nu_2)}(y-y')
		\partial_{y'} \theta(s\,, x',y')\, \d y'
		\right] V(s\,, x')\,\d s\,\d x',
\end{align*}
by Plancheral's theorem. Therefore, another appeal to the stochastic Fubini
theorem yields 
\begin{equation}\label{theta:I2}
	I_2 = \int_{(0,t)\times\R^2} 
	p_{t-s}^{(\nu_1)}(x-x')
	p_{t-s}^{(\nu_2)}(y-y')
	\partial_{y'} \theta(s\,, x', y') V(s\,, x')\,\d s\,\d x'\d y'.
\end{equation}
We combine \eqref{theta:I1} and \eqref{theta:I2} and apply
them in \eqref{theta:I1:I2} to see that $\theta$ is indeed a mild
solution to \eqref{Kraichnan}.

Next we prove the uniqueness of the mild solution $\theta$
under condition \eqref{ET^2}.

Let $\theta$ denote any mild solution to \eqref{Kraichnan} that satisfies
\eqref{ET^2}, starting from the same initial profile $\theta_0$. Recall that for all $t>0$
and $x,y\in\R$, the following holds a.s.:
\begin{align*}
	\theta(t\,,x\,,y) &= \int_{\R^2} p_t^{(\nu_1)}(x-x') p_t^{(\nu_2)}(y-y')
		\theta_0(x', y')
		\, \d x' \d y' \\
	& \hskip1in + \int_{(0,t)\times\R^2} 
		p_{t-s}^{(\nu_1)}(x-x') p_{t-s}^{(\nu_2)}(y-y')\,
		\partial_{y'} \theta(s\,, x', y') V(s\,,x') \d s\,\d x' \d y'.
\end{align*}
Since the left-hand side is a.s.\ in $L^1(\d y)$ and the first term 
on the right-hand side is also a.s.\ in $L^1(\d y)$ by Assumption II, 
we see that the second term on the right-hand side is also a.s. in $L^1(\d y)$.
Thus, we may multiply both sides of the preceding by 
$\exp(i y \xi)$ and integrate $[\d y]$ to find that
\begin{align*}
	\wh{\theta}(t\,,x\,,\xi) &= 
		\int_{\R^2}p_t^{(\nu_1)}(x-x')\theta_0(x',y')
		\e^{-\nu_2 \xi^2 t-i y' \xi}\, \d x' \d y' \\
	&\hskip1in + \int_{(0,t)\times\R^2} p_{t-s}^{(\nu_1)}(x-x') 
		\e^{-\nu_2 (t-s) \xi^2+i \xi y'}\,
		\partial_{y'} \theta(s\,, x', y') V(s\,,x') \,\d s\,\d x'\d y'\\
	&= \e^{-\nu_2 \xi^2 t} \int_{-\infty}^\infty p_t^{(\nu_1)}(x-x') \wh{\theta}_0(x', \xi)\, \d x' \\
	&\hskip1in + i\xi\int_{(0,t)\times\R} p_{t-s}^{(\nu_1)}(x-x') 
		\e^{-\nu_2 \xi^2 t+\nu_2 s \xi^2}\, \wh{\theta}(s\,, x', \xi) V(s\,, x')\,\d s\, \d x'
		\qquad\text{a.s.}
\end{align*}
That is, with probability one, the following holds for each $t>0$ and $x,\xi\in\R$:
\begin{align*}
	\e^{\nu_2 \xi^2 t}\wh{\theta}(t\,,x\,,\xi) &=   \int_{-\infty}^\infty 
		p_t^{(\nu_1)}(x-x') \wh{\theta}_0(x', \xi)\, \d x' \\
	&\hskip1in +  i \xi  
		\int_{(0,t)\times\R} p_{t-s}^{(\nu_1)}(x-x') \e^{\nu_2 s \xi^2}
		\wh{\theta}(s\,, x', \xi) V(s\,, x')\,\d s\,\d x'.
\end{align*}
In other words, we have shown that $(t\,,x)\mapsto\exp(\nu_2 \xi^2 t)\wh{\theta}(t\,,x\,,\xi)$ is 
the unique mild solution to equation \eqref{PAM} for every $\xi\in\R$.
Another way to state this is that if $\theta$ and $\wt{\theta}$ are mild solutions to
the fluid problem \eqref{Kraichnan}, both satisfying \eqref{ET^2}
and both having common initial profile $\theta_0$, then their Fourier transforms
[in the $y$ variable] are equal and hence $\theta$ and $\wt{\theta}$ are modifications
of one another, thanks to the uniqueness theorem of Fourier analysis. This proves  uniqueness.
Finally, we verify \eqref{IV}.

We have already shown that $\wh{\theta}=U$ in the sense of distributions.
Therefore,  \eqref{eq:Uu}
and the Parseval identity together imply that
for every $t>0$
and non-random functions $\psi_1,\psi_2\in\mathscr{S}(\R)$,
\[
	\<\theta(t)\,,\psi_1\otimes\psi_2\>_{L^2(\R^2)}
	=\frac{1}{2\pi}\int_{\R^2} \e^{-\nu_2\xi^2 t}
	u(t\,,x\,,\xi) \psi_1(x)
	\overline{\wh{\psi}_2(\xi)}\,\d x\,\d\xi.
\]
A similar argument implies that
\[
	\<\theta(0)\,,\psi_1\otimes\psi_2\>_{L^2(\R^2)}
	=\frac{1}{2\pi}\int_{\R^2}
	u_0(x\,,\xi) \psi_1(x)
	\overline{\wh{\psi}_2(\xi)}\,\d x\,\d\xi.
\]
Next, we write
\begin{align}\label{Tt-T0}
	&\left\| \<\theta(t)\,,\psi_1\otimes\psi_2\>_{L^2(\R^2)} - 
		\<\theta(0)\,,\psi_1\otimes\psi_2\>_{L^2(\R^2)}\right\|_2\\\notag
	&\hskip1in\le \int_{\R^2}
		\left\| \e^{-\nu_2\xi^2 t}u(t\,,x\,,\xi) -
		u_0(x\,,\xi) \right\|_2|\psi_1(x)\wh{\psi}_2(\xi)|\,\d x\,\d\xi\\\notag
	&\hskip1in\le\int_{\R^2}\e^{-\nu_2\xi^2 t}
		\left\| u(t\,,x\,,\xi) -
		u_0(x\,,\xi) \right\|_2|\psi_1(x)\wh{\psi}_2(\xi)|\,\d x\,\d\xi\\\notag
	&\hskip2.7in+ \int_{\R^2}\left( 1 -\e^{-\nu_2\xi^2t}\right)
		\| u_0(x\,,\xi) \|_2|\psi_1(x)\wh{\psi}_2(\xi)|\,\d x\,\d\xi.
\end{align}

According to Lemma \ref{lem:conclusion},
$\sup_{a,b\in\R}\|u_0(a\,,b)\|_2$ is finite. Therefore, we may
set $k=2$ in Lemma \ref{u-u:t=0}, and recall that $c_2=1$
[Lemma \ref{lem:Young}], in order to see that
the constant $\wt{C}$ of Lemma \ref{u-u:t=0} can be bounded above as follows:
As long as $t\in(0\,,1]$,
\begin{align*}
	\wt{C} \le K_1 + K_2|\xi|\exp\left( 
	\frac{\rho(0)\xi^2t}{2(1-\varepsilon)^2}\right),
\end{align*}
where $K_1$ and $K_2$ do not depend on $(t\,,\xi)$, and $\varepsilon\in(0\,,1)$
is the same constant that was held fixed in \eqref{eps:rho:T}.
Because of \eqref{eps:rho:T}, the condition ``$t\in(0\,,1]$'' implies that
\[
	\wt{C}\e^{-\nu_2\xi^2t}\le K_1 + K_2|\xi|\exp\left( -
	\left[\nu_2 - \frac{\rho(0)}{2(1-\varepsilon)^2}\right]\xi^2t\right)
	\le K_1+K_2|\xi|.
\]
Consequently, the dominated convergence theorem implies that
\begin{equation}\label{T->0:1}
	\lim_{t\downarrow0}\int_{\R^2}\e^{-\nu_2\xi^2 t}
	\left\| u(t\,,x\,,\xi) -
	u_0(x\,,\xi) \right\|_2|\psi_1(x)\wh{\psi}_2(\xi)|\,\d x\,\d\xi=0.
\end{equation}
Furthermore, Lemma \ref{lem:conclusion} ensures that
\[
	\int_{\R^2}\|u_0(x\,,\xi)\|_2|\psi_1(x)\wh{\psi}_2(\xi)|\,\d x\,\d\xi
	\le\|\psi_1\|_{L^1(\R)}\,\|\wh{\psi}_2\|_{L^1(\R)}
	\sup_{a,b\in\R} \|u_0(a\,,b)\|_2<\infty,
\]
whence it follows from the dominated convergence theorem that
\begin{equation}\label{T->0:2}
	\lim_{t\downarrow0} \int_{\R^2}\left( 1 -\e^{-\nu_2\xi^2t}\right)
	\| u_0(x\,,\xi) \|_2|\psi_1(x)\wh{\psi}_2(\xi)|\,\d x\,\d\xi=0.
\end{equation}
We obtain \eqref{IV} by combining \eqref{Tt-T0},
\eqref{T->0:1}, and \eqref{T->0:2}. This completes the last part of the 
demonstration of Theorem \ref{th:exist:U}.
\qed

\section{Proof of Theorem \ref{th:cont:U}}\label{sec:pf:th:cont:U}

In this section we verify the regularity Theorem \ref{th:cont:U}.
We also use this opportunity to study various ``curvilinear stochastic
integrals'' along the field $V$. In fact, we start with the latter topic.

\subsection{Smoothing the noise}

One of the objects that arises naturally is the random field
$(t\,,x)\mapsto \int_0^t V(s\,,x)\,\d s$. There is a well-known
method to construct this and related random fields from the generalized Gaussian
field $V$; see for example Kunita \cite[Section 6.2]{Kunita} for an indirect construction
and Hu and Nualart \cite{HuNualart} for a direct construction.
We will need to use aspects of the latter construction. With that
aim in mind define a smoothed approximation to the random distribution
$V(t\,,x)$
as follows: For all $\varepsilon,\delta>0$, $t\ge0$, and $x\in\R$ define
\begin{equation}\label{V:eps}
	V_{\varepsilon,\delta}(t\,,x) := \int_{\R^2} 
	p_\varepsilon^{(\nu)}(t-t')p_\delta^{(\nu)}(x-x')
	V(t',x')\,\d t'\,\d x'.
\end{equation}
The defining properties of the isonormal process $V$
ensure that $V_{\varepsilon,\delta}$ is a centered two-parameter
Gaussian random field with covariance
\begin{equation}\label{Cov:V:eps}
	\Cov\left[ V_{\varepsilon,\delta}(t\,,x) ~,~ V_{\alpha,\beta}(t',x')\right]
	= p_{\varepsilon+\alpha}^{(\nu)}(t-t') \cdot (p_{\delta+\beta}^{(\nu)}*\rho)(x-x'),
\end{equation}
for every $t,t'\ge0$ and $x,x'\in\R$.
See Appendix \ref{Append: Wiener int}. The following records these,
and a few other, properties of $V_{\varepsilon,\delta}$.

\begin{proposition}\label{pr:V:eps}
	For every $\varepsilon,\delta>0$, $V_{\varepsilon,\delta}$ is a 
	centered, 2-parameter, stationary
	Gaussian random field that has (up to a modification) $C^\infty$ trajectories.
\end{proposition}

\begin{proof}
	Choose and fix $\varepsilon,\delta>0$ throughout.
	We  need only to verify the smoothness of the random field 
	$(t\,,x)\mapsto V_{\varepsilon,\delta}(t\,,x)$.
	
	Let us first demonstrate that $V_{\varepsilon,\delta}$
	is a.s.\ continuous. Thanks to \eqref{Cov:V:eps},
	\[
		\Var\left[V_{\varepsilon,\delta}(t\,,x)\right] = 
		\Var\left[V_{\varepsilon,\delta}(0\,,0)\right] =
		\frac{1}{\sqrt{8\pi\nu\varepsilon}}
		\cdot\left( p_{2\delta}^{(\nu)}*\rho\right)(0),
	\]
	for every $t\ge0$ and $x\in\R$.
	Therefore, for all $t,h\ge0$ and $x,y\in\R$,
	\begin{align*}
		\E\left( \left| V_{\varepsilon,\delta}(t+h\,,x) - V_{\varepsilon,\delta}(t\,,x) \right|^2\right)
			&=2 \left[ p_{2\varepsilon}^{(\nu)}(0)-p_{2\varepsilon}^{(\nu)}(h)\right]
			\left( p_{2\delta}^{(\nu)}*\rho\right)(0)
		\intertext{and}
		\E\left( \left| V_{\varepsilon,\delta}(t\,,x) - V_{\varepsilon,\delta}(t\,,y) \right|^2\right)
			&=\frac{1}{\sqrt{2\pi\nu\varepsilon}}\,
			\left[ \left(p_{2\delta}^{(\nu)}*\rho\right)(0) -
			\left( p_{2\delta}^{(\nu)}*\rho\right)(x-y)
			\right].
	\end{align*}
	Let us examine the two expressions separately.
	
	Since $(\partial_x p_{2\varepsilon}^{(\nu)})(0)=0$ and 
	$(\partial^2_x p_{2\varepsilon}^{(\nu)})(0)=-(8\nu\varepsilon)^{-1}
	(2\varepsilon \nu \pi)^{-1/2}\neq0$, a Taylor expansion yields
	\begin{equation}\label{mod:V:eps:t}
		\lim_{h\downarrow0}
		\frac{1}{h^2}\,
		\E\left( \left| V_{\varepsilon,\delta}(t+h\,,x) - V_{\varepsilon,\delta}(t\,,x) \right|^2\right)
		= \frac{1}{{8(\nu\varepsilon)^{3/2}}\sqrt{2\pi}}\,\left( p_{2\delta}^{(\nu)}*\rho\right)(0)>0,
	\end{equation}
	uniformly for all $t\ge0$ and $x\in\R$.
	Also, by \eqref{rho:max},
	\begin{align*}
		\left( \partial_x(p_{2\delta}^{(\nu)}*\rho) \right)(0) &= 
			\left(  (\partial_x p_{2\delta}^{(\nu)})*\rho\right)(0)
			=-\frac{1}{4\nu\delta}\int_{-\infty}^\infty x p_{2\delta}^{(\nu)}(x)\rho(x)\,\d x=0,
		\intertext{because $p_{2\delta}\times\rho$ 
			is an even function (see \eqref{rho:max}), and}
		\left(\partial^2_x (p_{2\delta}^{(\nu)}*\rho) \right)(0)= 
			&\lim_{\varepsilon \to 0} \int_{-\infty}^\infty 
				\left( (\partial^2_x p_{2\delta}^{(\nu)})* \rho\right)(x) p_{\varepsilon}^{(\nu)}(x) \d x\\
			=&\lim_{\varepsilon \to 0} \frac{1}{2\pi} \int_{-\infty}^\infty -\xi^2 
				\wh{p}_{2\delta}^{(\nu)}(\xi)\e^{-\nu \varepsilon |\xi|^2} \mu(\d \xi) < 0,
	\end{align*}
	thanks to the dominated convergence theorem. 
	Therefore, 
	\begin{equation}\label{mod:V:eps:x}
		\lim_{y\to x}
		\frac{1}{(x-y)^2}\,
		\E\left( \left| V_{\varepsilon,\delta}(t\,,x) - 
		V_{\varepsilon,\delta}(t\,,y) \right|^2\right)
		= \frac{1}{8\nu \delta(\nu\varepsilon)^{1/2}\sqrt{2\pi}}\,
		\left( p_{2\delta}^{(\nu)}*\rho\right)(0)>0,
	\end{equation}
	uniformly for all $t\ge0$ and $x\in\R$.
	The asserted a.s.-continuity of $V_{\varepsilon,\delta}$ follows 
	from \eqref{mod:V:eps:t} and
	\eqref{mod:V:eps:x}, together with a suitable form of the Kolmogorov
	continuity theorem for Gaussian processes (see for example 
	\cite[Theorem C.6, p.\ 107]{Khoshnevisan}).
	
	In order to prove that $V_{\varepsilon,\delta}$ is a.s.\ smooth, let us first note that
	since $V_{\varepsilon,\delta}$ is a.s.\ continuous it has a.s.-measurable trajectories.
	Therefore,  a stochastic Fubini theorem (see Theorem \ref{th:stochastic:Fubini})  yields
	\[
		\int_{\R^2} p_\varepsilon^{(\nu)}(s-t) p_\delta^{(\nu)}(y-x) V_{\varepsilon,\delta}(t\,,x)\,\d t\,\d x
		= V_{2\varepsilon,2\delta}(s\,,y)\qquad\text{a.s.}
	\]
	Because $V_{\varepsilon,\delta}$ is a.s.\ continuous,
	the left-hand side is a classical convolution, and is easily seen to be a.s.\ a $C^\infty$
	function of $(s\,,y)$; therefore so is the right-hand side. Since $\varepsilon$
	and $\delta$ are positive and otherwise arbitrary, this proves that 
	$V_{2\varepsilon,2\delta}$---whence also
	$V_{\varepsilon,\delta}$---is $C^\infty$ a.s. This completes the proof.
\end{proof}

\subsection{Curvilinear stochastic integrals}
We now can construct stochastic integrals of the form
$A_t := \int_0^t V(s\,,f(s))\,\d s$ where $f:\R\to\R$ is continuous and independent
of $V$. One can think of $A_t$ as the total amount of $V$-noise that is accumulated
along the graph of $f$. As such, $A_t$ can be thought of as a curvilinear stochastic
integral. The terminology is borrowed in essence from the work of
Bertini and Cancrini \cite{BC}.
We change the notation slightly from the above discussion, however,
in order to accomodate our later needs.

\begin{lemma}\label{lem:V}
	Let $X=\{X_s\}_{s\geq 0}$ denote an a.s.-continuous
	stochastic process that is independent of the Gaussian noise $V$.
	Then, for every $t>0$ and $x\in\R$, 
	\[
		\int_0^t V(s\,,x+X_{t-s})\,\d s :=
		\lim_{\varepsilon,\delta\downarrow0}
		\int_0^t V_{\varepsilon,\delta}(s\,,x+X_{t-s})\,\d s
	\]
	exists boundedly in $L^2(\Omega)$. Moreover,
	\[
		\E\left[\int_0^t V(s\,,x+X_{t-s})\,\d s \right]=0
		\quad\text{and}\quad
		\Var\left[ \int_0^t V(s\,,x+X_{t-s})\,\d s \right] = t\rho(0).
	\]
\end{lemma}

The proof of this, and the next result, rely on the 
following consequence of the elementary properties of
Wiener integrals: The conditional distribution of the 4-parameter process 
$(\varepsilon\,,\delta\,,t\,,x)\mapsto V_{\varepsilon,\delta}(t\,,x)$
is centered Gaussian, given the process $X$. In fact, we use
elementary fact several times in the sequel, frequently without explicitly mentioning
the fact itself.

\begin{proof}[Proof of Lemma \ref{lem:V}]
	Choose and fix $t>0$ and $x\in\R$. Then,
	\begin{equation}\label{EV|B}
		\E\left[\left. \int_0^t V_{\varepsilon,\delta}(s\,,x+X_{t-s})\,\d s
		\ \right|\, X\right]=0\qquad\text{a.s.\
		for all $\varepsilon,\delta>0$}.
	\end{equation}
	In particular, $\int_0^t V_{\varepsilon,\delta}(s\,,x+X_{t-s})\,\d s$
	has mean zero.  Furthermore,  \eqref{Cov:V:eps} implies that
	\begin{align*}
		&\E\left(\left. \int_0^t V_{\varepsilon,\delta}(s\,,x+X_{t-s})\,\d s\cdot
			\int_0^t V_{\alpha,\beta}(r\,,x+X_{t-r})\,\d r\ \right|\ X\right)\\
		&\hskip2.7in=\int_0^t\d s\int_0^t\d r\
			p_{\varepsilon+\alpha}(s-r)\left( p_{\delta+\beta}*\rho\right)(X_{t-s}-X_{t-r}).
	\end{align*}
	Note that
	$s\mapsto
	\int_0^t p_{\varepsilon+\alpha}(s-r) ( p_{\delta+\beta}*\rho )(X_{t-s}-X_{t-r})\,\d r$
	is a.s.\ continuous uniformly on $[0\,,t]$. Furthermore,
	\[
		\int_0^t p_{\varepsilon+\alpha}(s-r) ( p_{\delta+\beta}*\rho )(X_{t-s}-X_{t-r})\,\d r
		\le\rho(0)\int_0^t p_{\varepsilon+\alpha}(s-r) \,\d r\le\rho(0),
	\]
	by \eqref{rho:max}.  Since $\rho$ is uniformly continuous---see \eqref{rho:max}---it
	follows from the Feller property of the heat semigroup that
	\begin{equation}\label{CovV|B}
		\lim_{\varepsilon,\alpha,\delta,\beta\downarrow0}
		\E\left(\left. \int_0^t V_{\varepsilon,\delta}(s\,,x+X_{t-s})\,\d s\cdot
		\int_0^t V_{\alpha,\beta}(r\,,x+X_{t-r})\,\d r\ \right|\ X\right)
		= t\rho(0).
	\end{equation}
	We have also seen that, for every $\varepsilon,\delta,\alpha,\beta>0$,
	\[
		0\le \E\left(\left. \int_0^t V_{\varepsilon,\delta}(s\,,x+X_{t-s})\,\d s\cdot
		\int_0^t V_{\alpha,\beta}(r\,,x+X_{t-r})\,\d r\ \right|\ X\right)
		\le t\rho(0)\quad\text{a.s.}
	\]
	Therefore, 
	\[
		\lim_{\varepsilon,\alpha,\delta,\beta\downarrow0}
		\E\left[ \int_0^t V_{\varepsilon,\delta}(s\,,x+X_{t-s})\,\d s\cdot
		\int_0^t V_{\alpha,\beta}(r\,,x+X_{t-r})\,\d r\right]
		= t\rho(0),
	\]
	by the bounded convergence theorem. This shows, in particular, that
	\[
		\lim_{\varepsilon,\alpha,\delta,\beta\downarrow0}
		\E\left( \left| \int_0^t V_{\varepsilon,\delta}(s\,,x+X_{t-s})\d s
		- \int_0^t V_{\alpha,\beta}(r\,,x+X_{t-r})\,\d r\right|^2\right)=0.
	\]
	Thus, we see that $(\varepsilon\,,\delta)\mapsto 
	\int_0^t V_{\varepsilon,\delta}(r\,,x+X_{t-r})\,\d r$
	is a Cauchy net in $L^2(\Omega)$. This, \eqref{EV|B}, and
	\eqref{CovV|B} together imply the lemma.
\end{proof}

It is not hard to prove that the 
construction of the just-defined curvilinear stochastic integral 
does not depend essentially on the particular smoothing choices 
that were made in the construction of $V_{\varepsilon,\delta}$.
The following lemma is the first step toward establishing this fact.

\begin{lemma}\label{lem:V:eps:2}
	Let $X$ be as in Lemma \ref{lem:V},
	and $\psi,\phi:\R\to\R$ be 
	two non-random $C^{\infty}$ functions with compact support {such that $\int_{-\infty}^\infty\psi(x)\d x = \int_{-\infty}^\infty \phi(x)\d x =1$. }
	Define $\phi_{\varepsilon}(x):= \varepsilon^{-1}\phi(x/\varepsilon)$ 
	and $\psi_\varepsilon(x) := \varepsilon^{-1}\psi(x/\varepsilon)$
	for every $\varepsilon>0$
	and $x\in\R$, and let
	\begin{equation}\label{V:bar}
		\bar{V}_{\varepsilon,\delta}(t\,,x) := 
		\int_{\R_+\times\R} \psi_\varepsilon(t-t')\phi_\delta(x-x') V(t',x')\,\d t'\d x',
	\end{equation} 
	for all $\varepsilon,\delta>0$, $t\ge0$,  $x\in\R$.
	Then, for every fixed $\varepsilon,\delta>0$, $\bar{V}_{\varepsilon,\delta}$ is a 
	centered, 2-parameter, stationary
	Gaussian random field that has (up to a modification) $C^\infty$ trajectories.
	Moreover, for every $t>0$ and $x\in\R$, 
	\[
		\int_0^t \bar{V}(s\,,x+X_{t-s})\,\d s :=
		\lim_{\varepsilon,\delta\downarrow0}
		\int_0^t \bar{V}_{\varepsilon,\delta}(s\,,x+X_{t-s})\,\d s
		\qquad\text{%
		exists in $L^2(\Omega)$.}
	\]
\end{lemma}

\begin{proof}
	If $\varepsilon >0$ and $\delta >0$ are fixed,
	then $\bar{V}_{\varepsilon,\delta}$ is a well-defined random field,
	thanks to the defining properties of the Wiener integral. 
	In order to show that $(t\,,x)\mapsto \bar{V}_{\varepsilon,\delta}(t\,,x)$ is a.s.\
	smooth let us define, for every integer $n\ge0$ and all reals $\varepsilon,\delta>0$,
	$t\ge0$, and $x\in\R$,
	\[
		\bar{V}^{(n)}_{\varepsilon, \delta}(t\,,x) := 
		\int_{\R^2} \psi_{\varepsilon}(t-s)\frac{\d^n}{\d x^n} 
		\phi_{\delta}(x-y) V(s\,, y) \,\d y\, \d s.
	\]
	Since $\varphi_{\delta}$ and $\psi_\varepsilon$ 
	have compact support, it is possible to check directly that $\bar{V}^{(n)}_{\varepsilon, \delta}$ 
	is a well-defined, centered Gaussian random field for every $n\ge0$. Also, by \eqref{rho:max},
	\begin{align*}
		&\E \left(\left| \bar{V}^{(n)}_{\varepsilon, \delta}(t\,,x) 
			- \bar{V}^{(n)}_{\varepsilon, \delta}(t\,,x')\right|^2\right)\\
		=& \int_{\R^3} \left[\psi_{\varepsilon}(t-s)\right]^2 \left(\frac{\d^n}{\d x^n}
			\phi_{\delta}(x-y) - \frac{\d^n}{\d (x')^n}\phi_{\delta}(x'-y)\right)\\
		&\hskip2in\times\left(\frac{\d^n}{\d x^n}
			\phi_{\delta}(x-y') - \frac{\d^n}{\d (x')^n}\phi_{\delta}(x'-y')\right)
			\rho(y-y')\, \d y\, \d y' \d s\\
		\leq &C |x-x'|^2\,,
	\end{align*}
	where $C$ is a real number that depends only on $(\varepsilon\,,\delta\,,\rho)$.
	The Kolmogorov continuity theorem implies that, with probability one,
	$\bar{V}^{(n)}_{\varepsilon, \delta}(t)$ is continuous for every $t\ge0$
	[up to a modification, which we always assume]. 
	An  application of the stochastic Fubini's theorem (Theorem \ref{th:stochastic:Fubini})
	now yields the following:
	For every non-random test function $\varphi:\R\to\R$ on $\R$, 	
	\begin{align*}
	&\int_{\R^3} \psi_{\varepsilon}(t-s) \frac{\d^n}{\d x^n} 
		\phi_{\delta}(x-y) V(s\,, y) \,\d y\, \d s\, \varphi(x)\, \d x\\
	&\hskip1in=
		\int_{\R^2} \psi_{\varepsilon}(t-s) \left(\int_{-\infty}^\infty 
		\frac{\d^n}{\d x^n} \phi_{\delta}(x-y) \varphi(x)\, \d x\right) V(s\,, y)\, \d y\, \d s\\
	&\hskip1in=
		(-1)^n\int_{\R^2} \psi_{\varepsilon}(t-s) \left(\int_{-\infty}^\infty  
		\phi_{\delta}(x-y) \frac{\d^n}{\d x^n} \varphi(x)\, \d x\right) V(s\,, y)\, \d y\, \d s\\
	&\hskip1in=
		(-1)^n\int_{-\infty}^\infty \frac{\d^n}{\d x^n} \varphi(x)\,\d x \int_{\R^2}\d y\,\d s\
		\psi_{\varepsilon}(t-s)  \phi_{\delta}(x-y) 
		V(s\,, y) ,
	\end{align*}
	almost surely. It follows 
	that $\bar{V}^{(n)}_{\varepsilon, \delta}(t)$ is a.s.\ the $n$-th order weak derivative 
	of $x\mapsto \bar{V}_{\varepsilon, \delta}(t\,,x)$ for every $\varepsilon,\delta>0$ and $t\ge0$.
	Since $\bar{V}^{(n)}_{\varepsilon, \delta}(t)$ is continuous,
	we may conclude that $\bar{V}^{(n)}_{\varepsilon, \delta}(t)$ is a.s.\ the $n$-th order classical 
	derivative of $\bar{V}_{\varepsilon, \delta}(t)$ for every $t\ge0$. 
	In particular, it follows that $\bar{V}_{\varepsilon, \delta}(t)$ is $C^{\infty}$ a.s. 
	
	One can prove that $\bar{V}_{\varepsilon,\delta}(\cdot\,,x)$ is $C^\infty$ a.s.\
	for every $x\in\R$ using the same sort of argument.
	
	To prove the $L^2(\Omega)$ convergence of 
	$\int_0^t \bar{V}_{\varepsilon, \delta}(s\,, x+ X_{t-s})\d s$, we first note that a.s.,
	\begin{align*}
	&\E \left[\left.\int_0^t \bar{V}_{\varepsilon, \delta}(s\,, x+ X_{t-s})\,\d s \cdot
		\int_0^t \bar{V}_{\varepsilon', \delta'}(s'\,, x+ X_{t-s'})\d s' \ \right|\, X\right]\\
	&=\int_{(0,t)^2}\d s\,\d s' 
		\int_{\R^2}\d y\,\d y'\
		\psi_{\varepsilon}(s-r)\psi_{\varepsilon'}(s'-r) \phi_{\delta}(x+X_{t-s}-y)
		\phi_{\delta'}(x+X_{t-s'}-y')\rho(y-y')\\
	&=\int_{(0,t)^2}\d s\,\d s'\int_{\R^2} \d y\,\d y'\
		\psi_{\varepsilon}(s-r)\psi_{\varepsilon'}(s'-r) 
		\phi_{\delta}(y)\phi_{\delta'}(y')\rho\left(y-y'-X_{t-s}+X_{t-s'}\right).
	\end{align*}
	Since both $\rho$ and $X$ are continuous [see \eqref{rho:max}],
	$(y\,,y',s\,,s')\mapsto
	\rho(y-y'-X_{t-s}+X_{t-s'})$ is continuous a.s., and hence it follows
	from the dominated convergence theorem that the preceding expression
	tends to $t\rho(0)$ as $(\varepsilon\,,\varepsilon',\delta\,,\delta')\to(0\,,0\,,0\,,0)$.
	This implies that $\int_0^t \bar{V}_{\varepsilon, \delta}(s\,, x+ X_{t-s})\,\d s$
	is a Cauchy net in $L^2(\Omega)$, and thus completes the proof. 
\end{proof}

Now that the curvilinear stochastic integral $\int_0^t \bar{V}(s\,,x+X_{t-s})\,\d s$
is defined, we prove that it agrees with $\int_0^t V(s\,,x+X_{t-s})\,\d s$.
In other words, the following result proves that the construction of
$\int_0^t V(s\,,x+X_{t-s})\,\d s$ does not depend on the particular
choise of the heat kernel as the smoother in the definition of
$V_{\varepsilon,\delta}$.

\begin{proposition}\label{pr:V:eps:unique}
	Choose and fix $\phi,\psi$ as in Lemma \ref{lem:V:eps:2},
	and define  for all $\varepsilon,\delta>0$,
	the random field $\bar{V}_{\varepsilon,\delta}$ 
	via \eqref{V:bar}.  Then,
	\[
		\int_0^t\bar{V}(s\,,x+X_{t-s})\,\d s = \int_0^t V(s\,,x+X_{t-s})\,\d s
		\qquad\text{a.s.,}
	\]
	for every $t>0$ and $x\in\R$.
\end{proposition}

\begin{proof}
	Define, for all $\varepsilon,\delta,t>0$, and $x \in\R$,
	\begin{gather}\label{Y:V}
		\mathcal{Y}(t\,,x) := \int_0^t V(s\,,x+X_{t-s})\,\d s,\qquad
		\bar{\mathcal{Y}}(t\,,x) := \int_0^t \bar{V}(s\,,x+X_{t-s})\,\d s,\\\notag
		\mathcal{Y}_{\varepsilon,\delta}(t\,,x) := \int_0^t 
			V_{\varepsilon,\delta}(s\,,x+X_{t-s})\,\d s,\qquad
		\bar{\mathcal{Y}}_{\varepsilon,\delta}(t\,,x) := \int_0^t 
			\bar{V}_{\varepsilon,\delta}(s\,,x+X_{t-s})\,\d s.
	\end{gather}
	By the stochastic Fubini theorem (see
	Theorem \ref{th:stochastic:Fubini}), for all $a,b,\varepsilon,\delta>0$,
	{ \[
		\int_{\R^2}\psi_a(t-t')\phi_b(x-x')
		\mathcal{Y}_{\varepsilon,\delta}(t'\,,x')\,\d t'\,\d x'
		=\int_{\R^2}p_\varepsilon(t-t')p_\delta(x-x')
		\bar{\mathcal{Y}}_{a,b}(t'\,,x')\,\d t'\,\d x'\qquad\text{a.s.}   
	\] 
	where $\psi_a$ and $\phi_b$ are as defined in Lemma \ref{lem:V:eps:2}. }
	Send $(a\,,b)$ to $(0\,,0)$ in order to see that
	\[
		\mathcal{Y}_{\varepsilon,\delta}(t\,,x) = \int_{\R^2}
		{p_\varepsilon(t-t')p_\delta(x-x')\bar{\mathcal{Y}}(t'\,,x')\,\d t'\,\d x'}\qquad\text{a.s.}
	\]
	Send $(\varepsilon\,,\delta)\to(0\,,0)$ in order to finish.
\end{proof}

%
%

\subsection{An infinite-dimensional Brownian motion}\label{subsec:ID:BM}

Curvilinear stochastic integrals of the form \eqref{Y:V} arise frequently in
the study of polymer measures, among other places;
see for example Bertini and Cancrini \cite{BC} and 
Carmona and Molchanov \cite{CM}, together with their volumnous combined references. 
In that theory, $X$ is frequently 
a nice linear diffusion (such as 1-dimensional Brownian motion),
and $\int_0^t V(s\,,x+X_{t-s})\,\d s$ represents the cost of letting
the corresponding time-reversed space-time Brownian motion
to run through an external space-time environment $V$. Perhaps
the simplest example of a curvilinear stochastic integral is 
obtained when we set $X\equiv0$. In that case,
it follows from Lemma \ref{lem:V} that
\(
	(t\,,x)\mapsto \int_0^t V(s\,,x)\,\d s
\)
is a centered Gaussian process with 
\begin{equation}\label{Cov:F}
	\Cov\left[\int_0^t V(s\,,x)\,\d s\ ,\, \int_0^{t'} V(s',x')\,\d s'\right]
	=\min(t\,,t')\rho(x-x').
\end{equation}
In other words, $t\mapsto \int_0^t V(s\,,\cdot)\,\d s$
is a cylindrical Brownian motion with homogeneous spatial correlation
function $\rho$.  

We make two remarks about this Brownian motion next.

\begin{remark}
	It follows easily from the preceding that if $\rho$ is a constant
	[$\rho(x)=\rho(0)$ for all $x\in\R$], then
	\[
		\E\left(\left| \int_0^t V(s\,,x)\,\d s - \int_0^t V(s\,,x')\,\d s\right|^2\right)=0
		\qquad\text{for all $x,x'\in\R$ and $t>0$}.
	\]
	Thus, we may think of $\int_0^t V(s\,,x)\,\d s=\sqrt{\rho(0)}\, W_t$,
	given the representation \eqref{V:W} of the random generalized function $V$.
	
	More generally, a small variation on this argument shows that if $(X\,,Y)$ is independent of $V$
	and $X$ and $Y$ are continuous random processes, then
	for all $t>0$ and $x\in\R$,
	\[
		\int_0^t V(s\,,x+X_{t-s})\,\d s= \int_0^t V(s\,,x+Y_{t-s})\,\d s
		\qquad\text{a.s.,}
	\]
	whenever $\rho$ is a constant. We can set $Y\equiv0$ in order to see that
	when $\rho$ is a constant,
	\[
		\int_0^t V(s\,,x+X_{t-s})\,\d s = \sqrt{\rho(0)}\, W_t\qquad\text{a.s.}
	\]
\end{remark}

\begin{remark}\label{rem:pt:Kraichnan}
	Choose and fix some $t>0$.	
	It is possible to show, using standard methods from Gaussian analysis
	(see for example Adler \cite[Theorem 2.2.2, p.\ 27]{Adler}) 
	that the stationary Gaussian process
	$x\mapsto \mathcal{W}(t\,,x) := \int_0^t V(s\,,x)\,\d s$ has three H\"older-continuous
	derivatives (say in $L^2(\Omega)$) if and only if $\rho\in C^{6+\varepsilon}$
	for some $\varepsilon>0$. In this case, the Kunita theory of stochastic flows
	(see \cite[Ch.\ 6]{Kunita}) implies that the infinite-dimensional Stratonovich
	SDE \eqref{pt:Kraichnan} has a unique solution. We referred to this fact,
	without detailed explanation, in the Introduction.
\end{remark}

\subsection{A probabilistic representation of the solution}

We use the curvilinear stochastic integral of the preceding section 
in order to write the solution to the generalized Kraichnan model \eqref{Kraichnan}
probabilistically in terms of an exogenous Wiener measure. First we introduced
two simple $\sigma$-algebras $\mathcal{V}$ and $\mathcal{T}_0$.

\begin{definition}\label{def:V:T}
	Let $\mathcal{V}$ denote the $\sigma$-algebra generated by all
	random variables of the form
	$\int_{\R_+\times\R}\varphi(t\,,x)V(t\,,x)\,\d t\,\d x,$
	where $\varphi:(0\,,\infty)\times\R\to\R$ is measurable and
	satisfies 
	\[
		\int_0^\infty\d s \int_{-\infty}^\infty\d x  \int_{-\infty}^\infty 
		\d y\ \vert\varphi(s\,,x)\varphi(s\,,y)\vert\, \rho(x-y)<\infty.
	\]
	Also, let $\mathcal{T}_0$ denote the $\sigma$-algebra generated by all random 
	variables of the form $\theta_0(x\,,y)$, where $x$ and $y$ are real numbers.
\end{definition}

Then we have the following probabilistic representation of the solution
to \eqref{Kraichnan}.

\begin{theorem}\label{th:FK}
	Let $B$ and $\bar{B}$ be two independent Brownian motions
	that are totally independent of $\mathcal{V}\vee\mathcal{T}_0$, and
	have respective speeds $\Var (B_1)=2 \nu_1$ and $\Var (\bar{B}_1) =2\kappa$, where
	\begin{equation}\label{kappa}
		\kappa := \nu_2 - \tfrac12\rho(0).
	\end{equation}
	Suppose $\theta_0$ satisfies Assumptions \ref{assum:1}--\ref{assum:4},
	and let $\theta$ denote the solution to \eqref{Kraichnan},
	with parameters $\nu_1$ and $\nu_2$.
	Then, 
	\[
		\theta(t\,,x\,,y) 
		=\E\left[\left. \theta_0\left(x+B_t\,,y+ \bar{B}_t -
		\int_0^t V(s\,,x+B_{t-s})\,\d s\right)\ \right|\ \mathcal{V}\vee\mathcal{T}_0\right].
	\]
\end{theorem}

Theorem \ref{th:FK} follows readily from Proposition \ref{pr:PAM:FK} below
and the inversion theorem of Fourier analysis. We leave the
elementary details of the proof to the interested reader.

\begin{proposition}\label{pr:PAM:FK}
	Let $u[\xi]$ be the solution to
	\eqref{PAM} for every $\xi\in\R$, subject to initial data $\wh\theta_0$,
	where $\theta_0$ satisfies Assumptions \ref{assum:1}--\ref{assum:4}. 
	Then, for all $t\ge0$ and $x,\xi\in\R$,
	\begin{equation}\label{eq:PAM:FK}
		u(t\,,x\,,\xi) = \e^{t\xi^2\rho(0)/2}\,
		\E\left[\left. \wh{\theta}_0(x+B_t\,,\xi)\exp\left( i\xi\int_0^t V(s\,, x+B_{t-s})\,\d s\right)
		\ \right|\ \mathcal{V}\vee\mathcal{T}_0\right]\,,
	\end{equation}
	where $B$ is a Brownian motion independent of
	$\mathcal{V}\vee\mathcal{T}_0$ with $\Var (B_1)=2 \nu_1$. 
\end{proposition}

\begin{proof}
	We follow the argument of Hu and Nualart \cite{HuNualart} closely, making adjustments to account
	for the present, slightly different, setting. 
	
	In order to simplify the typsetting we will consider only the case that
	$\theta_0$---hence also $u_0$---is non random. To obtain the general case
	from this one, one simply replaces all of the following expectation operators
	by conditional expectation operators, given $\mathcal{T}_0$, without
	altering the course of the proof.
	
	Define  $v(t\,,x\,,\xi)$ to be the quantity on the right-hand side of \eqref{eq:PAM:FK};
	that is,
	\[
		v(t\,,x\,,\xi) = \e^{t\xi^2\rho(0)/2}\,
		\E\left[\left. \wh{\theta}_0(x+B_t\,,\xi)
		\exp\left( i\xi\int_0^t V(s\,, x+B_{t-s})\,\d s\right)
		\ \right|\ \mathcal{V}\right].
	\]
	We are going to show that $v(t\,,x\,,\xi)$ solves the SPDE \eqref{PAM}
	in  mild form \eqref{mild:u}. This and the uniqueness of the solution
	to \eqref{PAM} [Theorem \ref{th:u:exist}] together will
	imply that $v(t\,,x\,,\xi)$ and $u(t\,,x\,,\xi)$
	coincide almost surely for all $t\ge0$ and $x,\xi\in\R$, and complete the proof.
	
	To this end, recall the space $\mathcal{H}$ from \eqref{H} in the appendix.
	Since $\rho$ is positive definite \emph{a priori}, it follows that $\|\,\cdots\|_{\mathcal{H}}$
	is indeed a Hilbert norm, with corresponding inner product
	\[
		\<\varphi_1\,,\varphi_2\>_{\mathcal{H}} =
		\int_0^{\infty}\d t\int_{-\infty}^\infty\d x \int_{-\infty}^\infty\d y\ 
		\varphi_1(t\,, x) \varphi_2(t\,,y)  \rho(x-y),
	\]
	for every $\varphi_1,\varphi_2\in C_c^{\infty}((0\,,\infty)\times\R).$
	And of course $\mathcal{H}$ is a Hilbert space, once endowed with the latter
	inner product.
	
	Define
	\[
		S_{t,x}(\varphi) := \E \left[ v(t\,,x\,,\xi) F_{\varphi} \right]
		\qquad\text{for every $\varphi \in \mathcal{H}$, $t\ge0$, and $x\in\R$},
	\]
	where 
	\[
		F_{\varphi} := \exp \left(\int_{\R_+\times\R}
		\varphi(t\,,x)V(t\,,x)\,\d t\,\d x - \tfrac12 
		\|\varphi\|_{\mathcal{H}}^2 \right).
	\]
	Thanks to the construction of $\int_0^t V(s\,, x+B_{t-s})\,\d s$ (see Lemma \ref{lem:V}
	and its proof), 
	\[
		\E \left(\int_0^t V(s\,, x+B_{t-s})\,\d s \times 
		\int_{\R_+\times\R}  \varphi(s\,,y)V(s\,,y)\,\d s\,\d y \ \bigg | \ B\right)
		=\int_0^t \left( \varphi(s)*\rho\right)(B_{t-s}+x)\,\d s.
	\]
	Therefore, we may first condition on $B$, and then use the fact that $V$
	is Gaussian, in order to deduce from \eqref{u_0:T_0} that
	\[
		S_{t,x}(\varphi)
		= \E \left[u_0(x+B_t\,, \xi)\exp \left(i \xi \int_0^t 
		\left( \varphi(s)*\rho\right))(B_{t-s}+x)\,\d s
		\right)\right].
	\]
	By the classical Feynman--Kac formula for deterministic PDEs,
	the function $(t\,,x)\mapsto S_{t,x}(\varphi)$ is the unique solution to the
	diffusion equation,
	\[
		\partial_t S_{t,x}(\varphi) 
		= \nu_1 \partial^2_x S_{t,x}(\varphi) 
		+ i \xi S_{t,x}(\varphi)  \cdot (\varphi(t)*\rho)(x),
	\]
	with initial profile $u_0(\cdot\,,\xi)$. In particular,
	the Duhamel principle yields
	\[
		S_{t,x}(\varphi) = \left(p_t^{(\nu_1)}*u_0\right)(x\,, \xi) 
		+ i\xi \int_0^t\d s \int_{-\infty}^\infty\d y\ p_{t-s}^{(\nu_1)}(x-y)
		S_{s,y}(\varphi) \cdot (\varphi(s)*\rho)(y).
	\]
	Let $D$ denote the Malliavin derivative that corresponds to the
	infinite-dimensional Brownian motion $t\mapsto\int_0^t V(s\,,\cdot)\,\d s$
	(see Nualart \cite{Nualart}).  It is well known that $DF_\varphi = \varphi F_\varphi$ a.s.
	(see Nualart \cite{Nualart}). Consequently, Fubini's theorem and the integration by parts formula of
	Malliavin calculus (see Nualart \cite{Nualart}) together imply that 
	\begin{align*}
		&\E \left[ v(t\,,x\,, \xi)\cdot F_{\varphi}\right]\\ 
		&= \left(p_t^{(\nu_1)}* u_0\right)(x\,,\xi) + i\xi \E
			\left[ \int_0^t \d s\int_{-\infty}^\infty\d y\
			p_{t-s}^{(\nu_1)}(x-y) v(s\,,y\,,\xi) 
			F_{\varphi} \cdot (\varphi(s)*\rho)(y)\right]\\
		&= \left(p_t^{(\nu_1)}* u_0\right) (x\,,\xi) 
		   	+ i \xi \E\left[
			\left\langle p_{t-\cdot}^{(\nu_1)}(x-\cdot)
			v[\xi] \,, D F_{\varphi}
			\right\rangle_{\mathcal{H}}
			\right].
	\end{align*}
	Because our noise $V$ is white in time, 
	the adjoint [divergence] of the operator $D$, acting on predictable random
	field $X$, is simply the Walsh integral of that random field $X$ (see Nualart \cite{Nualart}).
	Therefore, it follows that
	\[
		\E \left[ v(t\,,x\,, \xi)\cdot F_{\varphi}\right]
		= \left( p_t^{(\nu_1)}* u_0 \right) (x\,,\xi) 
		+ i \xi \E \left[\int_{(0,t)\times\R}
		p_{t-s}^{(\nu_1)}(x-y) v(s\,,y\,,\xi)V(s\,,y) \,\d s\,\d y\cdot  F_{\varphi} \right].
	\]
	Because the family $\{ F_\varphi\}_{\varphi\in\mathcal{H}}$ is total in 
	$L^2(\Omega\,,\mathcal{V},\P)$ (see Nualart \cite{Nualart}), it follows from the elementary properties of
	conditional expectations that
	$v(t\,,x\,,\xi)$ solves \eqref{PAM}. This is what we had set out to prove.
\end{proof}

\begin{proof}[Proof of Theorem \ref{th:FK}]
	We now compare \eqref{eq:Uu} with
	Proposition \ref{pr:PAM:FK}, and recall \eqref{kappa}, in order to see that
	\begin{align*}
		U(t\,,x\,,\xi) &:= \e^{-\nu_2\xi^2 t} u(t\,,x\,,\xi)\\
		&\ = \e^{-t\xi^2\kappa}\,
			\E\left[\left. \wh{\theta}_0(x+B_t\,,\xi)\exp\left( i\xi\int_0^t V(s\,, x+B_{t-s})\,\d s\right)
			\ \right|\ \mathcal{V}\vee\mathcal{T}_0\right].
	\end{align*}
	As a consequence of this formula, and thanks to the definition of $\kappa$
	(see \eqref{kappa}), we readily obtain the bound,
	\[
		\|U(t\,,x\,,\xi)\|_1\le\exp\left(-t\xi^2\left[\nu_2-\frac{\rho(0)}{2}\right]\right)
		\left\| \wh{\theta}_0(x+B_t\,,\xi)\right\|_1.
	\]
	Therefore, the following random field---defined earlier in \eqref{T:U}---is well defined
	\[
		\theta(t\,,x\,,y) := \frac{1}{2\pi}\int_{-\infty}^\infty\e^{-iy\xi}\, U(t\,,x\,,\xi)\,\d\xi.
	\]
	Moreover, because of Assumption \ref{assum:1} of the Introduction,
	\begin{equation}\label{T:L1}
		\|\theta(t\,,x\,,y)\|_1 \le \frac{1}{\pi}\int_{-\infty}^\infty
		\exp\left(-t\xi^2\left[\nu_2-\frac{\rho(0)}{2}\right]\right)\| \wh{\theta}_0(x+B_t\,,\xi)\|_1
		\,\d\xi<\infty,
	\end{equation}
	thanks, additionally, to the fact
	that because of the independence of $\theta_0$ and $B$,
	\[
		\sup_{\xi\in\R}\|\wh{\theta}_0(x+B_t\,,\xi)\|_1\le \sup_{a\in\R}\int_{-\infty}^\infty
		\|\theta_0(a\,,b)\|_1\,\d b<\infty.
	\]
	Now, the inversion formula of Fourier transforms ensures that
	\[
		\theta(t\,,x\,,y)
		= \frac{1}{2\pi}\int_{-\infty}^\infty \e^{-iy\xi -\kappa\xi^2t}
		\E\left( \left. \e^{i\xi\mathcal{Y}(t\,,x)}\, \wh{\theta}_0(x+B_t\,,\xi) 
		\ \right|\ \mathcal{V}\vee\mathcal{T}_0\right)\d\xi
		\qquad\text{a.s.,}
	\]
	where $\mathcal{Y}$ was defined in \eqref{Y:V}, with $X$ replaced by the
	Brownian motion $B$.
	The validity of the absolute integrability condition \eqref{T:L1} ensures that Fubini's
	theorem is applicable {(in \eqref{T:L1} we can replace $\|\hat{\theta}_0(x+B_t\,, \xi)\|_1$ by $\|\hat{\theta}_0(x+B_t\,, \xi)\|_2$ to check that stochastic Fubini (Theorem \ref{th:stochastic:Fubini}) is applicable}) and yields
	\begin{align*}
		\theta(t\,,x\,,y) &= \E\left[\left.\frac{1}{2\pi}\int_{-\infty}^\infty 
			\e^{-i\xi(y -\mathcal{Y}(t,x)) - \kappa\xi^2t }\,\wh{\theta}_0(x+B_t\,,\xi)\,
			\d\xi\ \right|\, \mathcal{V}\vee\mathcal{T}_0\right]\\
		&=\E\left[\left.
			\left( \theta_0(x+B_t\,,\cdot)* p_{(\kappa/\nu)t}^{(\nu)}\right)( y -
			\mathcal{Y}(t\,,x))\ \right|\ \mathcal{V}\vee\mathcal{T}_0\right].
	\end{align*}
	This is equivalent to the assertion of Theorem \ref{th:FK}.
\end{proof}

\subsection{Proof of Theorem \ref{th:cont:U}}

In the previous subsections
we  introduced some of the ingredients of the proof of Theorem \ref{th:cont:U}.
We are now ready to establish Theorem \ref{th:cont:U}. Throughout, we assume
the hypotheses of Theorem \ref{th:cont:U}.

Define the random fields $V_{\varepsilon,\delta}$
and $\mathcal{Y}$ respectively by \eqref{V:eps} and \eqref{Y:V},
so that
\[
	\theta(t\,,x\,,y) 
	=\E\left[\left. \theta_0\left(x+B_t\,,y+ \bar{B}_t -
	\mathcal{Y}(t\,,x)\right)\ \right|\ \mathcal{V}\vee\mathcal{T}_0\right]
	\qquad\text{a.s.},
\]
for all $t>0$ and $x,y\in\R$. The following is a first step toward estimating
the smoothness properties of the random field $\theta$.

Recall the random field $\mathcal{Y}$ from \eqref{Y:V}.

\begin{lemma}\label{lem:Y-Y}
	If \eqref{rho:cont} holds for some $C_*>0$,
	then for every $k\ge 2$, there exists a real number 
	$C_*(k)=C_*(k\,,\rho(0))$
	such that 
	\[
		\E\left( |\mathcal{Y}(t+h\,,x)-\mathcal{Y}(t\,,x')|^k\right)
		\le C_*(k)t^{k/2} \left\{ |x-x'|^{k\varpi/2} + h^{k\varpi/4}\right\},
	\]
	uniformly for every $t>0$, $x\in\R$, $x'\in[x-1\,,x+1]$, and $h\in(0\,,1)$.
\end{lemma}

\begin{proof}
	Thanks to \eqref{Cov:V:eps}, 
	\begin{align*}
		&\Cov\left[ \left.\int_0^{t+h} V_{\varepsilon,\delta}(s\,,x+B_{t+h}-B_s)\,\d s
			\ , \int_0^t V_{\varepsilon,\delta}(s',x'+B_t-B_{s'})\,\d s'\ \right|\ B\right]\\
		&\hskip1.4in=\int_0^{t+h}\d s\int_0^t\d s'\
			p_{2\varepsilon}(s-s') \cdot \left( p_{2\delta} *\rho\right)
			\left(x-x' + B_{t+h}-B_t + B_{s'} - B_s\right),
	\end{align*}
	almost surely for every $\varepsilon,\delta,t,h>0$
	and $x,x'\in\R$.
	Let $\varepsilon$ and $\delta$  both tend to zero and appeal to Lemma \ref{lem:V}
	to see that
	\[
		\Cov\left[ \left.\mathcal{Y}(t+h\,,x) \,, \mathcal{Y}(t\,,x')\ \right|\ B\right]
		= t \rho\left( x-x' + B_{t+h} - B_t \right)
		\qquad\text{a.s.}
	\]
	In particular,
	\[
		\E\left(\left. \ \left| \mathcal{Y}(t+h\,,x) - \mathcal{Y}(t\,,x')\right|^2\
		\right|\  B\right)
		= 2t\left[ \rho(0) - \rho\left( x-x' + B_{t+h} - B_t \right) \right]
		\qquad\text{a.s.}
	\]
	Since the conditional law of $\mathcal{Y}$ given $B$ is Gaussian,
	elementary properties of mean-zero Gaussian processes tell us
	that for all $k\ge 2$, $t,h>0$, and $x,x'\in\R$,
	\[
		\E\left(\left. \ \left| \mathcal{Y}(t+h\,,x) - \mathcal{Y}(t\,,x')\right|^k\
		\right|\  B\right) = \frac{(4t)^{k/2}}{\sqrt\pi}\Gamma\left( 
		\frac{k+1}{2}\right)\cdot\left[ \rho(0) - 
		\rho\left( x-x' + B_{t+h} - B_t \right) \right]^{k/2},
	\]
	almost surely.
	If $k\ge 2$, $t>0$, $x,x'\in\R$, and $0<h<1$, 
	then by \eqref{rho:cont} and Brownian scaling,
	\[
		\E\left( \left| \mathcal{Y}(t+h\,,x) - \mathcal{Y}(t\,,x')\right|^k\right) 
		\le\frac{(4C_*t)^{k/2}}{\sqrt\pi}\,\Gamma\left( 
		\frac{k+1}{2}\right)\cdot
		\E\left(\left| x-x' + \sqrt h\, B_1 \right|^{k\varpi/2}\right).
	\]
	It follows easily from this that
	\begin{align*}
		\E\left( \left| \mathcal{Y}(t+h\,,x) - \mathcal{Y}(t\,,x')\right|^k\right) 
		\le \wt{a}_k t^{k/2}\cdot\left\{ |x-x'|^{k\varpi/2} + h^{k\varpi/4}
		\E\left( |B_1|^{k\varpi/2}\right)\right\},
	\end{align*}
	a.s., where
	\[
		\wt{a}_k := \frac{2^{k/2}\cdot(4C_*)^{k/2}}{\sqrt\pi}\,\Gamma\left( 
		\frac{k+1}{2}\right).
	\]
	The result follows.
\end{proof}

We are ready for the following.

\begin{proof}[Proof of Theorem \ref{th:cont:U}]
	By the probabilistic representation 
	of the solution (see Theorem \ref{th:FK}) to see that for all $t>0$, $x,x',y\in\R$,
	and $k\ge 2$,
	\begin{align*}
		&\|\theta(t\,,x\,,y) - \theta(t\,,x',y)\|_k\\
		&\qquad\le\left\| \theta_0\left( x+B_t\,,y+\bar{B}_t-\mathcal{Y}(t\,,x)\right) - 
			\theta_0\left( x'+B_t\,,y+\bar{B}_t-\mathcal{Y}(t\,,x)\right) \right\|_k\\
		&\hskip1in + \left\| \theta_0\left( x'+B_t\,,y+\bar{B}_t-\mathcal{Y}(t\,,x)\right) - 
			\theta_0\left( x'+B_t\,,y+\bar{B}_t-\mathcal{Y}(t\,,x')\right) \right\|_k\\
		&\qquad\le\wt{A}_k\left\{ |x-x'|^{\alpha} 
			+ \left\| \mathcal{Y}(t\,,x) - \mathcal{Y}(t\,,x')\right\|_k^{\zeta}\right\},
	\end{align*}
	thanks to  \eqref{T_0-T_0}
	and the conditional form of the Jensen's inequality. Therefore, Lemma \ref{lem:Y-Y}
	yields
	\begin{equation}\label{T-T:x}
		\|\theta(t\,,x\,,y) - \theta(t\,,x',y)\|_k
		\le \wt{A}_k |x-x'|^{\alpha} + [C_*(k)]^{\zeta} 
		t^{\zeta/2} |x-x'|^{\zeta\varpi/2}.
	\end{equation}
	Similarly, for all $t,h>0$, $x,y\in\R$, and $k\ge2$,
	\begin{align}\notag
		&\|\theta(t+h\,,x\,,y) - \theta(t\,,x\,,y)\|_k\\\notag
		&\qquad\le\left\| \theta_0\left( x+B_{t+h}\,,y+\bar{B}_{t+h}
			-\mathcal{Y}(t+h\,,x)\right) - 
			\theta_0\left( x+B_t\,,y+\bar{B}_{t+h}
			-\mathcal{Y}(t+h\,,x)\right) \right\|_k\\\notag
		&\hskip1in + \left\| \theta_0\left( x+B_t\,,y+\bar{B}_{t+h}-\mathcal{Y}(t+h\,,x)\right) 
			-\theta_0\left( x+B_t\,,y+\bar{B}_t-\mathcal{Y}(t\,,x)\right) \right\|_k\\\notag
		&\qquad\le\wt{A}_k\left\{ \| B_{t+h}-B_t\|_k^{\alpha} + 
			\|\mathcal{Y}(t+h\,,x) - \mathcal{Y}(t\,,x) \|_k^{\zeta}
			+ \|\bar{B}_{t+h}-\bar{B}_t\|_k^{\zeta} \right\}\\
		&\qquad\le \wt{A}_k \|B_1\|_k^{\alpha}\, h^{{\alpha}/2} 
			+ [C_*(k)]^{{\zeta}} t^{{\zeta}/2} h^{{\zeta\varpi}/4}
			+ \wt{A}_k\|\bar{B_1}\|_k^{\alpha}h^{\zeta/2}.
			\label{T-T:t}
	\end{align}
	Finally, for every $t>0$, $x,y,y'\in\R$, and $k\ge 2$,
	\begin{align}\notag
		&\|\theta(t\,,x\,,y) - \theta(t\,,x\,,y')\|_k\\\notag
		&\hskip1in=\left\| \theta_0\left( x+B_{t+h}\,,y+\bar{B}_t-\mathcal{Y}(t\,,x)\right) - 
			\theta_0\left( x+B_t\,,y'+\bar{B}_t-\mathcal{Y}(t\,,x)\right) \right\|_k^\zeta\\
		&\hskip1in\le\wt{A}_k|y-y'|^\zeta.
		\label{T-T:y}
	\end{align}
	The a.s.-smoothness of $y\mapsto \theta(t\,,x\,,y)$ was established
	in Theorem \ref{th:exist:U}. Therefore,
	\eqref{T-T:x},  \eqref{T-T:t}, and \eqref{T-T:y} together imply the result,
	thanks to a suitable version of the Kolmogorov continuity theorem.
\end{proof}

\section{Proofs of Propositions \ref{pr:diss}, \ref{pr:pos},  \ref{pr:mass}, and \ref{pr:comparison}}

Propositions \ref{pr:diss}, \ref{pr:pos}, and \ref{pr:mass} are
relatively simple measure-theoretic consequences of  Theorem \ref{th:FK}.
We verify those propositions in order.

\begin{proof}[Proof of Proposition \ref{pr:diss}]
	According to Theorem \ref{th:FK}, for every $(t\,,x\,,y)\in(0\,,\infty)\times\R^2$,
	\begin{equation}\label{T:T}
		\theta(t\,,x\,,y) = \E\left[\left.
		\left( \theta_0(x+B_t\,,\cdot)* p_{(\kappa/\nu)t}^{(\nu)}\right)( y -
		\mathcal{Y}(t\,,x))\ \right|\ \mathcal{V}\vee\mathcal{T}_0\right] \qquad\text{a.s.,}
	\end{equation}
	where the curvilinear stochastic integral $\mathcal{Y}$ was defined in \eqref{Y:V}.
	Both sides are continuous [up to a modification], thanks to Theorem \ref{th:cont:U}.
	Therefore, we may appeal to the continuous modification instead to see that
	the preceding identity holds for all $(t\,,x\,,y)\in(0\,,\infty)\times\R^2$ outside 
	a single $\P$-null set.  Because $z\mapsto p_{(\kappa/\nu)t}^{(\nu)}(z)$ is maximized
	at $z=0$ and the maximum value is $(2\pi\kappa t)^{-1/2}$,
	\[
		\sup_{|x|\le m}\sup_{y\in\R}|\theta(t\,,x\,,y)| \le 
		\frac{1}{\sqrt{2\pi\kappa t}}\,
		\sup_{\substack{x\in\Q\\|x|\le m}}
		\int_{-\infty}^\infty |\theta_0(x+B_t\,,w)|\,\d w.
	\]
	The triangle inequality and Assumption \ref{assum:3} together imply
	that
	\[
		\int_{-\infty}^\infty \Big\| |\theta_0(x\,,y)| - |\theta_0(x',y)| \Big\|_k\,\d y
		\le \int_{-\infty}^\infty\| \theta_0(x\,,y) - \theta_0(x',y)\|_k\,\d y
		\le C_0|x-x'|^\alpha,
	\]
	for all $x,x'\in\R$ and $k\ge 2$. Therefore, the Kolmogorov continuity theorem
	ensures that $a\mapsto \int_{-\infty}^\infty |\theta_0(a\,,w)|\,\d w$ has
	a continuous modification, whence 
	\[
		{\sup_{{|x|\le m}}}\int_{-\infty}^\infty
		|\theta_0(x\,,w)|\,\d w<\infty,
	\]
	almost surely for all $m>0$. The result follows.
\end{proof}

\begin{remark}\label{rem:diss}
	We pause to prove an assertion that was made in the Introduction.
	Namely, that the dissipation rate in \eqref{L2:diss} is unimproveable.
	Consider
	\[
		\theta_0(x\,,y) = p_1^{(\kappa)}(y)\text{ \ and \ }\rho(x)=1
		\qquad\text{for all $x,y\in\R$},
	\]
	where $\kappa$ was defined in \eqref{kappa}. Recall that,
	because $\rho$ is a constant, we can write
	$\int_0^t V(s\,,x+B_{t-s})\,\d s =W_t$ where $W$ is
	a standard Brownian motion (see \S\ref{subsec:ID:BM}). 
	Therefore, Theorem \ref{th:FK} and
	the semigroup properties of $p^{(\kappa)}$ together yield
	that $\theta(t\,,x\,,y) = p^{(\kappa)}_{1+t}( y - W_t)$
	for all $t>0$ and $x,y\in\R$.
	It follows immediately from this that $\theta(t)\ge0$ and
	\[
		\sup_{y\in\R} \theta(t\,,x\,,y)  = \frac{1}{\sqrt{4\pi\kappa(1+t)}}.
	\]
	In particular, Proposition \ref{pr:diss} guarantees an uppper bound
	on the dissipation rate of the passive scalar that 
	is unimproveable.
\end{remark}

\begin{proof}[Proof of Proposition \ref{pr:pos}]
	Let $\mathbb{W}$ denote the law of the process $B$. We may view $\mathbb{W}$ as a 
	probability measure on the usual space $C[0\,,\infty)$ of real-valued,
	continuous functions on $[0\,,\infty)$.
	Theorem \ref{th:FK} and Fubini's theorem together imply that
	\begin{align*}
		&\theta(t\,,x\,,y) \\
		&=\frac{1}{\sqrt{2\pi\kappa t}}\int_{C[0,\infty)}\mathbb{W}(\d f)
			\int_{-\infty}^\infty\d w\  \theta_0(x+f(t)\,,w)
			\exp\left( -\frac{\left[w-y+\int_0^t V(s\,, x+f(t-s))\,\d s
			\right]^2}{2\kappa t}\right),
	\end{align*}
	almost surely. Of course,  $\int_0^t V(s\,, x+f(t-s))\,\d s$
	is not defined for every $f\in C[0\,,\infty)$. But it is well defined
	for $\mathbb{W}$-almost every $f\in C[0\,,\infty)$ by Lemma \ref{lem:V}.
	
	It follows essentially immediately from the preceding display that
	for every $t>0$ and $x,y\in\R$, $\P\{\theta(t\,,x\,,y)>0\}=1$. This is however
	a weaker statement than the one that was announced in Proposition
	\ref{pr:pos}. [N.B. The quantifiers.] In order to prove the full result we need
	to pay attention to a few measure-theoretic details.
	
	According to Theorem \ref{th:cont:U}, both sides of the preceding display are
	continuous up to a modification. Therefore, we may replace each
	side with its continuous modification as is usual to see that the preceding
	identity holds for all $t>0$ and $x,y\in\R$ outside a single $\P$-null
	set. In particular, outside a single null set,
	if $\theta(t\,,x\,,y)=0$ for some $(t\,,x\,,y)\in(0\,,\infty)\times\R^2$,
	then $\theta_0(x+f(t)\,,b)=0$ for almost every $(b\,,f)\in\R\times C[0\,,\infty)$
	[for that same fixed triple $(t\,,x\,,y)$]. 
	
	Now, suppose to the contrary that 
	$\theta(t\,,x\,,y)=0$ for some $(t\,,x\,,y)\in(0\,,\infty)\times\R^2$.
	If this were so, then the preceding discussion and Fubini's theorem together show that 
	\begin{equation}\label{W=1}
		\mathbb{W}\{ f\in C[0\,,\infty):\ \theta_0(x+f(t)\,,b)=0\}=1
		\qquad\text{for a.e.\ $b\in\R$.}
	\end{equation}
	Fix any such $b\in\R$ and observe that
	$\mathcal{Z}(b):=\{a\in\R:\, (a\,,b)=0\}$ is a Lebesgue-null set, by Fubini's theorem. 
	Since the distribution of $B_t$ is mutually absolutely continuous
	with respect to the Lebesgue measure, it follows that 
	$\mathbb{W}\{f:\,x+f(t)\in\mathcal{Z}(b)\}=0$,
	and hence $\mathbb{W}\{f:\, \theta_0(x+f(t)\,,b)=0\}=0$. This contradicts \eqref{W=1}.
\end{proof}

\begin{proof}[Proof of Proposition \ref{pr:mass}]
	Simply integrate both sides of \eqref{T:T} with respect to $y$,
	using Fubini's theorem.
\end{proof}

\begin{proof}[Proof of Proposition \ref{pr:comparison}]
	Just as \eqref{T:T} is valid for all $t\ge0$ and $x,y\in\R$ off a single
	$\P$-null set, so is the following:
	\[
		\wt{\theta}(t\,,x\,,y) = \E\left[\left.
		\left( \wt{\theta}_0(x+B_t\,,\cdot)* p_{(\kappa/\nu)t}^{(\nu)}\right)( y -
		\mathcal{Y}(t\,,x))\ \right|\ \mathcal{V}\vee\wt{\mathcal{T}}_0\right] \qquad\text{a.s.},
	\]
	where $\wt{\mathcal{T}}_0$ is defined just as $\mathcal{T}_0$ was,
	but with $\wt{\theta}_0$ in place of $\theta_0$ everywhere (see Definition
	\ref{def:V:T}).
	One just writes out the right-hand sides of the above and \eqref{T:T}
	as integrals against Wiener measure (see the proof of Proposition \ref{pr:diss})
	to deduce the result.
\end{proof}

\section{The Stratonovich  equation}

Our efforts, thus far, have produced an It\^o/Walsh type solution
to the generalized Kraichnan model \eqref{Kraichnan} (see Theorem \ref{th:exist:U})
which frequently has good local regularity properties (see Theorem \ref{th:cont:U}).
As was pointed out in the Introduction, a drawback of this construction is that
it works only when $\nu_2>\frac12\rho(0)$. Next we study ``the Stratonovich solution''
to \eqref{Kraichnan}. As a by-product of our construction it will follow that
the Stratonovich solution to \eqref{Kraichnan} exists for all possible
choices of $\nu_1,\nu_2>0$. We construct our ``Stratonovich solution'' directly, 
using an old idea of Wong and Zakai \cite{WongZakai}. See also 
McShane \cite{McShane} and
Ikeda, Nakao, and Yamato \cite{INY} for some closely-related results.

\subsection{On the Wong--Zakai theorem}

Before we discuss the Stratonovich solution to the Kraichnan model
\eqref{pre:Kraichnan}, we would like to say a few things about the classical Wong--Zakai theory for
one-dimensional diffusions \cite{WongZakai}. This material is in many
ways classical. Still, we feel that the following viewpoint
might be of some interest, and so include it here. It is easy to
make rigorous the material that follows in any case. See Friz and Hairer \cite{FrizHairer}
for the rigorous details told in a modern setting, and Hairer and Pardoux \cite{HairerPardoux}
for a recent rigorous version of this argument,
in a highly non-trivial, infinite-dimensional setting. 

Let $\{W(t)\}_{t\ge0}$ be a standard Brownian motion and 
$\phi$ be a smooth and bounded probability density function. Set $\phi_\varepsilon(x) := \varepsilon^{-1}
\phi(x/\varepsilon)$ for all $\varepsilon>0$ and $x\in\R$
and  $W_\varepsilon(0):=0$
for all $\varepsilon>0$ in order to see that
\[
	W_\varepsilon(t) := (W*\phi_\varepsilon)(t) = 
	{\int_{0}^\infty}\phi_\varepsilon(t-s)W(s)\,\d s
	\qquad[t\ge0]
\]
defines a smooth Gaussian process for every $\varepsilon>0$. Consider the 
random ODE,
\begin{equation}\label{dX=dW}
	\d X_\varepsilon(t) =\sigma(X_\varepsilon(t))\,\d W_\varepsilon(t)
	\qquad[t>0],
\end{equation}
which, classical theory ensures,
has a unique solution for every $\varepsilon>0$ as long as $\sigma$
is sufficiently smooth. Because $X_\varepsilon(t+h) - X_\varepsilon(t) = \int_t^{t+h} 
\sigma(X_\varepsilon(s))\,\d W_\varepsilon(s)$, 
we can Taylor expand  $s\mapsto\sigma(X_\varepsilon(s))$ for $s\approx t$ 
in order to see that if $h\approx0$, then
\[
	X_\varepsilon(t+h) - X_\varepsilon(t)  \approx
	\sigma(X_\varepsilon(t))(\nabla_h W_\varepsilon)(t)
	+ \sigma'(X_\varepsilon(t))
	\int_t^{t+h}(\nabla_{s-t} X_\varepsilon)(t)\,
	\d W_\varepsilon(s),
\]
where $(\nabla_h f)(t) := f(t+h) - f(t)$.
We glean from the above also that $(\nabla_{s-t}X_\varepsilon)(t)
\approx\sigma(X_\varepsilon(t))(\nabla_{s-t}W_\varepsilon)(t)$, to
leading order, and hence
\begin{align*}
	X_\varepsilon(t+h) - X_\varepsilon(t) 
		&\approx
		\sigma(X_\varepsilon(t))(\nabla_h W_\varepsilon)(t)
		+ \sigma'(X_\varepsilon(t))
		\sigma(X_\varepsilon(t))\int_t^{t+h}(\nabla_{s-t}W_\varepsilon)(t)\,
		\d W_\varepsilon(s)\\
	&=\sigma(X_\varepsilon(t))(\nabla_h W_\varepsilon)(t) + \tfrac12
		\sigma(X_\varepsilon(t))\sigma'(X_\varepsilon(t))
		\left[(\nabla_h W_\varepsilon)(t)\right]^2,
\end{align*}
after a line, or two, of elementary calculus.
If the preceding approximation were of sufficiently high quality
(it is!), then we would be able to write, for $n\gg1$ large but fixed,
\begin{align*}
	X_\varepsilon(t) - X_\varepsilon(0) &\approx\sum_{0\le j\le nt}
		(\nabla_{1/n}X_\varepsilon)(j/n)\\
	&\approx \sum_{0\le j\le nt}
		\sigma(X_\varepsilon(j/n))(\nabla_{1/n}W_\varepsilon)(j/n)
		+\tfrac12\sum_{0\le j\le nt}(\sigma\sigma')(X_\varepsilon(j/n))
		\left[(\nabla_{1/n}W_\varepsilon)(j/n)\right]^2.
\end{align*}
In particular, if $X(t) := \lim_{\varepsilon\downarrow0} X_\varepsilon(t)$  existed (it does!),
then simple continuity considerations imply that $X$ would have to satisfy
\[
	X(t) - X(0) \approx \sum_{0\le j\le nt}
	\sigma(X(j/n))(\nabla_{1/n}W)(j/n)
	+\tfrac12\sum_{0\le j\le nt}(\sigma\sigma')(X(j/n))
	\left[ (\nabla_{1/n}W)(j/n)\right]^2,
\]
provided only that $n\gg1$. Let $n\to\infty$ and appeal to elementary
properties of the It\^o integral in order to conclude that $X$ must then solve the
It\^o stochastic differential equation,
\[
	\d X(t) = \sigma(X(t))\,\d W(t) + \tfrac12(\sigma\sigma')(X(t))\,\d t.
\]
This is essentially the Wong and Zakai theorem \cite{WongZakai}. A somewhat surprising
feature of that theorem is that it implies among other things that
the limit $X$ of $X_\varepsilon$
does \emph{not} satisfy the It\^o SDE $\d X = \sigma(X)\,\d W$, as one might 
guess from a first look at \eqref{dX=dW}. Rather, $X$ solves a Stratonovich SDE:
The stochastic integral
\[
	\int_0^t \sigma(X(s))\circ\d W(s)
	:= \int_0^t\sigma(X(s))\,\d W(s) + \frac12\int_0^t(\sigma\sigma')(X(s))\,\d s
\]
is the Stratonovich stochastic integral of $\sigma(X)$, and the Wong--Zakai theorem
implies that a ``physical approximation'' to a stochastic differential equation should
typically be understood as a Stratonovich SDE (and not an It\^o SDE).
Armed with this philosophy we next turn to ``physical approximations''
of the Kraichnan model \eqref{pre:Kraichnan}.

\subsection{A Wong--Zakai theory for the Kraichnan model}

Let $V_{\varepsilon,\delta}$ denote an $(\varepsilon\,,\delta)$-smoothing of
the noise model $V$, as was done in \eqref{V:eps}, and consider the following
smoothed version of \eqref{pre:Kraichnan}:
\begin{equation}\label{pre:Kraichnan:smooth}
	\partial_t \theta_{\varepsilon,\delta}(t\,,x\,,y) = 
	\nu\Delta\theta_{\varepsilon,\delta}(t\,,x\,,y) + 
	\partial_y \theta_{\varepsilon,\delta}
	(t\,,x\,,y)\,
	V_{\varepsilon,\delta}(t\,,x),
\end{equation}
on $(0\,,\infty)\times\R^2$, subject to initial data $\theta_{\varepsilon,\delta}(0\,,x\,,y)
=\theta_0(x\,,y)$. Since  $V_{\varepsilon,\delta}$ is a.s.\ smooth
(see Proposition \ref{pr:V:eps}),
\eqref{pre:Kraichnan:smooth} is a random, second-order PDE with smooth
coefficients and hence has a unique classical solution $\theta_{\varepsilon,\delta}$ a.s.
Motivated by the material of the previous section, we may make the following definition.

\begin{definition}\label{def:Stratonovich sol}
	We say that $\theta=\theta(t\,,x\,,y)$ is the \emph{Stratonovich solution
	to \eqref{pre:Kraichnan}} if $\theta(t\,,x\,,y)=
	\lim_{\varepsilon,\delta\downarrow0}\theta_{\varepsilon,\delta}(t\,,x\,,y)$ 
	(in probability) for every $t>0$ and $x,y\in\R$.
\end{definition}
There is in fact an integration theory associated to this definition, as was the
case in finite dimensions. But we will not need that theory here, and so will not
discuss it.

We introduce analogous notation to the one earlier as follows.

Let
\[
	U_{\varepsilon,\delta}(t\,,x\,,\cdot) :=
	\wh{\theta}_{\varepsilon,\delta}(t\,,x\,,\cdot)
\]
denote the Fourier transform of
$y\mapsto\theta_{\varepsilon,\delta}(t\,,x\,,y)$ in the sense of distributions.
Clearly, $U_{\varepsilon,\delta}$ solves weakly the following random PDE:
\[
	\partial_t U_{\varepsilon,\delta}(t\,,x\,,\xi) =
	\nu\partial^2_x U_{\varepsilon,\delta}(t\,,x\,,\xi) 
	-\nu\xi^2 U_{\varepsilon,\delta}(t\,,x\,,\xi)
	+ i\xi V_{\varepsilon,\delta}(t\,,x) U_{\varepsilon,\delta}(t\,,x\,,\xi),
\]
subject to $U_{\varepsilon,\delta}(0\,,x\,,\xi) = \wh{\theta}_0(x\,,\xi)$.
In particular,
\[
	u_{\varepsilon,\delta}(t\,,x\,,\xi) := \e^{\nu\xi^2 t} U_{\varepsilon,\delta}(t\,,x\,,\xi)
\]
solves the random PDE,
\[
	\partial_t u_{\varepsilon,\delta}(t\,,x\,,\xi) =
	\nu\partial^2_x u_{\varepsilon,\delta}(t\,,x\,,\xi) 
	+ i\xi V_{\varepsilon,\delta}(t\,,x) u_{\varepsilon,\delta}(t\,,x\,,\xi),
\]
subject to $u_{\varepsilon,\delta}(0\,,x\,,\xi) = \wh{\theta}_0(x\,,\xi)$. 
We invoke classical theory once again to see that the unique solution to
the preceding PDE is
\[
	u_{\varepsilon,\delta}(t\,,x\,,\xi) = \E
	\left[\left. \wh{\theta}_0(x+B_t\,,\xi) \exp\left(
	i\xi\int_0^t V_{\varepsilon,\delta}(s\,,x+B_{t-s})\,\d s\right)
	\ \right|\ \mathcal{V}\vee\mathcal{T}_0\right]\qquad\text{a.s.},
\]
where the notation is the same as before.
In the case where $i\xi$ is replaced by $\xi$, this is for example found
in Freidlin \cite{Freidlin}. The present, more complex, case enjoys essentially exactly the
same proof [which we omit, as a result]. In this way, we see that
\[
	U_{\varepsilon,\delta}(t\,,x\,,\xi) =\E
	\left[\left. \wh{\theta}_0(x+B_t\,,\xi) \exp\left( - \nu\xi^2 t +
	i\xi\int_0^t V_{\varepsilon,\delta}(s\,,x+B_{t-s})\,\d s\right)
	\ \right|\ \mathcal{V}\vee\mathcal{T}_0\right]\qquad\text{a.s.,}
\]
and hence
\begin{align*}
	\left\| U_{\varepsilon,\delta}(t\,,x\,,\xi) \right\|_1
		&= \e^{-\nu\xi^2t}\left\| u_{\varepsilon,\delta}(t\,,x\,,\xi)\right\|_1
		\le \e^{-\nu\xi^2t}\left\| \wh{\theta}_0(x\,,\xi)\right\|_1\\
	&\le\e^{-\nu\xi^2t}\int_{-\infty}^\infty \left\| \theta_0(x\,,y)\right\|_1\,\d y.
\end{align*}
It follows from this and Assumption \ref{assum:2} that
$U_{\varepsilon,\delta}(t\,,x\,,\cdot)\in L^1(\R)$ a.s., whence
\begin{align*}
	&\theta_{\varepsilon,\delta}(t\,,x\,,y) 
		= \frac{1}{2\pi}\int_{-\infty}^\infty \e^{-i\xi y} U_{\varepsilon,\delta}(t\,,x\,,\xi)\,
		\d\xi\\
	&=\frac{1}{2\pi}\int_{-\infty}^\infty
		\E\left[\left. \wh{\theta}_0(x+B_t\,,\xi) \exp\left( -i\xi y - \nu\xi^2 t +
		i\xi\int_0^t V_{\varepsilon,\delta}(s\,,x+B_{t-s})\,\d s\right)
		\ \right|\ \mathcal{V}\vee\mathcal{T}_0\right]\d\xi,
\end{align*}
by the inversion theorem of Fourier transforms.
Because of first Fubini's theorem, and then another round of
Fourier inversion, this yields
\begin{align*}
	\theta_{\varepsilon,\delta}(t\,,x\,,y) &= 
		\E\left[\left.\frac{1}{2\pi}\int_{-\infty}^\infty
		\wh{\theta}_0(x+B_t\,,\xi) \,\e^{ -i\xi y - \nu\xi^2 t +
		i\xi\int_0^t V_{\varepsilon,\delta}(s,x+B_{t-s})\,\d s}\d \xi
		\ \right|\ \mathcal{V}\vee\mathcal{T}_0\right]\\
	&= \E\left[\left. \ \theta_0\left( x+\sqrt{2\nu}\, W_t\,,
		y + \sqrt{2\nu}\,W'_t 
		-\int_0^t V_{\varepsilon,\delta}\left(s\,,x+\sqrt{2\nu}\,W_{t-s}\right)\d s\right)\ \right|\ 
		\mathcal{V}\vee\mathcal{T}_0\right],
\end{align*}
where $W'$ is a standard, linear Brownian motion that is independent of $(B\,,V)$,
and $W_t:=(2\nu)^{-1/2}\, B_t$.
It is now easy to deduce from 
Lemma \ref{lem:V} and the dominated convergence theorem that when $\theta_0$ satisfies assumptions \ref{assum:2} - \ref{assum:4} and $\theta_0$ is bounded, 
$\theta(t\,,x\,,y) := \lim_{\varepsilon,\delta\downarrow0}\theta_{\varepsilon,\delta}(t\,,x\,,y)$
exists in probability, and for every $t>0$ and $x,y\in\R$,
\begin{equation}\label{theta:1:2}\begin{split}
	&\theta(t\,,x\,,y) \\
	&= \E\left[\left. \ \theta_0\left( x+\sqrt{2\nu}\, W_t\,,
		y + \sqrt{2\nu}\, W'_t 
		-\int_0^t V\left(s\,,x+ \sqrt{2\nu}\, W_{t-s} \right)\d s\right)\ \right|\ 
		\mathcal{V}\vee\mathcal{T}_0\right],
\end{split}\end{equation}
almost surely.
This is exactly the same solution
as the one in Theorem \ref{th:FK}, with $\nu_1=\nu$
except in the latter theorem,
$B'$ was replaced by a Brownian motion with speed $\kappa=
\nu_2-\frac12\rho(0)$; equivalently, we obtain the above
from Theorem \ref{th:FK} when we set $\nu_2=\nu+\frac12\rho(0)$. 
Thus, we have proved the following.

\begin{theorem}\label{th:Strat}
	Choose and fix an arbitrary $\nu>0$.
	Then, the Stratonovich solution to the Kraichnan flow \eqref{pre:Kraichnan}
	is the same as the [It\^o--Walsh] solution to the generalized
	Kraichnan flow \eqref{Kraichnan} with $\nu_1=\nu$
	and $\nu_2=\nu+\frac12\rho(0)$.
\end{theorem}

We emphasize that, whereas the It\^o-Walsh solution to \eqref{pre:Kraichnan}
exists only if $\nu>\frac12\rho(0)$ [see Theorem
\ref{th:exist:U} and Assumption \ref{assum:1} with {$\nu_1=\nu_2=\nu$}],
the Stratonovich solution exists for all $\nu>0$. This is because,
tautologically, $\nu_2>\frac12\rho(0)$
in Theorem \ref{th:Strat}. 

The following is a simple consequence of the preceding probabilistic representation
\eqref{theta:1:2} of the Stratonovich solution to \eqref{pre:Kraichnan}.

\begin{corollary}\label{cor:inviscid}
	Suppose $\theta_0$ satisfies \eqref{T_0-T_0} and is bounded.
	Let $\nu>0$ and define $\theta^{(\nu)}$ to be the Stratonovich solution of
	\eqref{pre:Kraichnan}. Then, for every $k\ge 2$, $t>0$, and $x,y\in\R^2$,
	\[
		\lim_{\nu\downarrow0}\theta^{(\nu)}(t\,,x\,,y) = 
		\theta_0\left( x\,,
		y  -\int_0^t V(s\,,x)\,\d s\right) := \theta^{(0)}(t\,,x\,,y)	
		\qquad\text{in $L^k(\Omega)$.}
	\]
\end{corollary}

In other words, the ``Stratonovich solution''  $\theta^{(0)}$
to the inviscid form of \eqref{pre:Kraichnan} is solved
by formally applying the method of characteristics, as one would do 
in the classical PDE setting when $V$ is smooth. 

The Stratonovich solution to the
inviscid form of \eqref{pre:Kraichnan} is very easy to understand:
\[
	\theta^{(0)}(t\,,x\,,y) = \theta_0\left( x\,, y+\mathcal{W}(t\,,x)\right),
\]
where $\{\mathcal{W}(t)\}_{t\ge0}$ is the cylindrical Brownian motion defined by
\[
	\mathcal{W}(t\,,x) := -\int_0^t V(s\,,x)\,\d s
	\qquad\text{for all $t\ge0$ and $x\in\R$}.
\]
More precisely, the proof of Lemma \ref{lem:V} shows immediately that
$\mathcal{W}$ is a centered Gaussian process whose covariance
is described by \eqref{Cov:F}.

\begin{proof}[Proof of Corollary \ref{cor:inviscid}]
	We use \eqref{theta:1:2} and write
	\[
		\left\| \theta^{(\nu)}(t\,,x\,,y) - \theta^{(0)}(t\,,x\,,y) \right\|_k
		\le J_1 + J_2 + J_3,
	\]
	where
	\begin{align*}
		J_1 &:=\Bigg\| \theta_0\left( x+\sqrt{2\nu}\, W_t\,, y+\sqrt{2\nu}\, W_t'
			-\int_0^t V\left( s\,,x+\sqrt{2\nu}\, W_{t-s}\right)\d s\right)\\
		&\hskip1.8in - \theta_0\left( x\,, y+\sqrt{2\nu}\, W_t'
			-\int_0^t V\left( s\,,x+\sqrt{2\nu}\, W_{t-s}\right)\d s\right) \Bigg\|_k,\\
		J_2&:=\Bigg\| \theta_0\left( x\,, y+\sqrt{2\nu}\, W_t'
			-\int_0^t V\left( s\,,x+\sqrt{2\nu}\, W_{t-s}\right)\d s\right)\\
		&\hskip2.6in - \theta_0\left( x\,, y
			-\int_0^t V\left( s\,,x+\sqrt{2\nu}\, W_{t-s}\right)\d s\right)\Bigg\|_k,\\
		J_3&:=\Bigg\| \theta_0\left( x\,, y
			-\int_0^t V\left( s\,,x+\sqrt{2\nu}\, W_{t-s}\right)\d s\right)
			- \theta_0\left( x\,, y
			-\int_0^t V ( s\,,x)\,\d s\right)\Bigg\|_k.
	\end{align*}
	In accord with \eqref{T_0-T_0},
	\begin{align*}
		J_1^k &\le \wt{A}_k (2\nu)^{k\alpha/2}\E\left( |W_t|^{k\alpha}\right)=
			O\left(\nu^{k\alpha/2}\right)\qquad\text{as $\nu\downarrow0$},
			\text{ and}\\
		J_2^k &\le \wt{A}_k(2\nu)^{k\zeta/2}\E\left( |W_t'|^{k\zeta}\right)
			= O\left( \nu^{k\zeta/2}\right)\qquad\text{as $\nu\downarrow0$}.
	\end{align*}
	Thus, it suffices to prove that $J_3$
	converges to zero as $\nu\downarrow0$. Since
	$V$ is conditionally Gaussian, given the process $W$,
	Lemma \ref{lem:V} and its proof yield
	\begin{align*}
		\left\|\int_0^t V\left( s\,,x+\sqrt{2\nu}\, W_{t-s}\right)\d s
			- \int_0^t V ( s\,,x)\,\d s\right\|_2^2
			&= 2\int_0^t\E\left[ \rho(0) - \rho\left( 
			\sqrt{2\nu}\, W_{t-s}\right)\right]\d s
	\end{align*}
	which goes to $0$ as $\nu\downarrow0$ by the dominated convergence theorem.  Because the $L^k(\Omega)$-norm
	of a centered Gaussian random variable is proportional to the $(k/2)$th power
	of its variance, the conditional form of Jensen's inequality yields,
	\begin{align*}
		&\left\|\int_0^t V\left( s\,,x+\sqrt{2\nu}\, W_{t-s}\right)\d s
			- \int_0^t V ( s\,,x)\,\d s\right\|_k^k\\
		&\hskip2in\le  \E\left(|W_1|^k\right)\cdot
			\left\|\int_0^t V\left( s\,,x+\sqrt{2\nu}\, W_{t-s}\right)\d s
			- \int_0^t V ( s\,,x)\,\d s\right\|_2^k\\
		&\hskip2in\to 0 
		\qquad\text{as $\nu\downarrow0$}.
	\end{align*}
	Because of \eqref{T_0-T_0}, this shows that $J_3^k \to 0$
	as $\nu\downarrow0$, and completes the proof.
\end{proof}

\section{Measure-valued initial profiles}

Temporarily let $G_{\theta_0}$ denote the Stratonovich solution to \eqref{pre:Kraichnan},
starting from an arbitrary non-random initial function $\theta_0$ (as in \eqref{pre:Kraichnan}) that satisfies
Assumptions \ref{assum:2}--\ref{assum:4} and is bounded. According to Theorem \ref{th:Strat} [see also 
\eqref{theta:1:2}], we can write for all $t>0$ and $x,y\in\R$,
\[
	G_{\theta_0}(t\,,x\,,y) = \E\left[\left.
	\theta_0\left( x+B_t\,, y+B_t' -
	\int_0^t V\left( s\,,x+B_{t-s}\right)\d s\right)
	\ \right|\ \mathcal{V}\right],
\]
almost surely, where $(B\,,B')$ is a  2-dimensional Brownian motion
that is independent of $V$ and satisfies $\Var(B_1)=\Var(B'_1)=2\nu$.

Define for all $t>0$ and $a\in\R$, a process $B^{(t,a)}$ as
\begin{equation}\label{BB}
	B^{(t,a)}_s :=  B_s - \left(\frac{s}{t}\right)\left( B_t - a\right)
	\qquad\text{for all $s\in[0\,,t]$}.
\end{equation}
Then, clearly $B^{(t,a)}$ is a Brownian bridge conditioned to start from
the space-time point $(0\,,0)$ and end at the space-time point
$(t\,,a)$, and run at speed $2\nu$. Furthermore, $B^{(t,a)}$ is independent of $B_t$.
Thus, we can condition on $B_t$ and write
\begin{align*}
	&G_{\theta_0}(t\,,x\,,y)\\
	&= \int_{\R^2}\E\left[\left.
		\theta_0\left( x+a\,, y+b -
		\int_0^t V\left( s\,,x+B^{(t,a)}_{t-s}\right)\d s\right)
		\ \right|\ \mathcal{V}\right] p^{(\nu)}_t(a)p^{(\nu)}_t(b)\,\d a\,\d b\\
	&= \int_{\R^2}\theta_0(x',y') p^{(\nu)}_t(x-x') 
		\E\left[ \left. p^{(\nu)}_t\left( y-y'-\int_0^t
		V\left( s\,, x + B^{(t,x'-x)}_{t-s}\right)\d s\right)\ \right|\
		\mathcal{V}\right] \d x'\d y'.
\end{align*}
It is easy to justify measurability and integrability, as well as
the use of Fubini's theorem here. Therefore, we refrain from further
mentioning those details. Instead,
let us observe that the stochastic process
$X_s := x+ B^{(t,x'-x)}_{t-s}$ [$0\le s\le t$] is a Brownian bridge
that is conditioned to go from the space-time point
$(0\,,x')$ to the space-time point $(t\,,x)$, run at speed $2\nu$.

Define
\begin{equation}\label{Gamma}
	\Gamma^{(\nu)}_t(x\,,y) := p^{(\nu)}_t(x) 
	\E\left[ \left. p^{(\nu)}_t\left( y -\int_0^t
	V\left( s\,, x + B^{(t,-x)}_{t-s}\right)\d s\right)\ \right|\
	\mathcal{V}\right] ,
\end{equation}
in order to see that $G_{\theta_0}(t\,,x\,,y) = (\theta_0*\Gamma^{(\nu)}_t)(x\,,y)$
a.s.\ for all $t>0$ and $x,y\in\R$. We can now deduce from the linearity of
the SPDE \eqref{pre:Kraichnan} the following result. But first 
let us note that if the initial 
condition $\theta_0$ is a finite measure on $\R^2$, then the smoothed 
version of \eqref{pre:Kraichnan}---that is, \eqref{pre:Kraichnan:smooth}---still has a unique classical solution $\theta_{\varepsilon, \delta}$ a.s. Thus, the Stratonovich solution to \eqref{pre:Kraichnan} with measure initial condition can be defined in exactly the same way as Definition \ref{def:Stratonovich sol}. We are ready to state the next result.

\begin{theorem}\label{th:Kraichnan:Strat:measure}
	Choose and fix $\nu>0$ and let
	$\mu$ be a non-random finite Borel measure on $\R^2$.
	Consider the SPDE
	\begin{equation}\label{pre:Kraichnan:measure}\left[\begin{split}
		&\partial_t G^{(\nu)}_\mu(t\,,x\,,y) 
			= \nu\Delta G^{(\nu)}_\mu(t\,,x\,,y) + 
			\partial_y G^{(\nu)}_\mu(t\,,x\,,y)\,V(t\,,x) 
			\quad\text{on $(0\,,\infty)\times\R^2$,}\\
		&\text{subject to}\quad G^{(\nu)}_\mu(0)=\mu
			\qquad\text{on $\R^2$}.
	\end{split}\right.\end{equation}
	Then, the  Stratonovich solution to \eqref{pre:Kraichnan:measure}
	is 
	\[
		G^{(\nu)}_\mu(t\,,x\,,y) = (\mu*\Gamma^{(\nu)}_t)(x\,,y),
	\]
	where $\Gamma^{(\nu)}$ was defined in \eqref{Gamma}.
\end{theorem}

One can also study It\^o--Walsh type solutions to the Kraichnan flow
\eqref{pre:Kraichnan:measure},
or even the generalized Kraichnan flow \eqref{Kraichnan} where $\theta_0$
is a finite Borel measure. We will avoid such generalizations here.
Instead let us emphasize only that, because of Theorem \ref{th:Kraichnan:Strat:measure},
\begin{equation}\label{def:Gamma}
	(t\,,x\,,y)\mapsto\Gamma^{(\nu)}_t(x\,,y) 
	= G_{\delta_0\otimes\delta_0}(t\,,x\,,y)
\end{equation}
is the Stratonovich solution to \eqref{pre:Kraichnan:measure},
starting from initial Borel measure $\mu = \delta_0\otimes\delta_0$ on $\R^2$.\footnote{%
In analogy with the previous subsection, we say that $\Gamma^{(\nu)}$ is 
a \emph{Stratonovich solution to \eqref{pre:Kraichnan:measure}} when we  mean
that $\Gamma^{(\nu)}_t(x\,,y)= \lim_{\varepsilon,\delta}
\Gamma^{(\nu,\varepsilon,\delta)}_t(x\,,y)$ in $L^2(\Omega)$
for all $t>0$ and $x,y\in\R$, where 
$\Gamma^{(\nu,\varepsilon,\delta)}$ denotes the (standard PDE) solution to 
the version of \eqref{pre:Kraichnan:measure} wherein $V$ is replaced
by $V_{\varepsilon,\delta}$.}

A question of general interest to engineers is ``what happens when $\nu\downarrow0$''?
When the initial data was a nice function, Corollary \ref{cor:inviscid} showed that
the answer is that the solution to \eqref{pre:Kraichnan} converges to the 
[formal] method-of-characteristics solution 
\[
	\theta^{(0)}(t\,,x\,,y) := \theta_0\left( x\,, y-\int_0^t V(s\,,x)\,\d s\right)
\]
to the inviscid case of \eqref{pre:Kraichnan}. 
Moreover,
\[
	\Cov\left[ \theta^{(0)}(t\,,x\,,y) \,, \theta^{(0)}(t\,,x',y')\right]
	= \int_{\R^2}\theta_0(x\,,y-a)\theta_0(x\,,y-b)f(a\,,b)\,\d a\,\d b,
\]
where $f$ denotes the joint probability density function of
$X:=\int_0^t V(s\,,x)\,\d s$ and $X':=\int_0^t V(s\,,x')\,\d s$. Since
$(X\,,X')$ is a centered Gaussian with
\[
	\Var(X) = \Var(X')=t\rho(0)
	\quad\text{and}\quad
	\Corr(X\,,X') = \frac{\rho(x-x')}{\rho(0)} := K(x-x')
\]
(see Lemma \ref{lem:V} and its proof),
we have
\begin{equation}\label{eq:Gaussian Joint Density}
	f(a\,,b) = \frac{1}{2\pi t\rho(0)\sqrt{ 1-|K(x-x')|^2}}
	\exp\left\{ -\frac{a^2+b^2
	-2abK(x-x')}{2t\rho(0)\left( 1 - |K(x-x')|^2\right)}\right\},
\end{equation}
for all $a,b\in\R$.

In the physically-interesting case that
$\theta_0=\delta_0\otimes\delta_0$, it is clear that the inviscid form of
\eqref{pre:Kraichnan} does not have a reasonable solution.
Intuitively speaking, this is because $\lim_{\nu\downarrow0}\Gamma^{(\nu)}_t$
does not exist as a nice random function. In order to see this, we next study
the small-$\nu$ behavior of the
covariance function of $\Gamma^{(\nu)}_t(x\,,y)$ for every fixed $t>0$.

\begin{theorem}\label{th:nu:to:0:Gamma}
	Suppose that $\rho$ is non increasing on $[0\,,\infty)$ and
	\begin{equation}\label{rhorho}
		\rho(w)=\rho(0)\quad\Longrightarrow\quad w=0.
	\end{equation}
	Then, for all $t>0$ and $x,y\in\R$, 
	\begin{equation}\label{eq:nu to 0 Gamma E}
		\lim_{\nu\downarrow0}
		\E\left[\frac{\Gamma^{(\nu)}_t(x\,,y)}{p^{(\nu)}_t(x)} \right] =
		p^{(\rho(0)/2)}_t(y);
	\end{equation}
	and also the following holds for all $t>0$ and $x,x',y,y'\in\R$
	with $x'\neq x$ and $y'\neq y$:
	\begin{align}\label{eq:nu to 0 Gamma Cov}
		&\lim_{\nu\downarrow0}
			\E\left[ \frac{\Gamma^{(\nu)}_t(x\,,y)}{p^{(\nu)}_t(x)}
			\cdot \frac{\Gamma^{(\nu)}_t(x',y')}{p^{(\nu)}_t(x')} \right]  \\\notag
		&\hskip1.1in=\E \left[ \frac{1}{2\pi t \rho(0)\sqrt{1-
			\left[\wt{K}(t\,;x\, , x')\right]^2}} \exp \left(- \frac{y^2 + y'^2 - 2 y y' 
			\wt{K}(t\,;x \, , x')}{2 t \rho(0)\left(1-\left[
			\wt{K}(t\,;x\, , x')\right]^2\right)} \right)\right],
	\end{align}
	where 
	\begin{equation}\label{eq: K tilde}
		\wt{K}(t\,;x\, , x') = \frac{1}{t\rho(0)}\cdot
		\int_0^t \rho\left(x-x'+B_{t-s}^{(t, -x)} - 
		\wt{B}_{t-s}^{(t, -x')}\right)\d s,
	\end{equation}
	with $B_{t-s}^{(t, -x)}$ and $\wt{B}_{t-s}^{(t, -x')}$ two independent Brownian bridges,
	both defined as in \eqref{BB}. Finally, suppose that the following
	complement to \eqref{rho:cont} also holds: There exists $c>0$ and
	$\varpi\in(0\,,2]$ such that
	\begin{equation}\label{rhorho1}
		\rho(0) - \rho(x)\geq c (1+o(1))|x|^{\varpi}
		\qquad\text{as $x \downarrow 0$}.
	\end{equation}
	Then, in fact \eqref{eq:nu to 0 Gamma Cov} is valid for all
	$t>0$ and $x,x',y,y'\in\R$.
\end{theorem}

\begin{remark}
	It is always the case that $\rho$ is maximized at the origin; see
	\eqref{rho:max}. Eq.\ \eqref{rhorho} says that this maximum is attained
	uniquely. In particular, \eqref{rhorho} is equivalent to the condition
	that $\rho(w)=0$ if and only if $w=0$.
\end{remark}

\begin{remark}
	Theorem \ref{th:nu:to:0:Gamma} tacitly says that the right-hand
	side of \eqref{eq:nu to 0 Gamma Cov} is strictly positive and finite
	under the full hypotheses of Theorem \ref{th:nu:to:0:Gamma}.
	As we shall see, this fact by itself is non trivial and has a delicate proof;
	see Proposition \ref{prop: finite negative moments} below.
\end{remark}

\begin{remark}
	Under the full
	hypotheses of Theorem \ref{th:nu:to:0:Gamma}, we find readily that
	the collection
	$\{ \Gamma^{(\nu)}_t(x\,,y)/p^{(\nu)}_t(x)\}_{\nu\in(0,1)}$
	is an $L^2(\Omega)$-tight sequence of random variables for every fixed $t>0$ and $x,y\in\R$.
	It might help to recall that this means that for every sequence $\nu_1>\nu_2>\cdots$,
	of positive numbers that
	descend to zero, there exists a finite random variable
	$\mathscr{L}_t(x\,,y)=\mathscr{L}_t(x\,,y;\{\nu_i\}_{i=1}^\infty)$ 
	such that
	\[
		\lim_{i\to\infty}\frac{\Gamma^{(\nu_i)}_t(x\,,y)}{%
		p^{(\nu_i)}_t(x)} = \mathscr{L}_t(x\,,y)\qquad\text{in $L^2(\Omega)$},
	\]
	and hence also weakly. It is possible that $\mathscr{L}_t(x\,,y)$ does not depend on
	the sequence $\{\nu_i\}_{i=1}^\infty$; equivalently,
	$ \Gamma^{(\nu)}_t(x\,,y)/p^{(\nu)}_t(x)\to \mathscr{L}_t(x\,,y)$ in $L^2(\Omega)$
	as $\nu\downarrow0$. If this were true, then it would imply the existence of a second-order
	type of invariant measure, consistent with a result of
	van Eijnden \cite{VandenEijnden} for a related, though slightly different, fluid model.
	For a different type of limit theorem, see Fannjiang \cite{Fannjiang}.
\end{remark}

The proof of Theorem \ref{th:nu:to:0:Gamma} requires the following
regularity result, which implies that if $\rho$ is non increasing on $[0\,,\infty)$, then:
\begin{compactenum}[1)]
\item Under Condition \eqref{rhorho}, the right-hand side of
	\eqref{eq:nu to 0 Gamma Cov} is strictly positive and finite
	provided that $x\neq x'$ and $y\neq y'$;
\item Under the more restrictive Condition \eqref{rhorho1}, the right-hand side of
	\eqref{eq:nu to 0 Gamma Cov} is strictly positive and finite for all $x,x',y,y'\in\R$.
\end{compactenum}
The more precise details follow.

\begin{proposition}\label{prop: finite negative moments}
	Let $\wt{K}(t\,, x\,, x')$ be defined in \eqref{eq: K tilde} 
	and assume that $\rho$ is non increasing on $[0, \infty)$. Then 
	for all $t>0$ and for every two distinct real numbers $x$ and $x'$,
	\begin{equation}\label{1-K:Lp}
		\frac{1}{1-\wt{K}(t\,,x\,,x')} \in L^p(\Omega)
		\qquad\text{for all $p\in[2\,,\infty)$}.
	\end{equation}
	Furthermore, \eqref{1-K:Lp} holds for all $t>0$ and $x,x'\in\R$
	provided  additionally that \eqref{rhorho1} holds.
\end{proposition}

\begin{proof}
	We  treat the two cases $x \neq x'$ and $x = x'$ separately. 
	
	{\it The case $x \neq x'$. }
	Because of \eqref{rho:max} and the positivity of $\rho$,
	\begin{align}\notag
		&\qquad [\rho(0)]^2 t^2 \left(1-[\wt{K}(t\,, x\,, x')]^2\right) \\\notag
		&= \int_0^t \left[
			\rho(0)-\rho\left( x-x' + B^{(t,-x)}_{t-s}
			-\wt{B}^{(t,-x')}_{t-s}\right)\right]\d s
			\cdot \int_0^t \left[
			\rho(0) + \rho\left( x-x' + B^{(t,-x)}_{t-s}
			-\wt{B}^{(t,-x')}_{t-s}\right)\right]\d s\\
		&\ge t\rho(0) \int_0^t \left[
			\rho(0)-\rho\left( x-x' + B^{(t,-x)}_{t-s}
			-\wt{B}^{(t,-x')}_{t-s}\right)\right]\d s
			\qquad\text{a.s.}
			\label{det:LB}
	\end{align}
	
	On one hand, we may consider \eqref{BB}, and observe that
	\[
		x + B^{(t,-x)}_{t-s} = x\left(\frac st\right) + 
		\left\{ B_{t-s} - \left(\frac{t-s}{t}\right)B_t\right\},
	\]
	and that the process within the curly brackets is a Brownian bridge from 
	$(0\,,0)$ to $(t\,,0)$, run at speed $2\nu$. A similar decomposition 
	is valid for $x'+\wt{B}^{(t,-x')}_{t-s}$.
	In particular, we may write
	\begin{equation}\label{x:x:lambda}
		x-x' + B^{(t,-x)}_{t-s} -\wt{B}^{(t,-x')}_{t-s}
		=(x-x')\left(\frac st\right) + B^\circ_s,
	\end{equation}
	where $B^\circ$ is a Brownian bridge from $(0\,,0)$ to $(t\,,0)$,
	run at speed $4\nu$. 
	
	For every $r \in[0\,,t/2]$ consider the event,
	\[
		\mathbf{E}(r) := \left\{ \omega\in\Omega:\ \inf_{s\in[t-r,t]}
		\left| (x-x')\left(\frac st\right) + B^\circ_s(\omega)  \right|
		> \frac{|x-x'|}{4}  \right\}.
	\]
	Since $\E(\sup_{s\in[0,t]}|B_s^\circ|)<\infty$,
	\[
		\P\left( \mathbf{E}(r)\right) \ge 
		\P \left\{\sup_{s \in [0, r]}|B^{\circ}_s| < 
		\frac{|x-x'|}{4}  \right\}
		\qquad\text{for all $r\in[0\,,t]$},
	\]
	and $r\mapsto r^{-1}\sup_{s\in[0,r]}\Var(B_s^\circ)$ is
	bounded uniformly above on $[0\,,t]$, an application of
	the  Borell, Sudakov--Tsirelson inequality (see \cite{Borell}
	and \cite{SudakovTsirelson}) yields
	\[
		\P\left( \mathbf{E}(r)\right) 
		\ge 1 - K\e^{-K/r}
		\qquad\text{for all $r\in[0\,,t/2]$},
	\]
	where $K=K(\nu\,,x-x',t)$ is a strictly-positive real number that does not
	depend on $r$.
	
	On the other hand, for all $r\in[0\,,t/2]$,
	\[
		\rho(0)-\rho\left( (x-x')\left(\frac st\right) + B^\circ_s \right)
		\ge \rho(0) - \rho\left( \frac{x-x'}{2}\right)
		\quad\text{for all $s\in[t-r\,,t]$},
	\]
	almost surely on $\mathbf{E}(r)$. This is because $\rho$ is
	assumed to be non increasing on $[0\,,\infty)$, and hence also 
	non decreasing on $(-\infty\,,0]$ by symmetry; 
	see \eqref{rho:max}.
	Keeping in mind \eqref{x:x:lambda}, the above inequality and
	\eqref{det:LB} together imply that, for all $r\in[0\,,t/2]$,
	\[
		\rho(0)^2 t^2 \left( 1-[\wt{K}(t\,, x\,, x')]^2\right) 
		\ge rt \rho(0)\left[\rho(0)-\rho\left(
		\frac{x-x'}{2}\right)\right]:= Lr\qquad
		\text{a.s.\ on $\mathbf{E}(r)$}.
	\]
	Clearly, $L=L(\rho(0)\,,t\,,x-x')$ does not depend on the numerical value
	of $r\in[0\,,t/2]$. It follows that 
	\begin{equation}\label{det:LB:Lr}
		\P\left\{ \rho(0)^2 t^2 \left(
		1-[\wt{K}(t\,, x\,, x')]^2\right) < Lr\right\}  \le K
		\e^{-K/r}
		\qquad\text{for all $r\in[0\,,t/2]$}.
	\end{equation}
	The inequality \eqref{det:LB:Lr} shows that
	the non-negative random variable $1-[\wt{K}(t\,, x\,, x'])^2$---whence
	also $1-\wt{K}(t\,,x\,,x')$---%
	has finite negative moments of all orders when $x\neq x'$.\\
	
	{\it The case that $x=x'$}. 
	Since $x=x'$, we  need only to estimate the quantity 
	$\int_0^t [\rho(0)-\rho(b_s)]\d s$, where $b$ denotes a Brownian bridge from 
	$0$ to $0$  during the time span $[0\,,t]$.
	
	Now let us suppose $\rho$ is non increasing on $(0\,,\infty)$ and
	satisfies \eqref{rhorho1}.
	Choose and fix a real number $\alpha>0$ that satisfies
	$\alpha<(2+\varpi)^{-1}.$
	Since $\varpi<2$, any $\alpha<\frac14$
	will do the job. Now, 
	\[
		\int_0^t\left[ \rho(0)-\rho(b_s)\right]\d s\ge\text{const}\cdot
		\varepsilon^{\alpha\varpi}\int_0^t\bm{1}_{\{
		|b_s| > \varepsilon^\alpha\}}\,\d s\qquad\text{a.s.}
	\]
	Therefore, we need to study the behavior of 
	$\P({\bf A}_\varepsilon)$, as $\varepsilon\downarrow0$, where
	\begin{equation}\label{eq:A-epsilon}
		 {\bf A}_\varepsilon := \left\{ \omega\in\Omega:\
		 \int_0^t\bm{1}_{\{
		|b_s(\omega)| > \varepsilon^\alpha\}}\,\d s
		<\varepsilon^{1-\alpha\varpi}\right\}
		\qquad\text{for all $\varepsilon>0$}.
	\end{equation}
	Consider the stopping times,
	\[
		T_\varepsilon := \inf\left\{ s>0:\ |b_s| = 2{\varepsilon^\alpha}\right\}.
	\]
	By the continuity of the trajectories of $b$, it follows easily that
	\[
		{\bf A}_\varepsilon \subseteq \left\{ \omega\in\Omega:\
		T_{\varepsilon} > \frac t2\right\}\cup\widetilde{\bf A}_\varepsilon,
	\]
	where
	\[
		\widetilde{\bf A}_\varepsilon := \left\{ \omega\in\Omega:\
		\sup_{{%
		0< s < \varepsilon^{1-\alpha \varpi} }}
		|b_{T_{\varepsilon}+s}(\omega) - 
		b_{T_\varepsilon}(\omega)| >{\varepsilon^\alpha}\right\}.
	\]
	A standard small-ball estimate for the Brownian bridge shows that
	there exist strictly-positive real numbers $c_1=c_1(t\,,\nu)$ and $c_2=c_2(t\,,\nu)$ such that
	\[
		\P\{ T_\varepsilon>t/2\} \le c_1
		\exp\left( - \frac{c_2}{\varepsilon^{2\alpha}}\right)
		\qquad\text{uniformly for all $\varepsilon\in(0\,,1)$}.
	\]
	Therefore, for all $k\ge2$,
	\begin{align*}
		\P({\bf A}_\varepsilon) &\le  c_1\exp\left( - \frac{c_2}{\varepsilon^{2\alpha}}\right)
		 	+2^k\varepsilon^{-\alpha k}\E\left(
			\sup_{{%
		0< s < \varepsilon^{1-\alpha \varpi} }}
			|b_{T_{\varepsilon}+s}(\omega) - b_{T_{\varepsilon}}(\omega)|^k \right)\\
		&\le c_1\exp\left( - \frac{c_2}{\varepsilon^{2\alpha}}\right) + c_3(k)
			\varepsilon^{(1-\alpha\varpi-2\alpha)k/2}\,,
	\end{align*}
	after a standard modulus-of-continuity estimate and the 
	[inhomogeneous] strong Markov property of Brownian bridge. In any case, since
	$\alpha>0$ and $k\ge2$ are arbitrary,
	\[
		\lim_{\varepsilon\to 0^+}\frac{%
		\log\P({\bf A}_\varepsilon)}{\log(1/\varepsilon)}=-\infty.
	\] 
	when $k$ is fixed, the $\lim$ is a constant. This proves, in particular,
	that $\P({\bf A}_\varepsilon)=o(\varepsilon^k)$ as $\varepsilon\downarrow0$
	for every $k>1$. Therefore, we can deduce from \eqref{eq:A-epsilon}  that 
	$\int_0^t[\rho(0)-\rho(b_s)]\,\d s$ has finite negative 
	moments of all orders. This completes the proof. 
\end{proof}

We are now ready to prove Theorem \ref{th:nu:to:0:Gamma}.

\begin{proof}[Proof of Theorem \ref{th:nu:to:0:Gamma}]
	To prove \eqref{eq:nu to 0 Gamma E}, we begin 
	with the expression \eqref{Gamma} in order
	to see that 
	\begin{equation}\label{EGEG}
		\E \left[\Gamma_t^{(\nu)}(x\, , y)\right] = p_t^{(\nu)}(x) 
		\cdot \E \left[ p_t^{(\nu)}\left(y- \int_0^t V
		\left(s\,, x+B_{t-s}^{(t, -x)}\right) \d s\right)\right].
	\end{equation}
	Given $B$, the conditional law of $\int_0^t V(s, x+B_{t-s}^{(t, -x)})\, \d s$ 
	is centered Gaussian 
	with [conditional] variance $t\rho(0)$. It follows from the tower property of 
	conditional expectations that
	\[
		 \E \left[ p_t^{(\nu)}\left(y- \int_0^t V
		 \left(s\,, x+B_{t-s}^{(t, -x)}\right) \d s\right)\right] = 
		 \int_{-\infty}^\infty 
		 p_t^{(\nu)}(y-z) p_t^{(\rho(0)/2)}(z)\,\d z = p_t^{(\nu + \rho(0)/2)}(y),
	\]
	whence
	\[
		\E \left[\Gamma_t^{(\nu)}(x\, ,y) \right]= p_t^{(\nu)}(x) 
		p_t^{(\nu + \rho(0)/2)}(y)
		\qquad\text{by \eqref{EGEG}}.
	\]
	This  readily implies \eqref{eq:nu to 0 Gamma E}. 
	
	In order to prove \eqref{eq:nu to 0 Gamma Cov}, 
	we may first condition on $B$ and $\wt{B}$ in order to see
	that $\int_0^t V(s\,, x+B_{t-s}^{(t, -x)}) \,\d s$ and 
	$\int_0^t V(s\,, x'+\wt{B}_{t-s}^{(t, -x')}) \,\d s$ are two centered Gaussian
	random variables with conditional covariance,
	\[
		\int_0^t \rho\left(x-x'+B_{t-s}^{(t, -x)}-\wt{B}_{t-s}^{(t, -x')}\right) \d s.
	\]
	Thus, we appeal to \eqref{eq:Gaussian Joint Density}, by 
	first conditioning on $B$ and $\wt{B}$, and find that 
	\begin{align*}
	&\E \left[ \left.
		p_t^{(\nu)}\left(y- \int_0^t V\left(s\,, x+ B_{t-s}^{(t, -x)}\right)\d s \right)
		p_t^{(\nu)} \left(y' - \int_0^t V\left(s, x'+ \wt{B}_{t-s}^{(t, -x')}\right)\d s \right)
		\ \right|\ B, \wt{B}\right]\\
	&\hskip.8in= \int_{\R^2} 
		\frac{p_t^{(\nu)}(y-a)p_t^{(\nu)}(y'-b) %
		}{2\pi t \rho(0)\sqrt{1-\left[\wt{K}(t\,;x\, , x')\right]^2}} 
		\exp \left(- \frac{a^2 + b^2 - 2 ab\wt{K}(t\,;x \, , x')}{2 t \rho(0)
		\left(1-\left[\wt{K}(t\,;x\, , x')\right]^2\right)} \right) \d a\, \d b.
	\end{align*}
	Manifestly, the right-hand side is strictly positive, and it is finite owing to 
	Proposition \ref{prop: finite negative moments}.
	
	Because $\nu \to p_t^{(\nu)}$ is an approximate identity, 
	\eqref{rho(0)>0} and the dominated convergence theorem together
	ensure that the above integral converges to 
	\[
		\frac{1}{2\pi t \rho(0)\sqrt{1-\left[
		\wt{K}(t\,;x\, , x')\right]^2}} \exp \left(- \frac{y^2 + 
		(y')^2 - 2 y y'\wt{K}(t\,;x \, , x')}{%
		2 t \rho(0)\left(1-\left[\wt{K}(t\,;x\, , x')\right]^2\right)} \right),
	\]
	as $\nu \to 0$. Thus, \eqref{eq:nu to 0 Gamma Cov} follows from 
	the definition \eqref{Gamma} of the random field $\Gamma^{(\nu)}$ 
	and an application of dominated convergence theorem together with 
	Proposition \ref{prop: finite negative moments}. 
\end{proof}

\section{Analysis in a special  case}
The literature on turbulence predicts a highly complex, ``multifractal,''
behavior for the solution to  the Kraichnan model \eqref{pre:Kraichnan}
at every fixed viscosity level $\nu>0$, ideally $\nu\approx0$;
see for example Warhaft \cite{Warhaft}.
In the case that $\theta_0$ is a nice function,
Theorem \ref{th:Strat} ensures that this sort of assertion ought not be valid
for the It\^o/Walsh solution, as well as the Stratonovich solution to \eqref{pre:Kraichnan}.
And the same continues to hold when $\theta_0$ is a nice measure, owing to
Theorem \ref{th:Kraichnan:Strat:measure}.

In this section we state and 
prove a prefatory version of the preceding prediction in the special case
that $\rho$ is a constant $\rho(0)>0$. In order to conform with what we think
might be the physically-interesting representation we consider the Stratonovich
solution only. In that case, we can re-write \eqref{pre:Kraichnan} as the following
infinite-dimensional Stratonovich stochastic differential equation:
\begin{equation}\label{Kraichnan:simple}\left[\begin{split}
	&\d\theta(t\,,x\,,y) =
		\nu \Delta\theta(t\,,x\,,y)\,\d t +  \sqrt{\rho(0)}\,\partial_y
		\theta(t\,,x\,,y) \circ\d W_t
		\qquad\text{for }(t\,,x\,,y)\in(0\,,\infty)\times\R^2,\\
	&\text{subject to }\theta(0\,,x\,,y) = \delta_0(x)\delta_0(y),
\end{split}\right.\end{equation}
where $W=\{W_t\}_{t\ge0}$ is standard Brownian motion.
See \S\ref{sec:pf:th:cont:U} for more details.

We saw in Theorem \ref{th:Kraichnan:Strat:measure}
that the Stratonovich solution 
to \eqref{Kraichnan:simple} is
$\theta = \Gamma^{(\nu)}$, which was defined in \eqref{Gamma}.
In the current setting where the spatial correlation $\rho$ is a constant,
the probabilistic expression for $\Gamma^{(\nu)}$ reduces to the following:
\begin{equation}\label{Gamma:reduced}
	\Gamma^{(\nu)}_t(x\,,y) = p^{(\nu)}_t(x)
	p^{(\nu)}_t\left( y - \sqrt{\rho(0)}\, W_t\right),
\end{equation}
which can be verified directly from It\^o/Stratonovich calculus as well (in the
present, simple setting). It follows in particular that the random function
$\Gamma^{(\nu)}_t$ tends to zero as $1/t$, as $t$ growth
without bound. For example,
\begin{equation}\label{eq:sup:Gamma}
	\sup_{x,y\in\R}\Gamma^{(\nu)}_t(x\,,y) = (4\pi\nu t)^{-1}
	\qquad\text{a.s.\ for every $t>0$.}
\end{equation}
The following theorem shows that the set of times where $\Gamma^{(\nu)}(0\,,0)$
behaves largely different from $1/t$, however, has a macroscopic
multifractal structure. In order to describe that multifractal behavior we
need a few notions from geometric measure theory of macroscopic
structures (see Barlow and Taylor
\cite{BT1989,BT1992}, for example).

\[
	\mathcal{C}_m(A) := \#\left\{j\in[m]:\
	[j\,,j+1]\cap A\neq\varnothing\right\},
\]
where $[m]:=\{0\,,\ldots,m\}$.
Define
\[
	\oDim(A) := \limsup_{m\to\infty}\frac{\log \mathcal{C}_m(A)}{\log m}
	\qquad
	\uDim(A) := \liminf_{m\to\infty}\frac{\log \mathcal{C}_m(A)}{\log m}.
\]
Then, for all sets $A\subset\R_+$,
\[
	0\le \uDim(A) \le \oDim(A)\le 1,
\]
In principle, all three inequalities can be strict.
But when $\oDim(A)=\uDim(A)$ we write $\Dim(A)$ for their common value.
The quantity $\Dim(A)$ is then referred to as the \emph{macroscopic
Minkowski (or fractal) dimension of $A$.}

In order to simplify the exposition somewhat we study only the large-time behavior of
$t\mapsto \Gamma^{(\nu)}_t(x\,,y)$ at $x=y=0$, since the point $(0\,,0)$
is slightly more distinguished than other points in light of the fact that
the initial data is $\delta_0\otimes\delta_0$. It is not hard to extend our analysis
to study the behavior of $t\mapsto\Gamma^{(\nu)}_t(x\,,y)$ for other values
of $x$ and $y$ though.

First, we observe that the typical behavior of $t\mapsto\Gamma^{(\nu)}_t(0\,,0)$ is
$\text{\rm const}/t$; see also \eqref{eq:sup:Gamma}. The following is a  fractal-analysis
version of such an assertion.

\begin{theorem}\label{th:Gamma:Dim:local:1}
	With probability one,
	\[
		\Dim\left\{ t>0:\ \Gamma^{(\nu)}_t(0\,,0) > \frac{K}{t}\right\}
		=\begin{cases}
			1&\text{if $0<K<(4\pi\nu)^{-1}$},\\
			0&\text{if $K\ge (4\pi\nu)^{-1}$}.
		\end{cases}
	\]
	Furthermore, {for any $x\in \R$, }
	\[
		\Dim\left\{ t>0:\ \Gamma^{(\nu)}_t(0\,,0) = 
		\frac{\exp\{-x^2/(4\nu t)\}}{4\pi\nu t}\right\}
		=\frac12\qquad\text{a.s.}
	\]
\end{theorem}

Among other things, Theorem \ref{th:Gamma:Dim:local:1} says that, asymptotically
as $t\to\infty$, $t\mapsto\Gamma^{(\nu)}_t(0\,,0)$ typically behaves as 
$K/t$ for all possible values of $K\in(0\,,(4\pi\nu)^{-1})$. Moreover, the set of
times were $\Gamma^{(\nu)}_t(0\,,0)>K/t$ for such a $K$ is a ``monofractal''
of full macroscopic Minkowski dimension. The following result shows that there are
 more subtle, logarithmic, corrections on whose scale a suitable log-scaling
 of the set of decay times of order $t^{-1}(\log t)^{-\delta}$ is a bona fide 
 macroscopic multifractal.

\begin{theorem}\label{th:Gamma:Dim:local:2}
	Choose and fix a real number $\delta>0$.
	Then, with probability one,
	\[
		\Dim\left(\log\left\{ t>\e:\ \Gamma^{(\nu)}_t(0\,,0) < 
		\frac{1}{t(\log t)^{\delta}}\right\}\right)
		= \left( 1 - \frac{2\delta\nu}{\rho(0)}\right)_+\qquad\text{a.s.}
	\]
\end{theorem}

Of course Theorem \ref{th:Gamma:Dim:local:2} has non-trivial content
if and only if 
\[
	\delta < \frac{\rho(0)}{2\nu} =: \mathcal{R}.
\]
Interesting enough, $\mathcal{R}$ is the ratio of 
turbulent diffusivity to thermal diffusivity and, as such,
plays a similar role to $\frac12\Pr$---half of the Prantdl number---in the
non stochastic setting. Larger values of
$\mathcal{R}$ translate to more turbulent transport of the underlying passive scalar;
see Grossmann and Lohse \cite{GrossmannLohse} and its extensive bibliography
for earlier physical (in some cases, experimental) observations that the
the multifractal behavior of $\Gamma^{(\nu)}$ is determined essentially 
solely by the value of the Prandtl (or Schmidt) number, here $\mathcal{R}$.
See \S\ref{sec:Fluids} for some more explanation of some of the physical
terminology that is used here.

In  light of the preceding remarks, Theorem \ref{th:Gamma:Dim:local:2}
implies that, as $\mathcal{R}$ gets larger, higher dissipation rates can be observed on 
non-trivial unbounded sets of greater macroscopic dimension. Stated yet in another way,
the larger the value of $\mathcal{R}$ the more multifractal is the rates of dissipation
of the passive scalar.

We begin the proofs with a technical lemma about standard Brownian motion.

\begin{lemma}\label{lem:W:lvlset:sup}
	Let $\{W_t\}_{t\ge0}$ denote a standard, linear Brownian motion, and
	$z\in\R$ and $\alpha>0$ be fixed numbers. Then, with probability one,
	\[
		\Dim\left\{ t\ge0:\ W_t=z\right\}=\tfrac12
		\quad\text{and}\quad
		\Dim\left\{ t\ge0:\ |W_t-z|< \alpha\sqrt t\right\}=1.
	\]
\end{lemma}

\begin{proof}
	The first part of the lemma is well known; see, for example
	Khoshnevisan \cite{Khoshnevisan2} 
	in the case that $W$ is replaced by a random walk. We make
	small adjustments to that proof in order to verify the first part of
	our lemma.
	
	Consider the following random subset of $\Z_+$:
	\[
		L := \left\{ t\ge0:\ W_t=z\right\},
		\quad\text{and let }J_N := \sum_{j=1}^N
		\1_{\{L\cap [j,j+1]\neq\varnothing\}},
	\]
	for all $N\in\mathbb{N}$.
	It is well known, and easy to verify directly from the Markov property of
	$W$, that for every $B>0$ there exist real numbers $C_1,C_2$---depending
	only on $(B\,,z)$---such that
	\begin{equation}\label{W:hit:prob}
		\frac{C_1}{\sqrt A}\le
		\P\left\{ \exists t\in[A\,,A+B]:\ W_t=z\right\}\le\frac{C_2}{\sqrt A}
		\qquad\text{for all $A\ge1$}.
	\end{equation}
	For example,  it is well known (as well as elementary) that
	\[
		\P\left\{ \exists t\in[A\,,A+B]:\ W_t=0\right\}
		=\frac2\pi\arccos\sqrt{\frac{A}{A+B}},
	\]
	which clearly implies \eqref{W:hit:prob} when $z=0$. 
	The case $z\neq0$ follows from potential-theoretic
	considerations; see \cite{Escape} for example.
	
	In any case, it follows that there exist real numbers $C_3,C_4>0$---depending
	only on $z\in\R$---such that
	\begin{equation}\label{E(J_N)}
		C_3\sqrt{N}\le \E(J_N)\le C_4\sqrt{N}\qquad
		\text{for all $N\in\mathbb{N}$}.
	\end{equation}
	In particular, Chebyshev's inequality implies
	that
	$\sum_{n=1}^\infty \P\{ J_{2^n}\ge 2^{(1+\varepsilon)n/2}\}<\infty$
	for all $\varepsilon>0$.
	Since $\varepsilon>0$ is arbitrary, the Borel--Cantelli lemma ensures that
	$(\log 2^n)^{-1}\log J_{2^n} \le\frac12+o(1)$ a.s.\ as $n\to\infty$. 
	Because $m\mapsto J_m$ is nondecreasing,
	a monotonicity argument then shows that 
	$\lim_{n\to\infty}(\log n)^{-1}\log J_n\le\frac12$ a.s.	
	This in turn implies that 
	\begin{equation}\label{Dim L le 1/2}
		\Dim(L)\le\tfrac12\quad\text{a.s.}
	\end{equation}
	Next we show that the above is in fact an a.s.\ identity,
	and hence prove the first assertion of the lemma.
	
	If $k$ and $j$ are integers that satisfy $k\ge j+2 > j\ge 1$, then
	we may apply the strong Markov property to the first time in
	$[j\,,j+1]$ that $W$ reaches $z$ in order to see that
	\begin{align*}
		&\P\left\{ L\cap [j\,,j+1]\neq\varnothing\,,
			L\cap[k\,,k+1]\neq\varnothing\right\}\\
		&\hskip1.5in\le \P\left\{ L\cap [j\,,j+1]\neq\varnothing\right\}
			\cdot\P\left\{ \exists s\in[k-j-1\,,k-j+1]:\
			W_s=0\right\}\\
		&\hskip1.5in\le \frac{C}{\sqrt{k-j-1}}\cdot \P\left\{ L\cap [j\,,j+1]\neq\varnothing\right\},
	\end{align*}
	for a real number $C$ that depends only on $z$; confer with \eqref{W:hit:prob}.
	Therefore,
	\begin{align*}
		\E(J_N^2) & \le  2\sum_{1\le j \le k\le N}
			\P\left\{ L\cap [j\,,j+1]\neq\varnothing\,,
			L\cap[k\,,k+1]\neq\varnothing\right\}\\
		&\le 2\sum_{\substack{1\le j \le N\\j\le k\le j+2}}
			\P\left\{ L\cap [j\,,j+1]\neq\varnothing\right\}
			+2C\sum_{\substack{1\le j\le N\\j+2\le k\le N}}
			\frac{\P\left\{ L\cap [j\,,j+1]\neq\varnothing\right\}}{
			\sqrt{k-j-1}}\\
		&=O\left( \left| \E(J_N)\right|^2\right)
			\qquad\text{as $N\to\infty$, by \eqref{E(J_N)}}.
	\end{align*}
	This, \eqref{E(J_N)}, and the Paley--Zygmund inequality 
	(see Lemma 7.3 in \cite{Khoshnevisan} for example)
	together imply that 
	\[
		\inf_{N\ge1}\P\left\{ J_N>\tfrac12 C_3\sqrt N\right\}:=q>0,
	\]
	whence $\Dim(L)=\limsup_{N\to\infty}(\log N)^{-1}\log J_N\ge\frac12$ with probability
	at least $q>0$. By the Kolmogorov 0-1 law, the latter event must in fact
	have full probability, whence $\Dim(L)\ge\frac12$ a.s.
	This and \eqref{Dim L le 1/2} together establish the first half of the lemma.
	
	The second part of the lemma follows from another second-moment
	computation.
	In order to simplify the notation let 
	\[
		\wt{L} := \left\{ t\ge0:\ |W_t-z| < \alpha\sqrt t\right\}.
	\]
	to be the random set whose dimension is supposed to be $1$. Elementary properties of
	the macroscopic Minkowski dimension ensure that it suffices to prove
	that $\Dim(\wt{L})\ge 1$ a.s. 
	
	For every integer $N>1$ define
	\[
		\wt{J}_N := \sum_{j=1}^N\1_{\{\wt{L}\cap [j,j+1]\neq\varnothing\}}.
	\]
	As $j\to\infty$,
	\[
		\P\left\{ \wt{L}\cap [j\,,j+1]\neq\varnothing\right\}
		\ge \P\left\{ |W_j-z|< \alpha\sqrt j\right\}
		\to \P\{|W_1|< \alpha\},
	\]
	which is strictly positive.
	Therefore, for all $N$ sufficiently large,
	\begin{equation}\label{EJ_N}
		\E (\wt{J}_N)  \ge \frac{N}{\sqrt 2}\,\P\{|W_1| < \alpha\}.
	\end{equation}
	Because $\wt{J}_N\le N$, whence also
	$\E(\wt{J}_N^2)\le N^2$,
	the Paley--Zygmund inequality implies that
	\[
		\P\left\{ \wt{J}_N > \tfrac12\E(\wt{J}_N)\right\}
		\ge \frac{\left| \P\{|W_1|< \alpha\}\right|^2}{8}
		\qquad\text{for all sufficiently-large $N$}.
	\]
	This, \eqref{EJ_N}, and Kolmogorov's 0--1 law together imply
	that $\limsup_{N\to\infty} (\wt{J}_N/N)>0$ a.s., which suffices
	to imply that $\Dim(\wt{L})\ge1$ a.s., and hence
	$\Dim(\wt{L})=1$ a.s.
\end{proof}

Armed with Lemma \ref{lem:W:lvlset:sup}, we can now derive Theorem
\ref{th:Gamma:Dim:local:1} fairly easily.

\begin{proof}[Proof of Theorem \ref{th:Gamma:Dim:local:1}]
	In accord with \eqref{Gamma:reduced},
	\[
		\Gamma^{(\nu)}_t(0\,,0) = \frac{1}{4\pi\nu t}
		\exp\left\{ -\frac{\rho(0)W_t^2}{4\nu t}
		\right\}\qquad[t>0].
	\]
	Therefore, if $0<K<(4\pi\nu)^{-1}$, then
	\[
		E(K) := \left\{ t>0:\ \Gamma^{(\nu)}_t(0\,,0) > \frac{K}{t}\right\}
		=\left\{ t>0:\ |W_t| < \alpha\sqrt{t}\right\},
	\]
	with
	\[
		\alpha := \sqrt{\frac{4\nu}{\rho(0)}\log\left( \frac{1}{4\pi\nu K}\right)}.
	\]
	Lemma \ref{lem:W:lvlset:sup} then implies that if $0<K<(4\pi\nu)^{-1}$, then
	$\Dim(E(K))=1$ a.s. If, on the other hand, $K\ge(4\pi\nu)^{-1}$, then
	$E(K)$ is empty and hence has zero macroscopic Minkowski dimension.
	This completes the proof of the first assertion of the theorem; the second
	assertion is a ready consequence of the first part of Lemma 
	\ref{lem:W:lvlset:sup}.
\end{proof}

As it turns out,  Theorem  \ref{th:Gamma:Dim:local:2} is a consequence of
the probabilistic representation of the solution to \eqref{pre:Kraichnan}
together with a large-scale fractal property of the Ornstein--Uhlenbeck process.

\begin{proof}[Proof of Theorem \ref{th:Gamma:Dim:local:2}]
	Let us fix some $\delta>0$ and consider the random set
	\[
		F := \left\{ t>\e:\ \Gamma^{(\nu)}_t(0\,,0) < 
		\frac{1}{t(\log t)^{\delta}}\right\}.
	\]
	Then,
	\[
		\log F = \left\{ t>1:\ |U_t| > \sqrt{\frac{4\nu}{\rho(0)}
		\log\left( \frac{t^{\delta}}{4\pi\nu}\right)}\right\},
	\]
	where
	\[
		U_t := \frac{W_{\exp(t)}}{\sqrt{\exp(t)}}\qquad[t\ge0].
	\]
	Choose and fix an arbitrary $\varepsilon\in(0\,,1)$,
	and fix $N>1$ such that
	\[
		\sqrt{\frac{4\nu\delta(1-\varepsilon)}{\rho(0)}
		\log t} <
		\sqrt{\frac{4\nu}{\rho(0)}
		\log\left( \frac{t^{\delta}}{4\pi\nu}\right)} < \sqrt{\frac{4\nu\delta(1+\varepsilon)}{\rho(0)}
		\log t}\qquad\text{for all $t>N$.}
	\]
	Elementary properties of the macroscopic dimension imply that
	\begin{equation}\label{Dim1}\begin{split}
		\Dim(\log F) &=\Dim\left\{ t>N:\ |U_t| > \sqrt{\frac{4\nu}{\rho(0)}
			\log\left( \frac{t^{\delta}}{4\pi\nu}\right)}\right\}\\
		&\le\Dim\left\{ t>N:\ |U_t|> \sqrt{\frac{4\nu\delta(1-\varepsilon)}{\rho(0)}
			\log t}\right\},
	\end{split}\end{equation}
	and similarly,
	\begin{equation}\label{Dim2}
		\Dim(\log F) \ge \Dim\left\{ t>N:\ |U_t|> \sqrt{\frac{4\nu\delta(1+\varepsilon)}{\rho(0)}
		\log t}\right\}.
	\end{equation}
	The stochastic process $\{U_t\}_{t\ge0}$ is a stationary Ornstein--Uhlenbeck process
	with covariance function $\Cov[U_s\,,U_t] = \exp\{-|t-s|\}$
	for $s,t\ge0$. Therefore, Theorem 6.1 of Weber \cite{Weber}  implies that
	\[
		\Dim\left\{ t>N:\ |U_t| > \sqrt{\alpha\log t}\right\} = \left( 1-\frac{\alpha}{2}\right)_+
		\quad\text{a.s.\ for all $\alpha>0$}.
	\]
	This, \eqref{Dim1}, and \eqref{Dim2} together imply
	Theorem \ref{th:Gamma:Dim:local:2}. 
\end{proof}

\section{A connection to fluid mechanics}\label{sec:Fluids}
For all $t\ge0$ and $x,y\in\R$ define
\[
	\bm{V}(t\,,x\,,y) = \begin{bmatrix}
		v_1(t\,,x\,,y)\\
		v_2(t\,,x\,,y)
	\end{bmatrix}
\]
to be a model for a 2-dimensional velocity field.
It is a generally-accepted fact that
the transport equation of a passive scalar in the field $V$ is governed by 
the following convection-diffusion equation:
\begin{equation}\label{cde}
	\frac{\partial\theta(t\,,x\,,y)}{\partial t} = \nu\Delta \theta(t\,,x\,,y) 
	- \frac{\partial (\theta v_1)(t\,,x\,,y)}{\partial x} 
	-\frac{\partial(\theta v_2)(t\,,x\,,y)}{\partial y},
\end{equation}
valid for all $t>0$ and $x,y\in\R$, subject to nice initial data $\theta(0):=\theta_0$.
The constant $\nu$ is strictly positive and referred to as \emph{thermal diffusivity}
for example when $\theta$ denotes temperature;  Kraichnan
\cite{Kraichnan1987}  refers to a closely-related quantity as
\emph{eddy diffusitivity}. Other, similar names, are used when 
$\theta$ denotes concentration, temperature, etc.

In fluid mechanics, $\nu$ is inversely proportional to the 
\emph{Reynolds number} of
the underlying fluid: Smaller values of $\nu$ imply more turbulence in the fluid.

We follow Majda \cite{Majda} and specialize to velocity fields that come from
so-called shear flows of the type,
\begin{equation}\label{Majda:V}
	\bm{V}(t\,,x\,,y) = \begin{bmatrix}
		0\\ \wt{V}(t\,,x)
	\end{bmatrix}.
\end{equation}
Among other things, such fluids are \emph{incompressible}
or divergence free; that is,
$\nabla\cdot\bm{V}=0$. In this way, the PDE
\eqref{cde} is simplified to the convection--diffusion equation,
\begin{equation}\label{cde:1}
	\partial_t \theta = \nu\Delta \theta 
	- \wt{V}\,\partial_y \theta.
\end{equation}
The partial differential equation \eqref{cde:1} has the same form 
as \eqref{Kraichnan}, but there is a small difference:
In general, the velocity field $\bm{V}$ is decomposed into 
its ``mean component'' $\mu\in\R$ and its ``fluctuating component'' $V=V(t\,,x)$
as follows:\footnote{%
	In order to simplify the technical aspects of this discussion
	we are assuming that $\mu\in\R$ is constant, though
	more general mean velocity fields can be considered as well.}
\begin{equation}\label{V:mu:V}
	\wt{V}(t\,,x) = \mu + V(t\,,x),
\end{equation}
and $\mu$ is not in general zero.
This is the so-called {\it Reynolds decomposition} of $\wt{V}$,
and the quoted terms are substitutes for the respective 
statements that $\mu$ is  deterministic and $V$ is random. 
When $V$ is a 
centered, generalized Gaussian random field with covariance
\eqref{cov:v}, the partial differential
equation \eqref{cde:1} is called the \emph{Kraichnan model}
for the 2-D flow described by $V$; see Kraichnan \cite{Kraichnan1968}. 
In this case, $\frac12\rho(0)$ is
the so-called \emph{turbulent diffusivity}.

Let $\theta$ denote the solution to the Kraichnan model \eqref{cde:1}
for the velocity model given by \eqref{Majda:V} and
\eqref{V:mu:V}. It is easy to make small adjustments to the arguments
of this paper in order to prove that, under Assumptions
\ref{assum:1} through \ref{assum:4}, the SPDE \eqref{cde:1}
has a solution that is unique in more or less the same sense as $\theta$ was
in the Introduction. Moreover, we have the following variation of
Theorem \ref{th:FK} that is valid in the present setting:
\[
	\theta(t\,,x\,,y) 
	=\E\left[\left. \theta_0\left(x+B_t\,,y+ \bar{B}_t +\mu t+
	\int_0^t V(s\,,x+B_{t-s})\,\d s\right)\ \right|\ \mathcal{V}\vee\mathcal{T}_0\right].
\]
That is, the introduction of the additional mean velocity field $\mu$ 
merely changes the mean function of the Brownian motion $\bar{B}$ from
its standard value zero to the mean velocity $\mu$. We leave the analysis
of this slightly more general model to the interested reader since the methods
of this paper cover this more general case as well.

\section{A multi-dimensional extension}
In this section we briefly study the following higher-dimensional analogue of 
the SPDE \eqref{Kraichnan}:
\begin{equation}\label{eq: Kraichnan n-d}
	\partial_t\theta(t\,,x\,,\bm{y})
	= \nu_1 \partial^2_x\theta(t\,,x\,,\bm{y})
	+ \nu_2 \sum_{j=1}^n \partial^2_{y_j} \theta(t\,,x\,,\bm{y})
	+ \sum_{j=1}^n \partial_{y_j} \theta(t\,,x\,,\bm{y}) V_j(t\,, x),
\end{equation}
where $\theta$ is a predictable random field, indexed by $\R_+\times \R \times \R^n$,
and the noise
\[
	\bm{V}(t\,,x) = 
	\begin{bmatrix}
		{V_1}(t\,,x) \\ V_2(t\,,x)\\ \vdots \\ V_n(t\,, x)
	\end{bmatrix}
\]
is centered Gaussian whose covariance function $\bm{\Sigma}$ is described by
\[
	\text{Cov}[V_i(t\,,x)\,, V_j(s\,, x')]=\delta_0(s-t)\rho_{ij}(x-x')
	\qquad\text{for all $s,t\ge0$ and $x,x'\in\R$},
\]
where $\bm{\rho}=(\rho_{i,j})_{1\le i,j\le n}:\R^{n}\to\R_+^{n\times n}$ is
the spatial correlation function of $\bm{V}$.\footnote{Interestingly enough,
the matrix $n^{-1}\bm{\rho}(0)$---sometimes known as
\emph{turbulent diffusitivity} has a role in the ensuing analysis
as $(2/n)$ times the closely-related matrix $\frac12\bm{\rho}(0)$,
which does have a physical meaning.}
Instead of writing out detailed
proofs, we merely point out how one solves \eqref{eq: Kraichnan n-d}
using analogies with the earlier case $n=2$ where the details were provided.

In complete analogy with the preceding sections, wherein $n$ was equal to $2$, 
we may take Fourier transforms with respect to the variable $\bm{y}$ 
in order to find that
\[
	U(t\,,x\,,\bm{\xi}) := \int_{\R^n} \e^{i\bm{\xi}\cdot\bm{y}}
	\theta(t\,,x\,,\bm{y})\,\d\bm{y}
\]
ought to solve the SPDE
\[
	\partial_t U(t\,, x \,, \bm{\xi}) = 
	\nu_1 \partial^2_x U(t\,, x \,, \bm{\xi}) - 
	\nu_2 \|\bm{\xi}\|^2 U(t\,, x \,, \bm{\xi}) + 
	i  U(t\,, x \,, \bm{\xi}) \sum_{j=1}^n \xi_j  V_j(t\,, x),
\]
where $\bm{\xi} := (\xi_1, \ldots, \xi_n) \in \R^n$
and $\|\bm{\xi}\|^2:=\xi_1^2+\cdots+\xi_n^2$. Once again, we follow the
procedure of the previous sections and define a random field $u$
via
\[
	U(t\,, x\,, \bm{\xi}) =\exp\left(-\nu_2 \|\bm{\xi}\|^2 t\right) u(t\,, x\,, \bm{\xi}),
\]
and arrive at the corresponding parabolic Anderson problems, 
\begin{equation}\label{eq:PAM n-d}
	\partial_t u(t\,, x \,, \bm{\xi}) = 
	\nu_1 \partial^2_x u(t\,, x \,, \bm{\xi}) +  
	i  u(t\,, x \,, \bm{\xi}) \bm{\xi}\cdot \bm{V}(t\,, x),
\end{equation}
solved pointwise for every $\bm{\xi}\in\R^n$.
Thus we see that the difference between \eqref{PAM} 
and \eqref{eq:PAM n-d} is that, instead of the multiplicative noise $i \xi V(t\,,x)$ in \eqref{PAM}, we have in \eqref{eq:PAM n-d} the noise $i \sum_{j=1}^n \xi_j  V_j(t\,, x)$. 
We now proceed  in almost exactly the same way as we did when $n$ was $2$,
and obtain the following $n$-dimensional extension of  Theorem \ref{th:u:exist}. 

Throughout, we write $F[\bm{z}]$ for the function $(t\,,x)\mapsto F(t\,,x\,,\bm{z})$
whenever applicable, notation being clear from context.

\begin{theorem}\label{th:u:exist n-d}
	Suppose $u_0:\Omega\times\R\times \R^n\to\C$ is a measurable random field
	that is independent of ${\bm{V}}$ and satisfies
	$\sup_{x\in\R}\E(|u_0(x\,,\bm{\xi})|^k)<\infty$ for every $k\geq 2$
	and  $\bm{\xi}\in\R^n$.
	Choose and fix some $\nu_1>0$. Then,  for every $\bm{\xi}\in\R^n$,
	\eqref{PAM} has a unique mild solution
	$u[\bm{\xi}]$ that satisfies the following for every $k\ge2$, $\varepsilon\in(0\,,1)$,
	$t>0$,  $x\in\R$ and $\bm{\xi} \in \R^n$:
	\[
		\sup_{x\in\R}\E\left( |u(t\,,x\,,\bm{\xi})|^k\right) \le 
		\varepsilon^{-k} \exp\left( \frac{kc_k\bm{\xi}' 
		\bm{\rho}(0)\bm{\xi}}{2(1-\varepsilon)^2}\,t\right)
		\sup_{x\in\R}\E\left( |u_0(x\,,\bm{\xi})|^k\right),
	\]
	where $c_k$ was defined earlier in Lemma \ref{lem:Young}.
\end{theorem}
If we assume that \eqref{u_0:1.5} holds, where now $\xi$ is replaced by
the vector $\bm{\xi}\in\R^n$,
then we can proceed in \emph{exactly}
the same way as we did in the proof of Lemma \ref{u-u:xi},
in order to show that, in the $n$-dimensional case, $\bm{\xi}\mapsto
u(t\,, x\,, \bm{\xi})$ has a  continuous, and thus Borel-measurable,
version. We can also obtain the probabilistic representation of $u$ as follows.

\begin{proposition}\label{pr:PAM:FK n-d}
	Assume that \eqref{u_0:1.5} holds. Then
	for every $t\ge0$, $x\in\R$, and $\bm{\xi}\in\R^n$,
	\[
		u(t\,,x\,,\bm{\xi}) = \e^{\frac12 t\,
		\bm{\xi}' \bm{\rho}(0)\bm{\xi}}\,
		\E\left[\left. \wh{\theta}_0(x+B_t\,,\bm{\xi})
		\exp\left( i\sum_{j=1}^n\xi_j\int_0^t V_j(s\,, x+B_{t-s})\,\d s\right)
		\ \right|\ \mathcal{V}\vee\mathcal{T}_0\right]\,,
	\]
	where $B$ is a Brownian motion independent of $\mathcal{V}\vee\mathcal{T}_0$ 
	with $\Var (B_1)=2 \nu_1$. 
\end{proposition}

 In order to obtain a probabilistic representation of $\theta$---and  
 also to prove the existence and uniqueness of the solution to \eqref{eq: Kraichnan n-d}---%
 we plan to compute the inverse Fourier transform of 
 $\R^n\ni\bm{\xi}\mapsto\exp\{-\nu_2 |\bm{\xi}|^2 t\} u(t\,, x\,, \bm{\xi})$.
 In analogy with the preceding sections, our methods show that this inverse Fourier transform
 exists provided only that
 $\nu_2 \bm{I} - \frac{1}{2}\bm{\rho}(0)$
is \emph{strictly} positive definite. Here,
$\bm{I}$ denotes the $n\times n$ identity matrix. 
These assertions can be summarized as follows.

\begin{theorem}\label{th:FK n-d}
	Assume that \eqref{u_0:1.5} holds and that
	$\nu_2 \bm{I} - \frac12\bm{\rho}(0)$ is strictly positive definite. 
	Let $B$ denote a standard linear Brownian motion,
	and $\bar{\bm{B}}$ a standard Brownian motion on $\R^n$,
	and assume that: 
	\begin{compactenum}
	\item $B$, $\bar{\bm{B}}$, and $\mathcal{V}\vee\mathcal{T}_0$
		are totally independent;
	\item $B$ has speed $\Var (B_1)=2 \nu_1$; and
	\item The covariance matrix for $\bar{\bm{B}}_1$ is $ 2\nu_2 \bm{I} - \bm{\rho}(0)$. 
	\end{compactenum}
	Then,  for all $t\ge0$, $x\in\R$, and $\bm{y}\in\R^n$,
	\[
		\theta(t\,,x\,,\bm{y}) 
		=\E\left[\left. \theta_0\left(x+B_t\,,\bm{y}+ \bar{\bm{B}}_t -
		\int_0^t \bm{V}(s\,,x+B_{t-s})\,\d s\right)\ \right|\ 
		\bm{\mathcal{V}}\vee\mathcal{T}_0\right],
	\]
	almost surely, where $\bm{\mathcal{V}}$ denotes the $\sigma$-algebra generated by
	$\bm{V}$, $\mathcal{T}_0$ is as before, and the random variable $\int_0^t\bm{V}(s\,,x+B_{t-s})\,\d s$
	is defined as in Lemma \ref{lem:V} in every coordinate.
\end{theorem}

Finally, we may consider instead the Stratonovich solution to equation 
\eqref{eq: Kraichnan n-d} by first replacing $\bm{V}$ by a smooth random noise
and taking limits afterward. The required extension to the present
$n$-dimensional setting does not require new ideas, and leads to the following:
\[
\begin{split}
	&\theta^{(\nu_1,\nu_2)}(t\,,x\,,\bm{y}) \\
	&\hskip.7in
		= \E\left[\left. \ \theta_0\left( x+\sqrt{2\nu_1}\, W_t\,,
		\bm{y} + \sqrt{2\nu_2}\, \bm{W}'_t 
		-\int_0^t \bm{V}\left(s\,,x+ \sqrt{2\nu_1}\, W_{t-s} \right)\d s\right)\ \right|\ 
		\bm{\mathcal{V}}\vee\mathcal{T}_0\right],
\end{split}
\]
where $W$ and $\bm{W}$ are  respectively linear and $n$-dimensional 
Brownian motions, both independent of each other as well as the $\sigma$-algebra
$\bm{\mathcal{V}}\vee\mathcal{T}_0$. In particular, we see that \eqref{eq: Kraichnan n-d}
has a Stratonovich solution for every $\nu_1,\nu_2>0$, with $\theta_0$ satisfying \ref{assum:2}--\ref{assum:4} and bounded. 
In particular, if \eqref{u_0:1.5} holds, then for every $\nu>0$, the Stratonovich solution of 
\[
	\partial_t \theta(t\,,x\,,\bm{y})
	= \nu \Delta \theta(t\,,x\,,\bm{y})
	+ \nabla_{\bm{y}} \theta(t\,,x\,,\bm{y}) \cdot \bm{V}(t\,, x),
\]
subject to initial data $\theta_0$ that follows Assumptions \ref{assum:1}--\ref{assum:4} and bounded 
is the following: For all $t\ge0$, $x\in\R$, and $\bm{y}\in\R^n$,
\[
	\theta^{(\nu)}(t\,,x\,,\bm{y}) = \E\left[\left. \ \theta_0\left( x+ \sqrt{2\nu}\, W_t\,,
	\bm{y} + \sqrt{2\nu}\, \bm{W}'_t 
	-\int_0^t \bm{V}\left(s\,,x+ \sqrt{2\nu}
	\, W_{t-s} \right)\d s\right)\ \right|\ 
	\bm{\mathcal{V}}\vee\mathcal{T}_0\right],
\]
almost surely.
We leave the other extensions (inviscid equations, measure-valued initial data, etc.) to
the interested reader.

\appendix

\section{Appendix: Stochastic integrals}\label{sec:SI}

In this appendix we briefly review aspects of the Walsh theory of stochastic integration, as it
pertains to the present setting. We use this opportunity to set forth some notation, and present
a stochastic Fubini theorem that plays an important role in the paper.

\subsection{The Wiener integral}\label{Append: Wiener int}
Let $C_c^{\infty}((0\,,\infty)\times\R)$ denote the usual vector space of all
infinitely-differentiable, compactly-supported,
real-valued functions on $(0\,,\infty)\times\R$, and define
$\mathcal{H}$ to be the completion of $C_c((0\,,\infty)\times\R)$
 in the norm $\|\cdots\|_{\mathcal{H}}$, where 
\begin{equation}\label{H}
	\|\varphi\|_{\mathcal{H}}^2 := \int_0^{\infty}\d t \int_{-\infty}^{\infty} 
	\d x \int_{-\infty}^{\infty} \d y\ \varphi(t\,, x) \varphi(t\,, y)\rho(x-y)\,. 
\end{equation}

Throughout, we let $(\Omega\,,\F,\P)$ be a probability space that is rich enough to support
a centered Gaussian process
$V:=\{V(\varphi)\}_{\varphi\in C^\infty_c((0,\infty)\times\R)}$  with
formal covariance form given in \eqref{cov:v}. More precisely put, $V$ is a centered
Gaussian process whose covariance function is described by
\[
	\Cov[V(\varphi) \,, V(\psi)] = \int_0^\infty\d t\int_{-\infty}^\infty\d x
	\int_{-\infty}^\infty\d x'\ \varphi(t\,,x)\psi(t\,,x')\rho(x-x'),
\]
for every $\varphi,\psi\in C^\infty_c((0\,,\infty)\times\R)$. 
The stochastic process $V$ is sometimes called
an \emph{isonormal}, or \emph{iso-Gaussian process}. 
According to the classical Wiener theory
(see \S\ref{Append: Wiener int}), we may identify $V$ with a 
linear isometry from $ C^\infty_c((0\,,\infty)\times\R)$
into the space of all random variables in $L^2(\P)$. Thus, we may also think of
$V$ as a \emph{Wiener integral}. For this reason, we also  write
\[
	V(\varphi) = 
	\int_{\R_+\times\R}\varphi(t\,,x)V(t\,,x)\,\d t\,\d x
	\qquad\text{for every $\varphi\in\mathcal{H}$}.
\]
As is usual, we may write
\[
	\int_{A\times B}\varphi(t\,,x)V(t\,,x)\,\d t\,\d x
	:= V(\varphi\1_{A\times B}),
\]
whenever $\varphi\in\mathcal{H}$,
and $A\subset\R_+$ and $B\subset\R$ are Borel sets.
Thus, it follows that we can extend the domain of definition of
$\wt{V}$ continuously to the full parameter space $\mathcal{H}$,
denote the extended process still by $\wt{V}$, and observe that
$\wt{V}$ has the same distribution as $V$. Bearing this 
convention in mind, it
follows that $V$ is a linear isometry from 
the full Hilbert space $\mathcal{H}$ into $L^2(\P)$.

Consider the special case that $\rho$ is a constant; that is,
$\rho(x) = \rho(0)$ for all $x\in\R$. Let $W:=\{W_t\}_{t\ge0}$ denote
a standard Brownian motion and define a stochastic process
$\{\wt{V}(\varphi)\}_{\varphi\in C^\infty_c(\R_+\times\R)}$ by setting
\begin{equation}\label{wt{V}}
	\wt{V}(\varphi) := \sqrt{\rho(0)}\,
	\int_{-\infty}^\infty \left( \int_0^\infty \varphi(t\,,x)\,
	\d W_t\right)\d x\qquad\text{for every $\varphi\in C^\infty_c(\R_+\times\R)$},
\end{equation}
where $\int_0^\infty\varphi(t\,,x)\,\d W_t$ is a standard Wiener integral---with respect
to Brownian motion $W$---for every
$x\in\R$.
It is easy to see that $\{\wt{V}(\varphi)\}_{\varphi\in C^\infty_c(\R_+\times\R)}$
is a centered Gaussian random field with covariance function
\begin{align*}
	\Cov\left[\wt{V}(\varphi)\,,\wt{V}(\psi)\right]
		&= \rho(0)\int_0^\infty\left(
		\int_{-\infty}^\infty\psi(t\,,x')\,\d x'\right)\left(\int_{-\infty}^\infty
		\varphi(t\,,x)\,\d x\right)\d t\\
	&= \Cov[V(\varphi)\,,V(\psi)]\qquad\text{for all 
		$\varphi,\psi\in C^\infty_c(\R_+\times\R)$}.
\end{align*}
Thus, it follows that there exists a unique, continuous extension of
$\wt{V}$ to a stochastic process $\{\wt{V}(\varphi)\}_{\varphi\in\mathcal{H}}$
whose law is the same as the law of $V$. In other words, whenever $\rho$ is a constant,
we may---and will---assume that $V$ has the form given by \eqref{wt{V}}.
In this sense, we see that if $\rho$ is a constant, then we can write $V$  as
\begin{equation}\label{V:W}
	V(t\,,x)\,\d t\,\d x = \sqrt{\rho(0)}\,\d W_t\,\d x,
\end{equation}
using informal infinitesimal notation.

\subsection{The Walsh integral}
The \emph{Walsh integral} is an extension of the Wiener integral
\[
	V(\Phi) := \int_{\R_+\times\R} V(t\,,x)\Phi(t\,,x)\,\d t\,\d x
\]
to the
case that $\Phi$ is a predictable random field that satisfies
\begin{equation}\label{PhiPhi}
	\int_0^\infty \d t\int_{-\infty}^\infty\d x\int_{-\infty}^\infty\d x'\
	\E\left(\left|\Phi(t\,,x)\Phi(t\,,x')\right|\right)\rho(x-x')<\infty; 
\end{equation}
see Walsh \cite{Walsh}
and especially Dalang \cite{Dalang} for details. 
Thanks to \eqref{rho:max} and Tonelli's theorem, \eqref{PhiPhi} is implied by the
following integrability condition:
\[
	\int_0^\infty \d t\left(\int_{-\infty}^\infty
	\left\|\Phi(t\,,x)\right\|_2\,\d x\right)^2 =
	\int_0^\infty \d t\left(\int_{-\infty}^\infty
	\sqrt{\E\left(|\Phi(t\,,x)|^2\right)}\,\d x\right)^2<\infty; 
\]
this fact is used several times in the paper.

As a noteworthy consequence of the construction of the Walsh
integral, we can see that for all such
random functions $\Phi$,
\[
	\Var\left[ V(\Phi)\right] = \E\left[\int_0^\infty \d t
	\int_{-\infty}^\infty\d x\int_{-\infty}^\infty\d x'\
	\Phi(t\,,x)\Phi(t\,,x')\rho(x-x')\right].
\]
This is the so-called \emph{Walsh isometry} for Walsh stochastic integrals. 

It is easy to see that if $\rho$ is a constant, then we may use the representation \eqref{wt{V}}
in order to find that
\[
	V(\Phi) = \sqrt{\rho(0)}\,
	\int_{-\infty}^\infty\left( \int_0^\infty\Phi(t\,,x)\,\d W_t\right)\d x,
\]
as long as, additionally, the following hold:
(1) $t\mapsto\Phi(t\,,x)$ is a predictable process for every $x\in\R$;
and (2) the It\^o integral map $x\mapsto\int_0^\infty\Phi(t\,,x)\,\d W_t$
is Lebesgue measurable. Indeed, by a standard approximation procedure,
it suffices to verify this assertion for processes of the form
$\Phi(t\,,x) = X\1_{(a,b)}(t)f(x)$ where $X\in L^2(\P)$ is measurable with
respect to the $\sigma$-algebra generated by all random variables of the 
form $\int_{(0,a)\times\R}\varphi(t\,,x)V(t\,,x)\,\d x$, as $\varphi$ roams over
$\mathcal{H}$, and  $f\in C^\infty_c(\R)$ is a nonrandom,
smooth, and compactly-supported function. In that case, 
$V(\Phi) = X V(\1_{(a,b)}\otimes f)$ and the assertion follows by
direct inspection, thanks to \eqref{wt{V}} and the defining properties of the Walsh stochastic integral.

\subsection{A stochastic Fubini theorem}
The stochastic Fubini's theorem is used a number of times in this paper. 
We cite, without proof, a suitable version of it here. It might help to recall
from \eqref{H} the space $\mathcal{H}$, and also the fact that
$\Phi[y]$ refers to the function $(t\,,x)\mapsto \Phi(t\,,x\,,y)$
for every $y\in\R$.

\begin{theorem}[\protect{\cite[Theorem 4.33, p.\ 110]{DZ}}]\label{th:stochastic:Fubini}
	Let $\{\Phi(t\,, x\,, y);\, t \geq 0\,, x , y \in \R\}$ 
	be a three-parameter, predictable random field that satisfies
	{$\int_{-\infty}^\infty \|\Phi[y]\|_{L^2(\Omega\times [0,T]; \mathcal{H})} 
	\d y< \infty$ }for every positive real number $T$.
	Then, 
	\[
		\int_{-\infty}^\infty \left(\int_{(0,T)\times\R} \Phi(s\,, x\,, y)V(s\,, x )
		\,\d x\, \d s\right) \d y = 
		\int_{(0,T)\times\R} \left(  \int_{-\infty}^\infty 
		\Phi(s\,, x\,, y)\,\d y \right) V(s\,, x)\, \d x\, \d s,
	\]
	almost surely $[\P]$.
\end{theorem}

\subsection{Elements of Malliavin calculus}
In this subsection we will outline the setup of Malliavin calculus. For a detailed treatment of this material, see Nualart \cite{Nualart}.  Let $F$ be a smooth and cylindrical random variable of the form 
\[
	F= f(V(\phi_1), \dots, V(\phi_n))\,,
\]
with $\phi_i \in \mathcal{H}$, where $\mathcal{H}$ 
is defined in Subsection \ref{Append: Wiener int}, 
$V(\phi_i):= \int_0^{\infty}\int_{-\infty}^\infty\phi_i(s\,, x)V(s\,,x)\d s \d x $ 
and $f\in C_p^{\infty}(\R^n)$ (namely $f$ and its partial derivatives 
have polynomial growth), then the Malliavin derivative $DF$ is 
the $\mathcal{H}$-valued random variable defined by 
\[
	DF= \sum_{j=1}^n \frac{\partial f}{\partial x_j}(V(\phi_1), \dots, V(\phi_n)) \phi_j\,.
\]
The operator $D$ is closable from $L^2(\Omega)$ into $L^2(\Omega; \mathcal{H})$ and we define the Sobolev space $\mathbb{D}^{1,2}$ as the closure of the space of smooth and cylindrical random variables under the norm 
\[
	\|DF\|_{1,2} = \sqrt{\E [F^2]+ \E \|DF\|^2_{\mathcal{H}}}\,.
\]
We denote by $\delta$ the adjoint of the derivative operator given by the duality formula 
\[
	\E [\delta (u) F] = \E [\langle DF, u\rangle_{\mathcal{H}}]\,,
\]
for any $F \in \mathbb{D}^{1,2}$ and any element $u\in L^2(\Omega; \mathcal{H})$ in the domain of $\delta$. 

Let us remark that in our context, that is, $V(t,x)$ is a Gaussian noise 
which is white in time and has certain covariance $\rho$ in the space, 
if $\R_+\times\R\ni (t\,,x)\mapsto u(t\,,x)$ is an adapted stochastic process such that 
$\E \int_0^{\infty} \int_{-\infty}^\infty\int_{-\infty}^\infty u(t\,,y)u(t\,,z)\rho(y-z)\,\d y\, \d z\, \d t < \infty$, 
then $u$ belongs to the domain of $\delta$ and $\delta(u)$ coincides with the Walsh integral: 
\begin{equation*}
	\delta(u) = \int_{\R_+\times\R} u(t\,,x)V(t\,,x)\,\d t \,\d x.
\end{equation*}

\spacing{0.9}

\bigskip
\small
\noindent\textbf{Jingyu Huang} [\texttt{jhuang@math.utah.edu}]\\
\noindent 
	Department of Mathematics, University of Utah, Salt Lake City ,UT 84112-0090\\
	
\noindent\textbf{Davar Khoshnevisan} [\texttt{davar@math.utah.edu}]\\
\noindent
	Department of Mathematics, University of Utah, Salt Lake City ,UT 84112-0090\\[.2cm]


\begin{thebibliography}{99}
%
\bibitem{Adler} Adler, Robert J. (1981). {\it The Geometry of 
	Random Fields}. Wiley, New York.
%
\bibitem{BT1989}Barlow, M. T. and S. J. Taylor (1989).
	Fractional dimension of sets in discrete spaces. 
	{\it J. Phys.\ A} {\bf 22}:2621--2626.
%
\bibitem{BT1992}Barlow, M. T. and S. J. Taylor (1992).
	Defining fractal subsets of $\Z^d$. 
	{\it Proc.\ London Math.\ Soc.}\ {\bf 64}:125--152.
%
\bibitem{BGK1996}Bernard, D. Gaw\c{e}dzki,
	and A. Kupiainnen (1996). Anomalous scaling in the $N$-point functions of a passive scalar. 
	{\it Phys.\ Rev.\ E}  (3) {\bf 54}{\it (3)}:2564--2572. 
%
\bibitem{BGK1998}Bernard, D. Gaw\c{e}dzki,
	and A. Kupiainnen (1998). Slow modes in passive advection. {\it J. Statist.\
	Phys.}\ {\bf 54}{\it (3)}:2564--2572.
%
\bibitem{BC}Bertini, Lorenzo, and Nicoletta Cancrini (1995).
	The stochastic heat equation: Feynman-Kac formula and intermittence. 
	{\it J. Statist.\ Phys.}\ {\bf 78}{\it (5-6):}1377--1401.
%
\bibitem{Borell}Borell, Christer (1975).
	The Brunn-Minkowski inequality in Gauss space. 
	{\it Invent.\ Math.}\ {\bf 30}{\it (2)}:207--216. 
%
\bibitem{BM1997}Bronski, Jared C., Richard M. McLaughlin (1997). 
	Scalar intermittency and the ground state of
	periodic Schr\"odinger equations. {\it Phys.\ Fluids} {\bf 9}:181--190.
%
\bibitem{BM2000}Bronski, Jared C., Richard M. McLaughlin (2000). 
	Rigorous estimates of the tails of the probability 	distribution 
	function for the random linear shear model. {\it J. Stat. Phys.} {\bf 98}{\it (3/4)}:897--915. 
%
\bibitem{BM2000a}Bronski, Jared C., Richard M. McLaughlin (2000). 
	The problem of moments and the Majda model for scalar
	intermittency  {\it Phys.\ Lett.\ A.} {\bf 265}:257--263. 
%
\bibitem{Burkholder1966}Burkholder, D. L. (1966).
	Martingale transforms.
	{\it Ann.\ Math.\ Statist.}\ {\bf 37}:1494--1504.
%
\bibitem{BDG1972}Burkholder, D. L., B. J. Davis, and R. F.  Gundy (1972).
	Integral inequalities for convex functions of operators on martingales.
	In: {\it Proceedings of the Sixth Berkeley Symposium 
	on Mathematical Statistics and Probability II}, 223--240, 
	University of California Press, Berkeley, California.
%
\bibitem{BG1970}Burkholder, D. L. and Gundy, R. F. (1970).
	Extrapolation and interpolation of quasi- linear operators on martingales.
	{\it Acta Math.}\ {\bf 124}:249--304.
%
\bibitem{CM}Carmona, Ren\'e A. and S. A. Molchanov (1994).
	Parabolic Anderson Problem and Intermittency,
	\emph{Memoires of the Amer.\ Math.\ Soc.} {\bf 108}
	American Mathematical Society, Rhode Island.
%
\bibitem{CKree}Carlen, Eric, Paul Kr\'ee (1991). $L^p$ estimates on 
	iterated stochastic integrals. {\it Ann. Probab.} {\bf 19}{\it (1)}:354-368. 
%
\bibitem{CelaniVincenzi}Celani, A. and D. Vincenzi (2002). Intermittency in passive
	scalar decay. {\it Phys.\ D} {\bf 172}:103--110.
%

\bibitem{Chow} Chow, Pao-Liu (2015). Stochastic partial differential equations. 
Second edition. Advances in Applied Mathematics. {\it CRC Press, Boca Raton}, FL, 2015. xvi+317 pp.

\bibitem{CK}Conus, Daniel, Davar Khoshnevisan (2012). On the existence and position 
	of the farthest peaks of a family of stochastic heat and wave equations. {\it Probab. 
	Theory Related Rields} {\bf 152}{\it (3--4)}:681-701. 
%
\bibitem{CJKS}Conus, Daniel, Mathew Joseph,  Davar Khoshnevisan, Shang-Yuan Shiu (2013).  
	On the chaotic character of the stochastic heat equation, II. 
	{\it Probab. Theory Related Fields} {\bf 156}{\it (3-4)}:483--533. 
%
\bibitem{CranstonZhao}Cranston, M. and Z. Zhao (1987).
	Conditional transformation of drift formula and potential theory for
	$\frac12\Delta+b(\cdot)\cdot\nabla$. {\it Comm.\ Math.\ Phys.}\ 
	{\bf 112}:613--625.
%
\bibitem{Dalang}Dalang, Robert C.,
	Extending the martingale measure stochastic integral with
	applications to spatially homogeneous s.p.d.e.'s,
	{\it Electron. J. Probab.}\ {\bf 4}, Paper no.\ 6 (1999) 29 pp.\
	(electronic). Available electronically at
	\url{http://www.math.washington.edu/~ejpecp}.
%
\bibitem{minicourse} Dalang, R. C., Davar Khoshnevisan, Carl Mueller,
	David Nualart, and Yimin Xiao (2009).
	{\it A Minicourse on Stochastic Partial Differential Equations},
	Springer, Berlin.
\bibitem{DZ}Da Prato, Giuseppe, Jerzy Zabczyk (2014)
	\emph{Stochastic Equations in Infinite Dimensions}, 
	(Encyclopedia of Mathematics and its Applications).  Cambridge University Press. 

\bibitem{Davis}Davis, Burgess (1976). 
	On the $L^p$ norms of stochastic integrals and other martingales. 
	{\it Duke Math. J.} {\bf 43}{\it (4)}:697--704. 

\bibitem{DoeringMytnik}D\"oring, Leif and Leonid Mytnik (2013).
	Longtime behavior for mutually catalytic branching with negative correlations. 
	In: {\it Advances in Superprocesses and Nonlinear PDEs}, pp.\ 93--111, 
	Springer Proc.\ Math.\ Stat.\ {\bf 38} Springer, New York, 2013. 
%
\bibitem{EyinkXin2000}Eyink, Gregory L. and Jack Xin (2000).
	Self-Similar Decay in the Kraichnan Model of a Passive Scalar.
	{\it J. Statist.\ Phys.}\ {\bf 100}{\it (3/4)}:679--741.
%
\bibitem{Fannjiang}Fannjiang, Albert C. (2004).
	Convergence of Passive Scalar Fields in Ornstein--Uhlenbeck 
	Flows to Kraichnan's Model.
	{\it J. Statist.\ Phys.}\ {\bf 114}{\it (1/2)}:115--135.
%
\bibitem{FK}Foondun, Mohammud, Davar Khoshnevisan (2013). 
	On the stochastic heat equation with spatially-colored random forcing. 
	{\it Trans. Amer. Math. Soc.} {\bf 365}{\it (1)}:409--458. 
%
\bibitem{Freidlin}Freidlin, Mark(1985). 
	\emph{Functional Integration and Partial Differential Equations.} 
	Annals of Mathematics Studies,  Princeton University Press. 
%
\bibitem{FrizHairer}Friz, Peter K., Martin Hairer (2014). 
	\emph{A Course on Rough Paths.}
	Universitext. Springer Cham, Heidelberg. xiv+251 pp.
%
\bibitem{GrossmannLohse}Grossmann, Siegfried and Detlef Lohse (1994).
	Fractal-dimension crossovers in turbulent passive scalar signals.
	{\it Europhys.\ Lett.}\ {\bf 27}{\it (5)}:347--352.
%
\bibitem{HairerPardoux} Hairer, Martin, \'Etienne Pardoux(2015). 
	A Wong-Zakai theorem for stochastic PDEs. {\it J. Math. Soc. Japan} 
	{\bf 67}{\it (4)}:1551--1604. 
%
\bibitem{HolzerSiggia}Holzer, Mark and Eric D. Siggia (1994).
	Turbulent mixing of a passive scalar.
	{\it Phys.\ Fluids} {\bf 6z}{\it (5)}:1820--1837.
%
\bibitem{HuNualart}Hu, Yaozhong, David Nualart(2009). 
	Stochastic heat equation driven by fractional noise
	and local time. {\it Probab. \ Theory \ Relat. \ Fields}. 
	{\bf 143}:285--328. 
%
\bibitem{INY} Ikeda, Nobuyuki,  Shintaro Nakao, and Yuiti Yamato
	(1977/1978). A class of approximations of Brownian motion. 
	{\it Publ.\ Res.\ Inst.\ Math.\ Sci.}\ {\bf 13}{\it 1}:285--300. 
%
\bibitem{Escape} Khoshnevisan, Davar (1997).
	Escape rates for L\'evy processes,
	{\it Stud.\ Sci.\ Math.\ Hung.} {\bf 33} 177--183.
%
\bibitem{Khoshnevisan} Khoshnevisan, Davar (2014).
	{\emph Analysis of Stochastic Partial Differential Equations.} 
	CBMS Regional Conference Series in Mathematics {\bf 119}. 
	American Mathematical Society, Providence, RI.
	
\bibitem{Khoshnevisan2} 	Khoshnevisan, Davar (1994). 
A discrete fractal in $\Z^1$
{\it Proc. Amer. Math. Soc.}, {\bf 120}{\it 2}: 577--584.
	
%
%
\bibitem{Kraichnan1968}Kraichnan, Robert H. (1968).
	Small-scale structure of a scalar field convected by turbulence.
	{\it Phys.\ Fluids}\ {\bf 11}{\it (5)}:945--953.
%
\bibitem{Kraichnan1987}Kraichnan, Robert H. (1987).
	Eddy Viscosity and Diffusivity: Exact Formulas and Approximations.
	{\it Complex Systems} {\bf 1}:805--820.
%
\bibitem{Krylov}Krylov, N. V. (1999). 
          An analytic approach to SPDEs. {\it Stochastic partial differential equations: six perspectives}, 185--242, 
Math. Surveys Monogr., 64, Amer. Math. Soc., Providence, RI.


%
\bibitem{Kunita} Kunita, H. (1990).
	{\it Stochastic Flows and Stochastic Differential Equations}.
	Cambridge University Press, Cambridge. 
%
\bibitem{LeJanRaimond} Le Jan, Yves and Olivier Raimond (2002).
	Integration of Brownian vector fields. {\it Ann.\ Probab.}\ 
	{\bf 30}{\it (2)}:826--873.
%
\bibitem{Majda} Majda, Andrew (1993). The random uniform shear layer: 
	An explicit example of turbulent diffusion with broad tail probability 
	distributions. {\it Phys. Fluids A} {\bf 5}{\it (8)}:1963--1970. 
%
\bibitem{Majdaa} Majda, Andrew (1993). Explicit Inertial Range 
	Renormalization Theory in a Model for Turbulent Diffusion.
	{\it J. Statist.\ Phys.}\ {\bf 73}{\it (3/4)}:515--542.
%
\bibitem{Mandelbrot}Mandelbrot, Benoit B. (1982).
	\emph{Fractal Geometry of Nature},
	W. H. Freeman and Co., San Francisco, Calif.
%
\bibitem{McShane} McShane, E. J.  (1970).
	Stochastic differential equations and models of random processes.
	{\it Proc.\ 6-th Berkeley Symp.\ on Math.\ Statist.\ and Prob.}\ 
	{\bf 3}:263--294.
%
\bibitem{Nualart}Nualart, David (2006). 
	\emph{The Malliavin Calculus and Related Topics}, 
	Springer, New York.
%
\bibitem{Osada}Osada, Hirofumi (1987).
	Diffusion processes with generators of generalized divergence form.
	{\it J. Math.\ Kyoto Univ.}\ {\bf 27}{\it (4)}:597--619.
%
\bibitem{SudakovTsirelson}Sudakov, V. N. and B. S. Tsirel'son (1978).
	Extremal properties of half-spaces for spherically invariant measures. 
	(in Russian) {\it J. Soviet Math.}\ {\bf 9}:9--18.
	(Translated from {\it Zap.\ Nauch.\ Sem.\ L.O.M.I.} {\bf 41} (1974):14--24.)
%
\bibitem{VandenEijnden} Vanden Eijnden, Eric (2001).  
	Non-Gaussian invariant measures for the Majda model 
	of decaying turbulent transport. 
	{\it Comm.\ Pure Appl.\ Math.}\ {\bf 54}{\it (9)}:1146--1167. 
%
\bibitem{Walsh} Walsh, John B. (1986).
	\emph{An Introduction to Stochastic Partial Differential Equations},
	in: \'Ecole d'\'et\'e de probabilit\'es de Saint-Flour, XIV---1984,
	265--439, Lecture Notes in Math., vol.\ 1180, Springer, Berlin.
%
\bibitem{Warhaft}Warhaft, Z. (2000).
	Passive scalars in turbulent flows.
	{\it Annu.\ Rev.\ Fluid Mech.}\ {\bf 32}:203--240.
%
\bibitem{Weber} Weber, M. (2004) Some examples of application of the metric entropy method. 
	{\it Acta Math. Hunger.} {\bf 105}{\it (1-2)}:39--83. 
%
\bibitem{WongZakai} Wong, Eugene, Moshe Zakai (1965). 
	On the relation between ordinary and stochastic differential equations. 
	{\it Int. J. Engng Sci. }\ {\bf 3}:213-229. 
%
\bibitem{Zhang}Zhang, Qi S. (2003).
	A strong regularity result for parabolic equations.
	{\it Comm.\ Math.\ Phys.}\ {\bf 244}{\it (2)}:245--260.
%
\end{thebibliography}
\end{document}